\documentclass[11pt,reqno]{amsart}
\usepackage{amsmath, amssymb, amscd, mathrsfs, slashed, url}

\usepackage{color}
\usepackage{extsizes}
\usepackage{xcolor}
\usepackage[all,cmtip]{xy}
\usepackage{multicol}
\usepackage{indentfirst}
\usepackage{latexsym}
\usepackage{bm}
\usepackage{graphicx}
\usepackage{subfigure,esint}
\usepackage{float}
\usepackage{verbatim}
\usepackage{comment}
\usepackage[top=1in, bottom=1in, left=1.25in, right=1.25in]{geometry}
\usepackage{epsfig,dsfont,amsthm,amsfonts,amsbsy}
\usepackage[colorlinks]{hyperref}

\usepackage[toc,page]{appendix}
\usepackage{tikz-cd}
\usepackage{amsthm,thmtools,xcolor}

\usepackage{adjustbox}
\usepackage{enumitem}
\setlist{nosep}
\setlist[itemize]{leftmargin=*}
\setlist[enumerate]{leftmargin=*,align=left,nolistsep}

\newtheorem{theorem}{Theorem}[section]

\newtheorem{conjecture}[theorem]{Conjecture}
\newtheorem{corollary}[theorem]{Corollary}

\newtheorem{definition}[theorem]{Definition}
\newtheorem{question}[theorem]{Question}

\newtheorem{lemma}[theorem]{Lemma}
\newtheorem{notation}[theorem]{Notation}

\newtheorem{proposition}[theorem]{Proposition}

\theoremstyle{definition}
\newtheorem{example}[theorem]{Example}
\newtheorem{remark}[theorem]{Remark}

\newcommand{\lan}{\langle }
\newcommand{\ran}{\rangle}
\newcommand{\da}{\dagger}

\newcommand{\MH}{\mathcal{H}}

\newcommand{\MS}{\mathcal{S}}
\newcommand{\MM}{\mathcal{M}}
\newcommand{\MO}{\mathcal{O}}

\newcommand{\CS}{\mathbb{C}^{\st}}

\newcommand{\End}{\mathrm{End}}

\newcommand{\tot}{\mathrm{tot}}
\newcommand{\MC}{\mathcal{C}}
\newcommand{\Tr}{\mathrm{Tr}}

\newcommand{\st}{\star}
\newcommand{\we}{\wedge}
\newcommand{\pa}{\partial}

\newcommand{\ti}{\times}
\newcommand{\SLR}{\mathrm{SL}_2(\mathbb{R})}

\newcommand{\vp}{\varphi}
\newcommand{\MA}{\mathcal{A}}

\newcommand{\al}{\alpha}

\newcommand{\be}{\beta}

\newcommand{\CC}{\mathbb C}

\newcommand{\lam}{\lambda}

\newcommand{\MP}{\mathcal{P}}

\newcommand{\diag}{\mathrm{diag}}

\newcommand{\MG}{\mathcal{G}}
\newcommand{\MGC}{\mathcal{G}^{\mathbb{C}}}

\newcommand{\mfR}{\mathfrak{R}}
\newcommand{\GLC}{\mathrm{GL}_2(\mbC)}
\newcommand{\SU}{\mathrm{SU}}

\newcommand{\MF}{\mathcal{F}}

\newcommand{\Lam}{\Lambda}

\newcommand{\MMH}{\MM_{\mathrm{Higgs}}}
\newcommand{\MMHYM}{\MM_{\mathrm{HYM}}}

\newcommand{\MB}{\mathcal{B}}
\newcommand{\mbP}{\mathbb{P}}
\newcommand{\mbB}{\mathbb{B}}
\newcommand{\mbQ}{\mathbb{Q}}
\newcommand{\mbH}{\mathbb{H}}

\newcommand{\ZT}{\mathbb{Z}_2}
\newcommand{\Dol}{\mathrm{Dol}}

\newcommand{\NS}{\mathrm{NS}}

\newcommand{\MSE}{\mathscr{E}}
\newcommand{\MSF}{\mathscr{F}}

\newcommand{\MSI}{\mathscr{I}}

\newcommand{\MSL}{\mathscr{L}}
\newcommand{\MSM}{\mathscr{M}}
\newcommand{\MSN}{\mathscr{N}}
\newcommand{\MSO}{\mathscr{O}}

\newcommand{\MSQ}{\mathscr{Q}}

\newcommand{\MSU}{\mathscr{U}}

\newcommand{\rank}{\mathrm{rank}}
\newcommand{\Id}{\mathrm{Id}}

\newcommand{\ps}{\mathrm{ps}}

\newcommand{\PGL}{\mathrm{PGL}}

\newcommand{\Aut}{\mathrm{Aut}}

\newcommand{\vol}{\mathrm{vol}}

\newcommand{\id}{\mathrm{id}}
\newcommand{\MMD}{\MM_{\mathrm{Dol}}}

\newcommand{\SLC}{\mathrm{SL}_2(\mathbb{C})}

\newcommand{\SL}{\mathrm{SL}}

\newcommand{\tX}{\tilde{X}}

\newcommand{\nil}{\mathrm{nil}}

\newcommand{\MD}{\mathcal{D}}

\newcommand{\Sym}{\mathrm{Sym}}

\newcommand{\mbR}{\mathbb{R}}

\newcommand{\om}{\omega}

\newcommand{\mbZ}{\mathbb{Z}}
\newcommand{\mbL}{\mathbb{L}}

\newcommand{\mbfE}{\mathbf{E}}

\newcommand{\mbfV}{\mathbf{V}}
\newcommand{\mbfW}{\mathbf{W}}

\newcommand{\mbfc}{\mathbf{c}}

\newcommand{\mbfs}{\mathbf{s}}

\newcommand{\vmu}{\vec{\mu}}

\newcommand{\mbC}{\mathbb{C}}
\newcommand{\msA}{\mathscr{A}}

\newcommand{\msE}{\mathscr{E}}
\newcommand{\msF}{\mathscr{F}}

\newcommand{\msJ}{\mathscr{J}}

\newcommand{\msL}{\mathscr{L}}

\newcommand{\msU}{\mathscr{U}}
\newcommand{\msV}{\mathscr{V}}
\newcommand{\msW}{\mathscr{W}}

\newcommand{\Div}{\mathrm{Div}}

\newcommand{\PSLR}{\mathrm{PSL}_2(\mathbb{R})}
\newcommand{\Higgs}{\mathrm{Higgs}}
\newcommand{\mfu}{\mathfrak{u}}
\newcommand{\GL}{\mathrm{GL}}

\newcommand{\Prym}{\mathrm{Prym}}
\newcommand{\Pic}{\mathrm{Pic}}

\newcommand{\Hit}{\mathrm{Hit}}

\newcommand\sh{\mathsf{h}}
\newcommand\ssd{\mathsf{sd}}
\newcommand{\VHS}{\mathrm{VHS}}

\usepackage{calligra}
\DeclareMathAlphabet{\mathcalligra}{T1}{calligra}{m}{n}

\declaretheoremstyle[
headfont=\color{blue}\normalfont\bfseries,
bodyfont=\color{blue}\normalfont\itshape,
]{colored}

\usepackage[utf8]{inputenc}
\usepackage[english]{babel}
\usepackage{fancyhdr}
\usepackage{accents}

\begin{document}
	\title[On the spectral variety for rank two Higgs bundles]{On the spectral variety for rank two Higgs bundles}
	\author{Siqi He} 
	\address{Siqi He, Institute of Mathematics, Academy of Mathematics and Systems Science, Chinese Academy of Sciences, Beijing, 100190, China}
	\email{\href{sqhe@amss.ac.cn}{sqhe@amss.ac.cn}}	

     \author{Jie Liu} %
\address{Jie Liu, Institute of Mathematics, Academy of Mathematics and Systems Science, Chinese Academy of Sciences, Beijing, 100190, China}
\email{\href{jliu@amss.ac.cn}{jliu@amss.ac.cn}}

	\maketitle
	\begin{abstract}
		In this article, we study the Hitchin morphism over a smooth projective variety $X$. The Hitchin morphism is a map from the moduli space of Higgs bundles to the Hitchin base, which in general not surjective when the dimension of X is greater than one. T.~Chen and B.~Ng\^{o} introduced the spectral base, which is a closed subvariety of the Hitchin base. They conjectured that the Hitchin morphism is surjective to the spectral base and also proved that the surjectivity is equivalent to the existence of finite Cohen-Macaulayfications of the spectral varieties.
		
		For rank two Higgs bundles over a projective manifold $X$, we explicitly construct a finite Cohen-Macaulayfication of the spectral variety as a double branched covering of $X$, thereby confirming Chen--Ng\^{o}'s conjecture in this case. Moreover, using this Cohen-Macaulayfication, we can construct the Hitchin section for rank two Higgs bundles, which allows us to study the rigidity problem of the character variety and also to explore a generalization of the Milnor-Wood type inequality.
	\end{abstract}
	\section{Introduction}
	
	Let $X$ be a smooth projective variety. A \emph{Higgs bundle} is a pair $(\MSE,\varphi)$ consisting of a holomorphic vector bundle $\MSE$ over $X$ and a holomorphic map $\varphi: \MSE \rightarrow \MSE \otimes \Omega_X^1$, subject to the integrability condition $\varphi \wedge \varphi = 0$. Under suitable stability condition, the non-abelian Hodge correspondence establishes a connection between irreducible representations of the fundamental group of $X$ and the stable Higgs bundles over $X$ \cite{donaldson1987twisted,hitchin1987self,corlette1988flat,Simpson1988Construction}. The moduli space of Higgs bundles has been extensively employed to investigate the topology and geometry of character varieties \cite{Simpson1991,Simpson1992,simpson1994moduli,simpson1996hodge}. 
 
    Let $\MMH$ be the moduli space of polystable Higgs bundles with rank $r$ and  $\MMH^{\textup{stack}}$ the moduli stack of Higgs bundle with rank $r$. A powerful tool to study the moduli space of Higgs bundles is the Hitchin morphism, defined by Hitchin \cite{hitchin1987self, hitchin1992lie}. More precisely, by taking the coefficients of the characteristic polynomial of the Higgs field, we obtain a map 
    \[
    \sh_X: \MMH^{\textup{stack}} \rightarrow \mathcal{A}_X := \bigoplus_{i=1}^r H^0(X,\Sym^i \Omega_X^1),
    \]
    which is called the \emph{Hitchin morphism}. The affine space $\MA_X$ is called the \emph{Hitchin base}.

    If $X$ is a smooth projective curve, the integrability condition $\varphi \wedge \varphi = 0$ is automatically satisfied and the Hitchin morphism $\sh_X:\MMH\rightarrow \MA_X$ is known to be surjective \cite{hitchin1987self, hitchin1987stable}. Furthermore, the moduli space of Higgs bundles forms an algebraically completely integrable system. The Hitchin and Beauville-Narasimhan-Ramanan (BNR) correspondence describes the fibers of the Hitchin morphism \cite{hitchin1987stable, BeauvilleNarasimhanRamanan1989}. For a general $\textbf{s} \in \mathcal{A}_X$, we can define the spectral curve $X_{\textbf{s}}$ and $\sh_X^{-1}(\textbf{s})$ can be identified with the rank one torsion-free sheaves on $X_{\textbf{s}}$.
	
	However, generalizing the Hitchin and BNR correspondence to higher dimension poses significant challenges. Due to the additional integrability condition $\varphi \wedge \varphi = 0$, the Hitchin morphism is no longer surjective in general. Extensive studies on the spectral correspondence can be found in \cite{donagi1995spectral, Simpson1991, schottenloher1995metaplectic, gallego2021higgs}.
	
	To address the integrability condition, Chen and Ng\^{o} introduced in \cite[Proposition 5.1]{ChenNgo2020} the \emph{spectral base} $\MS_X$ as a closed subset of the Hitchin base $\mathcal{A}_X$ and they proved that $\sh_X: \MMH^{\textup{stack}} \rightarrow \mathcal{A}_X$ factors through the inclusion $\iota_X: \MS_X \rightarrow \mathcal{A}_X$, yielding the following commutative diagram:
	\begin{equation}
		\begin{tikzcd}
			\MMH^{\textup{stack}} \arrow[d,"{\ssd_X}" left] \arrow{dr}{\sh_X} &  \\
			\MS_X \arrow{r}{\iota_X} & \MA_X .
		\end{tikzcd}
	\end{equation}

    The map $\ssd_X$ is called the \emph{spectral morphism} and an element of $\MS_X$ is called a \emph{spectral datum}. For every spectral datum $\textbf{s}\in \MS_X$, one can define a closed subscheme $X_{\textbf{s}} \subset \tot(\Omega_X^1)$, which is called the \emph{spectral variety}. Chen and Ng\^o proved in \cite[Proposition 6.3]{ChenNgo2020} that under some addition assumption on $X_{\textbf{s}}$, the fibre $\sh_X^{-1}(\textbf{s})$ can be identified with maximal Cohen-Macaulay sheaves of generic rank one over $X_{\textbf{s}}$. When $X$ is of dimension two, a good finite Cohen-Macaulayification of $X_{\textbf{s}}$ can always be abstractly constructed \cite[Theorem 7.3]{ChenNgo2020}, see also \cite[Theorem 1.2]{SongSun2021}. However, in general, the spectral variety may exhibit significant singularities, and it remains unknown whether generically rank one maximal Cohen-Macaulay sheaves over $X_{\textbf{s}}$ exist or not. Chen and Ngô proposed the following conjecture.
	\begin{conjecture}[\protect{\cite[Conjecture 5.2]{ChenNgo2020}}]
		\label{conj_ngo_chen}
		The spectral morphism $\ssd_X:\MMH^{\textup{stack}}\rightarrow \MS_X$ is surjective.
	\end{conjecture}
     In fact, the Conjecture \ref{conj_ngo_chen} is formulated in \cite{ChenNgo2020} for $G$-Higgs bundles, where $G$ is a split reductive group, and it is confirmed for $\textup{GL}_r(\mathbb{C})$ Higgs bundle over any smooth projective surfaces in \cite{ChenNgo2020} and \cite{SongSun2021}.

     From now on we shall always assume that $r=2$ unless otherwise stated. In this paper, we will investigate in details the explicit geometric structure of the spectral base $\MS_X$, which is closely related to the theory of rank one symmetric differentials introduced and studied by Bogomolov and de Oliveira \cite{bogomolov2011symmetric, bogomolov2013closed}. 
     Let $\MS_X^{\nil}$ be the closed subset of $\MS_X$ consisting of elements 
     \[
     \textbf{s}=(s_1,s_2)\in \MS_X\subset \MA_X:=H^0(X,\Omega_X^1)\oplus H^0(X,\Sym^2\Omega_X^1)
     \]
     such that $4s_2-s_1^2=0$. We should note that the Higgs bundle $(\MSE,\vp)$ with $\ssd_X(\MSE,\vp)\in \MS_X^{\nil}$ is nilpotent. Our first main theorem is the following existence result for Cohen-Macaulayfications of spectral varieties.
	
	\begin{theorem}
		\label{t.spectral-correspondence-GL2}
		Let $X$ be a projective manifold and let $\MS_X$ be the spectral base for $\GL_2(\mbC)$ Higgs bundle. Then for any element $\mbfs=(s_1,s_2)\in \MS_X\setminus \MS_X^{\nil}$, there exists a Cohen-Macaulayfication $\widetilde{X}_{\mbfs}$ of the spectral variety $X_{\mbfs}$ as a double branched covering of $X$ such that there is a bijective correspondence between isomorphism classes of maximal Cohen-Macaulay sheaves of generic rank one on $\widetilde{X}_{\mbfs}$ and isomorphism classes of rank two Higgs bundles $(\MSE,\varphi)$ on $X$ with $\Tr(\varphi)=s_1$ and $\det(\varphi)=s_2$. 
	\end{theorem}
	
	As an immediate application, we can confirm Chen--Ng\^{o}'s conjecture for the $\mathrm{GL}_2(\mathbb{C})$ case in a stronger sense.
	
	\begin{corollary}
		\label{c.Chen-Ngo-GL2}
		The spectral morphism $\ssd_X:\MMH\rightarrow \MS_X$ is surjective. In particualr, the Conjecture \ref{conj_ngo_chen} holds for $\mathrm{GL}_2(\mathbb{C})$ Higgs bundles.
	\end{corollary}

        The statement of Corollary \ref{c.Chen-Ngo-GL2} is stronger than the original  Conjecture \ref{conj_ngo_chen}: the restrictton of $\ssd_X$ to the subset $\MMH$ is already surjective. Moroever, we remark that in Corollary \ref{c.Chen-Ngo-GL2}, one cannot replace $\MMH$ by the Dolbeault moduli space $\MMD$ parametrising the topologically trivial polystable Higgs bundles of rank two -- see Example \ref{e.CI-Projective-spaces} (5).
        
        Furthermore,  as a byproduct of the proof of Theorem \ref{t.spectral-correspondence-GL2}, we establish that for every non-nilpotent rank two Higgs bundle, the Higgs field can be factored through a line bundle twisted Higgs bundle. Specifically, we prove the following result, which is of independent interests.

\begin{theorem}
\label{t.LtwistedHiggs}
Given a spectral datum $\textbf{s}=(s_1,s_2)\in \MS_X\setminus \MS_X^{\nil}$, there exists a line bundle $\msL$ and an element $0\not=\alpha\in H^0(X,\msL^{-1}\otimes \Omega_X^1)$ such that for every Higgs bundle $(\msE,\vp)$ with $\ssd_X([\MSE,\vp])=\mbfs$,  there exists an $\msL$-twisted Higgs field $\vp_0\in H^0(X,\End(\msE)\otimes \msL)$ such that 
\[
\vp=\alpha\circ\vp_0+\frac12s_1\otimes \Id_{\msE}.
\]
\end{theorem}
 
	Using the construction of the Cohen-Macaulayfication, we make several applications to understand the rank two character variety of $X$. When $X$ is a curve, the Hitchin morphism is surjective and there is a canonical section, called the \emph{Hitchin section} and constructed by Hitchin \cite{hitchin1987self, hitchin1992lie}, which is closely related to Teichmüller theory \cite{wienhard2018invitation}. As another application of Theorem \ref{t.spectral-correspondence-GL2}, we can define a canonical section of the Hitchin morphism for rank two Higgs bundles over higher dimensional projective manifolds, which generalizes the classical construction of Hitchin.
	
	\begin{theorem}
 \label{t.Hitchinsection}
		For the spectral morphism $\ssd_X:\MMH\to \MS_X$, there exists a Hitchin section $\chi_{\Hit}:\MS_X\setminus\{(0,0)\}\to \MMH$ with the image real Higgs bundles such that $\ssd_X\circ \chi_{\Hit}=\Id$. 
	\end{theorem}

In the remainder of the article, we delve into further discussions on the spectral base and its applications, which include exploring the rigidity problem of the character varieties and establishing a Milnor-Wood type inequality. Recall that a character variety is called \emph{rigid} if every representation is isolated. For smooth projective varieties $X$ with Picard number one, the rigidity of rank two character variety is fully determined by the homology of $X$.
\begin{theorem}
\label{t.PicardnumberoneRigidity}
    Let $X$ be a smooth projective variety with Picard number one.  Then the $\GLC$ character variety of $X$ is rigid if and only if $b_1(X)=0$ and for any unramified double covering $\tX\to X$, we have $b_1(\tX)=0.$
\end{theorem}

Let $\rho:\pi_1(X)\rightarrow \SU(1,1)$ be a reductive representation. As established in \cite{burger2007bounded, KoziarzMaubon2010}, the Toledo invariant for $\rho$ depends on a mobile curve class $\gamma$, and we denote $\tau_{\rho}$ as the Toledo invariant. Consequently, we obtain the following Milnor-Wood type inequality:

    \begin{theorem}
    \label{t.MWineq}
    	Let $X$ be a projective manifold and let $\rho:\pi_1(X)\rightarrow \SU(1,1)$ be a reductive representation. If $X$ is non-uniruled, then for any mobile curve class $\gamma$, we have
    	\[
    	|\tau_{\gamma}(\rho)| \leq \frac{1}{2}\deg_{\gamma}(K_X)=\frac{1}{2}c_1(K_X)\cdot \gamma.
    	\]
    \end{theorem}
This paper is organized as follows: In Section \ref{sec_Higgsbundles_nonabelian_Hodge}, we provide a brief overview of the non-abelian Hodge correspondence. In Section \ref{sec_spectral_base_rank_one_differentials}, we introduce the relationship between the spectral base and rank one symmetric differentials. In Section \ref{sec_spectralvarietyandCM}, we discuss the spectral variety and give an explicit construction of the Cohen-Macaulayfication. In particular, we prove Theorem \ref{t.spectral-correspondence-GL2}, Corollary \ref{c.Chen-Ngo-GL2} and Theorem \ref{t.LtwistedHiggs}. In Section \ref{sec_realHiggsbundle_Hitchin_section}, we will construct the Hitchin section for the Hitchin morphism and prove Theorem \ref{t.Hitchinsection}. In Section \ref{sec_rigidity_character_variety}, we will discuss the rigidity of character varieties and the Hitchin base. In particular we will prove Theorem \ref{t.PicardnumberoneRigidity}. In Section \ref{sec_further_discussions_applications}, we will discuss several applications of the Hitchin morphism, including the proof of Theorem \ref{t.MWineq} for the general Milnor-Wood type inequality.

\textbf{Acknowledgements.} 
	The authors also wish to express their gratitude to a great many people for their interest and helpful comments. 
	Among them are Damian Brotbek, Mark de Cataldo, Tsao-Hsien Chen, Qiongling Li, Shizhang Li, Rafe Mazzeo, Richard Wentworth. S.~He is supported by NSFC grant
	No.12288201. J.~Liu is supported by the National Key Research and Development Program of China (No. 2021YFA1002300), the NSFC grants (No. 12001521 and No. 12288201), the CAS Project for Young Scientists in Basic Research (No. YSBR-033) and the Youth Innovation Promotion Association CAS.
	
	\section{Higgs bundles and non-abelian Hodge correspondence}
	\label{sec_Higgsbundles_nonabelian_Hodge}
	In this section, we will discuss the moduli space of polystable Higgs bundles and the non-abelian Hodge correspondence over a smooth projective variety. This topic has been extensively studied in works such as \cite{hitchin1987self, Simpson1988Construction, simpson1994moduli, simpson1994moduli2}. We also recommend the surveys \cite{garcia2015introduction, Wentworth2016} for further reading. Furthermore, we will review the hyperk\"ahler structure on the Hitchin moduli space, which was introduced in \cite{hitchin1987self,fujiki2006hyperkahler}.
 
	\subsection{Hermitian geometry and the Hermitian--Yang-Mills equation}
 
	In this subsection we will introduce the background Hermitian geometry and give an introduction of the Hermitian-Yang-Mills equations.
 
	\subsubsection{Hermitian geometry}
 
	To begin with, we will review some terminology of Hermitian geometry. Let $X$ be a smooth projective variety. We fix a K\"ahler form $\omega$ with the associated Kh\"aler metric $H$. in an orthogonal frame $dz_1,\cdots,dz_n$ with $|dz_i|=2$, we could write $\omega=\frac{\sqrt{-1}}{2}\sum_{i=1}^ndz_i\wedge dz_i$. Then we could define the contraction map $\Lam:\Omega^{p,q}\to \Omega^{p-1,q-1}$ and its action on $\Omega^{1,1}$ is given by $\Lam(\sum^n_{i,j=1}a_{ij}dz_i\wedge d\bar{z}_j)=-2i\sum_{i=1}^n a_{ii}.$
	Moreover, for $\al\in \Omega^{1,1}$, we have $\omega^n=n!\cdot\vol_g$ with $\Lam \al\;\dot\vol_g=\frac{1}{(n-1)!}\al\we \omega^{n-1}$.
	
	Let $E$ be a complex smooth vector bundle with $H$ a fixed Hermitian metric on $E$. A connection $D$ on $(E,H)$ is called \emph{unitary} if for any sections $s_1,s_2\in \Gamma(X,E)$, we have 
    \[
    dH(s_1,s_2)=H(D s_1,s_2)+H(s_1,Ds_2).
    \]
	
	We denote by $\MA^{\mbC}(E)$ the space of complex connections on $E$. For any $D\in\MA^{\mbC}(E)$, we could write $\End(E)=\mfu(E)\oplus i\mfu(E)$ and $D=d_A+\phi$ using the Hermitian metric $g$, where $d_A$ is an unitary connection and $\phi\in \Omega^{1}(i\mfu(E))$. We can furthermore decompose into type
    \[
    d_A=\bar{\pa}_A+\pa_A,\quad \phi=\vp+\vp^{\da}
    \]
    with $\vp\in \Omega^{0,1}(i\mfu(E))$. In addition, we write $\MA(E,H)$ the space of unitary connection for the Hermitian bundle $(E,H).$ The curvature $F_D$ of the connection $D$ could be explicitly written as $$F_D=F_A+\phi\we\phi+d_A\phi,$$
	where $F_A+\phi\we\phi\in \Omega^2(\mfu(E))$ and $d_A\phi\in \Omega^2(i\mfu(E)).$
	
	For complex smooth vector bundle $E$ equipped with a connection $D$, the first Chern class $c_1(E)$ of $E$ could be represented by $\frac{\sqrt{-1}}{2\pi}\Tr(F_D)$. The \emph{degree $\deg_{\omega}(E)$ of the vector bundle $E$} (with respect to $\omega$) is defined to be
	\begin{equation}
		\deg_{\omega}(E)=\int_Xc_1(E)\wedge\omega^{n-1}=\frac{(n-1)!}{2\pi}\int_X\Tr(i\Lam F_D)\vol. 
	\end{equation}
	
	Let $\MGC(E):=\Aut(E)$ be the complex gauge transformation group of $E$. For any $g\in\MGC(E)$, the complex gauge transformation action on $g\cdot(d_A,\phi)$ is given by 
	$$
	g\cdot d_A:=g^{-1}\circ\bar{\pa}_A \circ g+g^{\da}\circ\pa_A\circ (g^{\da})^{-1},\;g\cdot\phi:=g^{-1}\vp g+g^{\da}\vp^{\da}(g^{\da})^{-1}.
	$$
	Moreover, for a Hermitian vector bundle $(E,H)$, the \emph{unitary gauge group} $\MG(E)$ consists of the gauge transformations $g\in \MGC(E)$ such that $H(gs_1,gs_2)=H(s_1,s_2)$ for any local sections $ s_1$ and $s_2$ of $E$.
 
	\subsubsection{Hermitian-Yang-Mills equations} 
 
    Now, we will introduce the Hermitian-Yang-Mills equations that have been studied in \cite{hitchin1987self,Simpson1988Construction}. 
	
	Given a Hermitian bundle $(E,H)$, the Hermitian-Yang-Mills equations are equations for a unitary connection $d_A$ and a Hermitian $1$-form $\phi=\vp+\vp^{\da}$ satisfying
	\begin{equation}
		\label{eq_Hitchin-Simpson}
		\begin{split}
			&F_A^{0,2}=0,\;\vp\wedge\vp=0,\;\bar{\pa}_A\vp=0,\\
			&\sqrt{-1}\Lam (F_A^{\perp}+[\vp,\vp^{\da}])=0,
		\end{split}
	\end{equation}
	where $F_A^{\perp}=F_A^{1,1}-\gamma\Id$ is the trace free part of $F_A^{1,1}$ and $\gamma=\frac{2\pi\deg_{\omega}(E)}{(n-1)!\rank(E)}.$ The first line of \eqref{eq_Hitchin-Simpson} is $\MGC(E)$ invariant and the second line of \eqref{eq_Hitchin-Simpson} is $\MG(E)$ invariant. 
	
	We could also write the Hermitian-Yang-Mills equation \eqref{eq_Hitchin-Simpson} using the complex connection $D=d_A+\varphi$. Recall we have the K\"ahler identity $\sqrt{-1}[\Lam,\bar{\pa}_A]=\pa_A^{*},\;\sqrt{-1}[\Lam,\pa_A]=-\bar{\pa}_A^*.$ Let $\MA_0^{\mbC}(E)$ be the set of complex connections $D\in \MA^{\mbC}(E)$ with $F_D\in \Omega^{1,1}$. 
 
	\begin{lemma}\label{lem_weinzenbock_identity}
		Let $D=d_A+\phi$ be a complex connection such that $D\in \msA_0^{\mbC}(E)$ and $\pa_A^{*}\vp=0$, then $F_A^{2,0}=0,\;\vp\we\vp=0,\;\bar{\pa}_A\vp=0$.
	\end{lemma}
 
	\begin{proof}
		As $D\in \msA_0^{\mbC}(E)$, the equality $ F_D^{2,0}=0$ implies $F_A^{2,0}+\vp\we \vp=0,\;\bar{\pa}_A\vp^{\da}=0.$ Using the K\"ahler identity, $\pa_A^{*}\vp=0$ implies $\Lam\bar{\pa}_A\vp=0$ thus $\Lam\pa_A\vp^{\da}=0.$ We compute
		\begin{equation}
			\begin{split}
				\bar{\pa}_A^{*}\pa_A\vp^{\da}&=-\sqrt{-1}\Lam \pa_A\pa_A\vp^{\da}+\sqrt{-1}\pa_A\Lam\pa_A\vp^{\da}=-\sqrt{-1}\Lam[F_A^{2,0},\vp^{\da}]=\sqrt{-1}\Lam[[\vp,\vp],\vp^{\da}].
			\end{split}
		\end{equation}
		Taking inner product with $\vp$, we obtain $\|
		\pa_A\vp^{\da}\|_{L^2}^2+\|\vp\we\vp\|_{L^2}^2=0$, which implies $\vp\we\vp=0$ and $\bar{\pa}_A\vp=0$.
	\end{proof}
	
	We define the following space
	\begin{equation}
		\MD(E)=\{D\in\msA^{\mbC}(E)\,|\,i\Lam F_D=\gamma\Id,\;F_D\in \Omega^{1,1}(\End(E))\}.
	\end{equation}
	By Lemma \ref{lem_weinzenbock_identity}, for $D=d_A+\phi$, the pair $(d_A,\phi)$ satisfies the Hermitian-Yang-Mills equations \eqref{eq_Hitchin-Simpson} if and only if $D\in \MD(E)$.
	
	The tangent space of $\MD(E)$ at the point $D$ is the set of $\mu\in \Omega^1(\End(E))$ such that $D+\mu$ satisfies the Hermitian-Yang-Mills equations up to first order:
	\begin{equation}
		\begin{split}
			T_D\MD(E):=\{\mu\in \Omega^1(\End(E))\,|\, D\mu\in \Omega^{1,1}(\End(E)),\;\Lam D\mu=0\}
		\end{split}
	\end{equation}
	The almost complex structure on $\GL(n,\mbC)$ defines a almost complex structure $J$ on $\MD(E)$ such that $J\mu=i\mu$. 
	
	As the unitary gauge transformation group $\MG(E)$ preserve \eqref{eq_Hitchin-Simpson}, the moduli space of the Hermitian-Yang-Mills equations is defined as
	\begin{equation}
		\begin{split}
			\MMHYM(E):=\MD(E)/\MG(E).
		\end{split}
	\end{equation}
	Moreover, the almost complex structure $J$ extends naturally to the $L^2$ complement of the $\MG(E)$ orbit and thus induces an almost complex structure on the Hermitian-Yang-Mills moduli space $\MMHYM(E)$. 
	
	\subsection{Higgs bundles and non-abelian Hodge correspondece}
 
	In this subsection, we will review the construction of the moduli space of rank two polystable Higgs bundle. Throughout this subsection $X$ is assumed to be an $n$-dimensional projective manifold equipped with a K\"ahler class $\omega$ associated to given ample divisor on $X$. 
 
	\subsubsection{Moduli space of Higgs bundles}
	Let $E$ be a complex smooth vector bundle over $X$. We write $\Omega^{p,q}(E)$ the $E$-valued $(p,q)$-forms on $X$. In the sequel of this paper, we will naturally identify the holomorphic structures on $E$ with the $\bar{\pa}$-operators $\bar{\pa}_E:\Omega^{p,q}(E)\to \Omega^{p,q+1}(E)$ satisfying the integrability condition $\bar{\pa}_E^2=0$. We denote by $\msE:=(E,\bar{\pa}_E)$ the holomorphic bundle with the holomorphic structure defined by $\bar{\pa}_E$ if there is no confusion. We also denote by $\Omega_X^1$ the holomorphic cotangent bundle of $X$.
	
	\begin{definition}
		A Higgs bundle on $X$ is a pair $(\msE,\vp)$, where $\msE$ is a holomorphic vector bundle on $X$ and $\vp\in H^0(X,\End(\msE)\otimes \Omega_X^1)$ such that $\vp\wedge \vp=0\in \Omega^{2}(\End(\msE))$. We call $\vp$ a Higgs field and the equation $\vp\we \vp=0$ the Higgs equation. 
	\end{definition}

    We will mainly focus on the rank two case in this paper and for simplicity we introduce the following notions.
    
    \begin{definition}
        A Higgs bundle $(\msE,\vp)$ on a projective manifold $X$ is said to be a $\GL_2(\mbC)$ Higgs bundle if $\msE$ is of rank two and is called an $\SLC$ Higgs bundle if it is a $\GLC$ Higgs bundle such that $\det(\msE)=\MSO_X$ and $\Tr(\vp)=0$. 
    \end{definition}
	
	Given a torsion free coherent sheaf $\MSF$ on an $n$-dimensional projective manifold $X$, the \emph{slope $\mu(\MSF)$} of $\MSE$ with respect to $\omega$ is defined to be 
    $$
    \mu(\MSF)=\frac{\deg_{\omega}(\MSF)}{\rank \MSF}=\frac{c_1(\MSF)\cdot \omega^{n-1}}{\rank \MSF}.
    $$ 
    
    Given a Higgs bundle $(\MSE,\varphi)$ on a projective manifold $X$, a coherent subsheaf $\MSF\subset \msE$ is said to be \emph{$\vp$-invariant} if $\vp(\MSF)\subset \MSF\otimes\Omega_X^1$. Now we can introduce the notation of slope stability.
	
	\begin{definition}
 \label{d.stability}
		A Higgs bundle $(\msE,\vp)$ is called stable (resp. semistable) if for all $\vp$-invariant coherent subsheaf $\MSF\subsetneq \msE$, we have 
        \[
        \mu(\MSF)<\mu(\msE)\quad (\textup{resp.}\,\mu(\MSF)\leq \mu(\MSE)).
        \]
        A Higgs bundle $(\msE,\vp)$ is called polystable if $(\msE,\vp)$ is semistable and 
       \[
       (\msE,\vp)\cong (\msE_1,\vp_1)\oplus\cdots \oplus (\msE_r,\vp_r)
       \]
       where $(\msE_i,\vp_i)$ are Higgs bundles with the same slope.
	\end{definition}
	
	Let $\MH(E)$ denote the set of Higgs bundles $(\MSE,\varphi)$ such that the underlying complex smooth vector bundle of $\MSE$ is isomorphic to some given complex smooth vector bundle $E$. Then the complex gauge transformation group $\MGC(E)$ acts on $\MH(E)$ as 
    \[
    g\cdot(\bar{\pa}_E,\vp)=(g^{-1}\circ\bar{\pa}_E\circ g,g^{-1}\circ\Phi\circ g),
    \]
    where $\MSE=(E,\bar{\pa}_E)$. We write $[(\msE,\vp)]$ be the equivalent class of $(\msE,\vp)$ in the orbit of $\MGC(E)$. Let $\MH^{\ps}(E)$ be the subset of polystable Higgs bundles in $\MH(E)$. Then we can define the following spaces:
	\begin{equation}
		\MM_{\Higgs}(E):=\MH^{\ps}(E)/\MGC(E),\;\MM_{\Higgs;0}(E):=\{(\msE,\vp)\in\MH^{\ps}(E)\,|\,\Tr(\vp)=0\}/\MGC(E).    
	\end{equation}
	
	A complex smooth vector bundle $E$ is called \emph{topologically trivial} if all the Chern classes of $E$ in $H^{*}(X;\mathbb{Q})$ vanish. Following \cite[Proposition 6.6]{simpson1994moduli2}, we define the Dolbeault moduli space $\MMD$ to be the moduli space parametrizing rank two topologically trivial polystable Higgs bundles on $X$, which is a quasiprojective variety. It follows from \cite[Proposition 3.4]{Simpson1988Construction} that a polystable Higgs bundle $(\msE,\vp)$ is topological trivial if $c_1(\msE)\cdot \omega^{n-1}=0$ and $c_2(\msE)\cdot \omega^{n-2}=0$. 
	
	There is a canonical almost complex structure on the Higgs bundle moduli space. Given $(a,b)\in \Omega^{0,1}(\End(E))\oplus \Omega^{1,0}(\End(E))$, then the tangent space to $\MH(E)$ at $(\bar{\pa}_E,\vp)$ is the pair $(a,b)$ such that $(\bar{\pa}_E+a,\vp+b)$ is a Higgs bundle up to first order. To be more explicitly, the tangent space at $(\bar{\pa}_E,\vp)$ could be written as
	\begin{equation}
		T_{(\bar{\pa}_E,\vp)}\MH(E)=\{(a,b)\in \Omega^{0,1}(\End(E))\oplus \Omega^{1,0}(\End(E))|\bar{\pa}_E b+[\vp,a]=0,\;[\vp,b]=0\}.
	\end{equation}
	The space $\MH(E)$ exists a natural almost complex structure $I$ defined as $I(a,b)=(ia,ib)$. Moreover, the almost complex strucutre $I$ extends to the $L^2$ completement of the $\MGC(E)$ orbit and thus induces an almost complex structure on $\MMH(E)$. 

    \subsubsection{Projectivization of vector bundles}
    \label{s.projectivisation}

    Let $\MSE$ be a holomorphic vector bundle of rank $\geq 2$ over $X$. Throughout this paper, we denote by $\mbP(\MSE)$ the projectivization of $\MSE$ \emph{in the sense of Grothendieck}. More precisely, denote by $\pi:\mbP(\MSE)\rightarrow X$ the natural projection. Then for any point $x\in X$, the fiber $\pi^{-1}(x)=\mbP(\MSE_x)$ parametrizes the codimension one subspaces of $\MSE_x$, or equivalently the one dimensional subspaces of its dual $\MSE^*_x$. Then we have the following Euler sequence
    \[
    0\rightarrow \MSU \rightarrow \pi^*\MSE \rightarrow \MSO_{\mbP(\MSE)}(1) \rightarrow 0,
    \]
    where $\MSU$ is the universal bundle; that is, the fiber of $\MSU$ over a point $(x,[H])\in \mbP(\MSE_x)$, coincides with the codimension one subspace $H$ of $\MSE_x$. The quotient line bundle $\MSO_{\mbP(\MSE)}(1)$ is called the \emph{tautological line bundle} of $\mbP(\MSE)$. Moreover, for any non-negative integer $k\geq 0$, we have a canonical isomorphism
    \[
    H^0(\mbP(\MSE),\MSO_{\mbP(\MSE)}(k)) \cong H^0(X,\Sym^k \MSE),
    \]
    which can be viewed as following: given an element $s\in H^0(\mbP(\MSE),\MSO_{\mbP(\MSE)}(k))$, then we can view the restriction $s|_{\mbP(\MSE_x)}$ as an element of $H^0(\mbP(\MSE_x),\MSO_{\mbP(\MSE_x)}(k))$ and it can be glued to be a global section of $H^0(X,\Sym^k\MSE)$ as $x$ varies.
	
	\subsubsection{The non-abelian Hodge correspondence over projective variety.}
	In this section, we will brieftly introduce the Hitchin equation over projective variety and the non-abelian Hodge correspondence developed in \cite{hitchin1987self,Simpson1988Construction}. We also refer \cite{garcia2015introduction,Wentworth2016} for a nice introduction.
	
	\begin{theorem}[\protect{\cite{hitchin1987self,Simpson1988Construction}}]\label{thm_NonabelianHodge}
		Given a polystable Higgs bundle $\msE=(E,\bar{\pa}_{E},\vp)$, 
		we defined $D=\bar{\pa}_{E}+(\bar{\pa}_{E})^{\da_{H}}+\vp+\vp^{\da_{H}}$, where $\da_{H}$ is the adjoint defined by $H$. Then there exists a complex gauge transformation $g\in\MGC(E)$ such that $D'=g\cdot D$ satisfies the for Hermitian--Yang-Mills equations \eqref{eq_Hitchin-Simpson}. 
	\end{theorem}
	
	Using the above theorem, one could define the Kobayashi-Hitchin map 
	\begin{equation}
		\begin{split}
			\Xi&:\MM_{\Higgs}(E)\to \MMHYM(E),\;\Xi(\bar{\pa},\vp):=D,
		\end{split}
	\end{equation}
	where $D$ is the complex connection obtained in Theorem \ref{thm_NonabelianHodge}. Moreover, $\Xi$ is a bijection and real analytic over the smooth locus. 
	
	When $\dim(X)\geq 2$, the existence the solutions to the Hermitian-Yang-Mills equations is not equivalent to flat connection. For a rank two vector bundle $\msE$, recall that the \emph{discriminant} of $\Delta(\msE)$ is defined to be $\Delta(\msE):=4c_2(\msE)-c_1^2(\msE)$ and we have $\Delta(\msE\otimes \msL)=\Delta(\msE)$ for any line bundle $\msL$. 
	\begin{theorem}[\protect{\cite[Proposition 3.4]{Simpson1988Construction}}]
		\label{thm_bogomolov}
		Let $(\msE,\vp)$ be a polystable rank two Higgs bundle with $D:=\Xi(\msE,\vp)$ be the complex connection obtained from the non-abelian Hodge correspondence, then $$\Delta(\msE).[\omega]^{n-2}=0$$ if and only if $F_D^{\perp}=0$. Moreover, if $c_1(\msE).[\omega]^{n-1}=0,\;c_2(\msE).[\omega]^{n-2}=0$, then $D$ is flat. 
	\end{theorem}
	
	Now, we could discuss the situation that a solutions to \eqref{eq_Hitchin-Simpson} defines a representation of the fundamental group $\pi_1(X)$ via the Riemann-Hilbert correspondence \cite{Simpson1988Construction}. When $\Delta(\msE).[\omega]^{n-2}=0$, the connection $D$ defines a flat connection on the projectivized bundle $\mbP(\msE)$ and the monodromy defines a representation $\rho:\pi_1(X)\to \PGL_2(\mbC)$. If $c_1(\msE).[\omega]^{n-1}=0$, then $F_{D}=0$ and we obtain a representation $\rho:\pi_1(X)\to \GL_2(\mbC)$. Suppose $(\msE,\vp)$ is an $\SLC$ Higgs bundle, i.e., $\det(\msE)=\MSO_X$ and $\Tr(\vp)=0$, and suppose $c_2(\msE).[\omega]^{n-2}=0$, then the representation defined by the connection $D=\Xi(\msE,\vp)$ would be $\rho:\pi_1(X)\to \SLC$. In particular, let $\mfR$ be the moduli space of completely reducible representations $\rho:\pi_1(X)\to \GLC$, then the Kobayashi-Hitchin map togather with the Riemann-Hilbert correspondences defines a bijective between the Dolbeaut moduli space $\MMD$ and $\mfR$. 
	
	\subsection{Hyperk\"ahler metric and moment maps}
 
	Let $(E,H)$ be a complex smooth Hermitian vector bundle. Recall that we denote by $\MA^{\mbC}(E)$ the space of connections on $E$ and by $\MA(E)$ the space of unitary connection on $(E,H)$. For any connection $D\in \MA^{\mbC}$, using the Hermitian metric $H$, there exists an unique decomposition $D=d_A+\phi$, where $d_A\in \MA(E)$ and $\phi\in \Omega^1(i\mfu(E))$ and this makes $D\in T^{*}\MA(E)$. Moreover, we have the decomposition $d_A=\pa_A+\bar{\pa}_A$, $\phi=\vp+\vp^{\da}$, which makes $\MH(E)\subset T^{*}\MA(E)$. 
	
	For the tangent space, we have 
	\begin{equation}
		\label{eq_decomposition_tangent_space}
		\begin{split}
			T(T^{*}\MA(E))=\Omega^1(\mfu(E))\oplus \Omega^1(i\mfu(E))=\Omega^1(\End(E))=\Omega^{0,1}(\End(E))\oplus \Omega^{1,0}(\End(E)).
		\end{split}
	\end{equation}
	For $\mu\in \Omega^1(\End(E))$, we define $$a:=\frac12(\mu-\mu^*)^{(0,1)}\in \Omega^{0,1}(\End(E)),\; b:=\frac12(\mu+\mu^*)^{(1,0)}\in\Omega^{1,0}(\End(E)),$$
	then $\mu\to (a,b)$ gives the identification in \eqref{eq_decomposition_tangent_space}. 
	
	For $(a,b)\in \Omega^{0,1}(\End(E))\oplus \Omega^{1,0}(\End(E))$, we could define the following complex structures 
	\begin{equation}
		\begin{split}
			I(a,b)=(ia,ib),\;J(a,b)=(ib^{*},-ia^{*}),\;K(a,b)=(-b^{*},a^{*}).
		\end{split}
	\end{equation}
	
	For $(a_1,b_1),\;(a_2,b_2)\in \Omega^{0,1}(\End(E))\oplus \Omega^{1,0}(\End(E))$, the $L^2$ metric could be written as
	\begin{equation}
		\begin{split}
			g((a_1,b_1),(a_2,b_2))=2\sqrt{-1}\int_M\Tr(a_1^{*}a_2-b_1b_2^{*})\omega^{n-1}.
		\end{split}
	\end{equation}
	The corresponding symplectic structures could be explicitly written as 
	\begin{equation}
		\begin{split}
			&\omega_I((a_1,b_1),(a_2,b_2))=2\Im\int_X\Lam\Tr(a_1^{\da}a_2+b_1b_2^{\da})\vol,\\
			&\Omega_I:=(\omega_J+i\omega_K)((a_1,b_1),(a_2,b_2))=2\sqrt{-1}\int_X\Tr\Lam(b_2a_1-b_1a_2)\vol.
		\end{split}
	\end{equation}
 
	Let $\MG(E)$ be the unitary gauge transformation, then $g\in\MG(E)$ acts on $(d_A,\phi)\in \msA^{\mbC}(E)$ as $g\cdot(d_A,\phi)=(g^{-1}d_Ag,g^{-1}\phi g).$ The corresponding moment maps $\vmu=(\mu_I,\mu_J,\mu_K)$ could be explicitly computed as
	\begin{equation}
		\begin{split}
			&\mu_I,\mu_J,\mu_K:\msA^{\mbC}(E)\to \Omega^0(\mfu(E)),\\
			&\mu_I(A,\phi)=\Lam(F_A+[\vp,\vp^{\da}]),\;(\mu_J+i\mu_K)(A,\phi)=2\pa_A^{*}\vp.
		\end{split}
	\end{equation}
	
	Therefore, by Lemma \ref{lem_weinzenbock_identity}, when restricting the action $\MG(E)$ on $\msA_0^{\mbC}(E)$, we have
	\begin{equation}
			\MMHYM(E)=(\msA_0(E)^{\mbC}\cap {\vmu}^{-1}(0))/\MG(E).
	\end{equation}
	Using the hyperK\"ahler reduction \cite{hitchin1987hyperkahler}, over the smooth locus of $\MMHYM(E)$, we have the hyperK\"ahler structure.
	\begin{theorem}[\protect{\cite[Theorem 8.3.1]{fujiki2006hyperkahler}\cite[Theorem 6.7]{hitchin1987self}}]
		There exists a hyperK\"ahler structure over the smooth locus of $\MMHYM(E)$. 
	\end{theorem}
	
	\section{Spectral base and rank one symmetric differentials}
 \label{sec_spectral_base_rank_one_differentials}
	In this section, we will introduce the spectral base for Higgs bundles over a projective manifold \cite{donagi1995spectral,ChenNgo2020}. In particular, we will explore in depth its relation with  rank one symmetric differentials of degree two and codimension one foliations on $X$.
	
	\subsection{Hitchin morphism and spectral base}
	
	The Hitchin morphism is a useful tool to study the Higgs bundle moduli space. In this subsection, we will introduce the the Hitchin morphism for projective manifolds, following \cite{hitchin1987self, hitchin1987stable, simpson1994moduli}. Let $X$ be a projective manifold. Then the \emph{Hitchin base} of $X$ is defined to be 
	\begin{equation}
		\MA_X:=H^0(X,\Omega_X^1)\oplus H^0(X,\Sym^2\Omega_X^1)
	\end{equation}
	
	Let $\MM_{\Higgs}$ be the space parametrizing all polystable rank two Higgs bundles on $X$. The \emph{Hitchin morphism} for the moduli space of Higgs bundle is defined as following: 
	\begin{equation}
		\begin{split}
			\sh_X:\MMH \to \MA_X,\;\sh_X[(\MSE,\vp)]:=(\Tr(\vp),\det(\varphi)).
		\end{split}
	\end{equation}

	\begin{theorem}[\protect{\cite{hitchin1987self,simpson1994moduli}}]
 \label{thm_Hitchinmap_proper}
		The restriction $\sh_X|_{\MMD}:\MMD\to \MA_X$ is proper and if $\dim(X)=1$, then it is also surjective . 
	\end{theorem}
	
	\begin{remark}
		The map $\sh_X$ is not proper over the Higgs bundle moduli space $\MMH$ with non-trivial Chern classes due to the appearance of Uhlenbeck bubbling. We refer the interested reader to \cite[Theorem 5.4]{he2020behavior} for a more precise statement.     
	\end{remark}
	
	In general the Hitchin morphism $\sh_X$ might not be surjective if $\dim(X)\geq 2$ due to the Higgs equation $\vp\we \vp=0$. In \cite{ChenNgo2020}, Chen and Ng\^o studied the image of $\sh_X$ in terms of universal spectral data. In the following we give an independent and elementary proof of their results in the rank two case. Let $s\in H^0(X,\Sym^2\Omega_X^1)$ be a symmetric differential of degree two. For a given point $x\in X$, recall that $s$ is said \emph{of rank one at $x$} if $s(x)=w^2$ for some $0\not=w\in\Omega_{X,x}^1$ and $s$ is said \emph{of rank zero at x} if $s(x)=0$. We define 
	\begin{equation}
		\MB_X:=\{s\in H^0(X,\Sym^2\Omega^{1}_X)\,|\,\rank_x s\leq 1, \,\forall\, x\in X\}    
	\end{equation}
	and then the \emph{spectral base $\MS_X$} for rank two Higgs bundles is defined as
	\begin{equation}
		\label{eq_spectral_base_GL2Higgs_bundle}
		\MS_X:=\{(s_1,s_2)\in \MA_X\,|\,4s_2-s_1^2\in \MB_X\}.
	\end{equation}  
    One can easily see from the definition that if $\MB_X=0$, then we must have $\MS_X=0$.
	
	\begin{proposition}[\protect{\cite[Proposition 5.1]{ChenNgo2020}}]
		\label{p.CN-rankone}
		The Hitchin morphism $\sh_X:\MMH\to \MA_X$ factors through the natural inclusion map $\iota_X:\MS_X\rightarrow \MA_X$. In other words, there exists a map $\ssd_X:\MMH\to \MS_X$ such that the following diagram commutes: 
		\begin{equation}
			\begin{tikzcd}[column sep=large,row sep=large]
				\MMH \arrow[d,"{\ssd_X}" left] \arrow{dr}{\sh_X} &  \\
				\MS_X \arrow[r,"{\iota_X}" below] & \MA_X .
			\end{tikzcd}
		\end{equation}
		The map $\ssd_X$ is called the spectral morphism.
	\end{proposition}
	
	\begin{proof}
		Let $(\MSE,\vp)$ be a Higgs bundle such that $\Tr(\vp)=s_1$ and $\det(\vp)=s_2$. Then we could define a new Higgs field $\vp':=\vp-\frac{1}{2}s_1$ on $\msE$, which satisfies 
		\begin{center}
			$\Tr(\vp')=0$, $\det(\vp')=s_2-\frac{1}{4} s_1^2$ and $\vp'\wedge \vp'=0.$
		\end{center}
		Let $dz_1,dz_2,\cdots, dz_n$ be a frame of $\Omega^{1}_X$ at $x\in X$. If we write $\vp'(x)=\sum_{i=1}^nB_idz_i$, then the condition $\vp'\we\vp'=0$ implies that $[B_i,B_j]=0$. Since $\Tr(\vp')=0$, we can write $\lambda_i,-\lambda_i$ as two eigenvalues of $B_i$. In particular, either $\det(\vp')(x)=0$, or there exists some $1\leq i \leq n$ such that $\lambda_i\not=0$ and $B_1,\cdots,B_n$ can be simultaneously diagonalized. In the latter case, we have $\vp'(x)=\diag(w,-w)$, where $w=\sum_{i=1}^n \lambda_i dz_i$, Hence we get $\det(\vp')=(\sqrt{-1}w)^2\in \MB_X$. 
	\end{proof}

	\begin{example}
		\label{e.RC}\label{e.CI-Projective-spaces}
		In the following we collect various interesting examples to illustrate the possibilities of $\MA_X$, $\MB_X$ and $\MS_X$.
		\begin{enumerate}
			\item Let $X$ be a rationally connected projective manifold. Then $H^0(X,\Sym^k\Omega_X^1)=0$ for any $k\geq 1$ by \cite[IV, Corollary 3.8]{Kollar1996}. Hence we have $\MA_X=\MB_X=\MS_X=0$.
			
			\item Let $X$ be a smooth projective variety with $c_1(X)=0$ in $H^2(X,\mathbb{Q})$, and $\pi _1(X)$ finite. It is proved by Kobayashi in \cite{Kobayashi1980a} as a direct consequence of Yau's theorem that $\MB_X=\MA_X=0$. 
			
			\item Let $X\subset \mbP^N$ be a smooth complete intersection such that either $\dim(X)\geq 3$ or $N=3$. Then we have $H^0(X,\Sym^2\Omega_X^1)=0$ by \cite{Brueckmann1986} and hence $\MB_X=0$. Moreover, by the Lefschetz hyperplane theorem, the variety $X$ is simply connected, so we have $\MA_X=0$. However, if $X\subset \mbP^4$ is a smooth complete intersection defined as following
			\[
			X:=\{x_0^d+\cdots+x_4^d=a_0 x_0^d+\cdots a_4 x_4^d=0\},
			\] 
			where $d\geq 5$ and $a_i$'s are general, then we have $H^0(X,\Sym^k\Omega_X^1)\not=0$ for any $k\geq 2$ (see \cite[Theorem 10]{Brueckmann1986} and \cite[\S\,3]{Brotbek2016}).
			
			\item  Let $X$ is the quotient of $\mbH\times \mbH$, the product of two copies of the complex upper half plane, by an irreducible discrete cocompact torsion free group $\Gamma\subset \Aut(\mbH)\times \Aut(\mbH)$ of holomorphic automorphisms defined by certain quaternion algebra, see \cite{Shavel1978} for more details. Such $X$ is called a \emph{quaternionic Shimura surface}. Then the product structure of $\mbH\times \mbH$ induces a splitting $\Omega_X^1\cong \MSL_1\oplus \MSL_2$. For any positive integer $n$, by \cite[Example 9.3]{Brunella2015}, we have
			\[
			h^0(X,\Omega_X^1) = h^0(X,\MSL^{\otimes n}_1) = h^0(X,\MSL^{\otimes n}_2) = 0.
			\]
			On the other hand, note that we have 
			\[
			\Sym^2\Omega_X^1\cong \MSL_1^{2}\oplus \MSL_1\otimes \MSL_2\oplus \MSL_2^{2}.
			\]
			It yields
			\[
			H^0(X,\Sym^2\Omega_X^1) = H^0(X,\MSL_1\otimes \MSL_2) = H^0(X,K_X).
			\]
			As a consequence, if we choose $\Gamma$ such that $p_g=h^0(X,K_X)\not=0$ (see for instance \cite[Example 2 and 3]{Shavel1978}), then we have $\MA_X\not=0$ but $\MS_X=\MB_X=0$.
			
			\item \label{example_simply_connected_nonvanishing_rankone_differential} Let $A$ be a three dimensional abelian variety. Let $Y$ be a smooth hypersurface section of $A$ of sufficiently high degree, which is invariant under the natural involution $\sigma=-\id$ and passes through one and only one of the fixed point $p$ of $\sigma$. The $\sigma$ induces a $\ZT$ action on $Y$ with fixed point $p$ and the quotient $Y/\ZT$ has an $A_1$ singularity at $p$. Let $X$ be the minimal resolution of $Y/\ZT$. It is shown in \cite[Page 1093]{bogomolov2011symmetric} that $X$ is simply connected and $\MB_X\neq 0$. 
		\end{enumerate}
		
	\end{example}
	
	The following result says that the spaces $\MA_X$, $\MB_X$ and $\MS_X$ are actually birational invariants of $X$ (see \cite[Theorem 5.3]{SongSun2021} for the general case).
	
	\begin{proposition}
		\label{p.birational-invariant-spectral-bases}
		Let $\pi:X'\rightarrow X$ be a birational map from a projective manifold $X'$ to $X$. Denote by $\pi^{\sharp}_{\MA}:\MA_{X}\rightarrow \MA_{X'}$ and $\pi^{\sharp}:H^0(X,\Sym^2\Omega_{X'}^1)\rightarrow H^0(X,\Sym^2\Omega_X^1)$ the natural injective linear map induced by the cotangent map of $\pi$. Then we have
		\begin{center}
			$\pi^{\sharp}_{\MA}(\MS_X)=\MS_{X'}$ and $\pi^{\sharp}(\MB_X)=\MB_{X'}$.
		\end{center}
	\end{proposition}
	
	\begin{proof}
		It is clear from the definition that $\pi^{\sharp}_{\MA}(\MS_X)\subset \MS_{X'}$ and $\pi^{\sharp}(\MB_X)\subset \MB_{X'}$. In the following we prove the reverse inclusion for $\MB_X$ and the proofs for $\MA_X$ and $\MS_X$ are similar. Since both $X$ and $X'$ are projective manifolds, there exists a non-empty Zariski open subset $U$ of $X$ such that $\textup{codim}(X\setminus U)\geq 2$ and $\pi^{-1}(U)\rightarrow U$ is an isomorphism. Consider the following composition
		\[
		H^0(X',\Sym^2\Omega_{X'}^1) \rightarrow H^0(\pi^{-1}(U),\Sym^2\Omega_{X'}^1|_{\pi^{-1}(U)}) \xrightarrow{\simeq} H^0(U,\Sym^2\Omega_U^1).
		\]
		It follows that every rank one symmetric differential $s_{X'}$ on $X'$ descends to a rank one symmetric differential $s_{U}$ on $U$. On the other hand, as $\textup{codim}(X\setminus U)\geq 2$ and $X$ is smooth, the restriction map 
		\[
		H^0(X,\Sym^2\Omega_X^1) \rightarrow H^0(U,\Sym^2\Omega_U^1)
		\]
		is an isomorphism. So $s_U$ extends uniquely to a symmetric differential $s_X$ on $X$, which is again rank one by the semi-continuity of the rank function. Finally, as $\pi^{\sharp}(s_X)=s_{X'}$ over the non-empty open subset $\pi^{-1}(U)$, we easily get $\pi^{\sharp}(s_X)=s_{X'}$ and we are done.
	\end{proof}
	
	\subsection{Rank one symmetric differentials}
	
	Let $0\not=s\in \MB_X$ be a rank one symmetric differential. If we regard it as an element of $H^0(\mbP(\Omega_X^1),\MSO_{\mbP(\Omega_X^1)}(m))$, then we can denote by $\Delta_s$ the effective Cartier divisor defined by $s$ on $\mbP(\Omega_X^1)$ (cf. \S\,\ref{s.projectivisation}). Recall that a divisor on $\mbP(\Omega^1_X)$ is called \emph{$\pi$-horizontal} if each of its irreducible components dominates $X$. Here comes the key observation as explained by Bogomolov and De Oliveira in \cite{bogomolov2011symmetric}.
	
	\begin{proposition}[\protect{\cite[p.1089]{bogomolov2011symmetric}}]
		\label{p.decomposition-symmetric-differetials}
		There exists a $\pi$-horizontal prime divisor $\Delta'_s$ such that $\Delta^h_s=2\Delta'_s$, where $\Delta^h_s$ is the $\pi$-horizontal part of $\Delta_s$. In particular, there exists a line bundle $\MSL$ on $X$ and non-zero elements 
		\begin{center}
			$\alpha\in H^0(X,\Omega_X^1\otimes \MSL^{-1})$ and $\tau \in H^0(X,\MSL^{2})$
		\end{center} 
	    such that
		\begin{enumerate}
			\item $\alpha$ does not vanish in codimension one;
			
			\item $s=\alpha^2\tau$.
		\end{enumerate}
		We will use $\textup{div}(s)$ to denote the divisor $\textup{div}(\tau)$ on $X$.
	\end{proposition}

	\begin{proof}
		Let $x\in X$ be an arbitrary point such that $s(x)\not=0$. Thanks to Proposition \ref{p.CN-rankone}, there exists $0\not=w\in \Omega^1_{X,x}$ such that $s(x)=w^2$. In particular, we have
		\[
		\Delta^h_s|_{\mbP(\Omega_{X,x}^1)} = \textup{div}(w^2) = 2\textup{div}(w),
		\]
		where $\textup{div}(w)$ is a hyperplane section of the projective space $\mbP(\Omega_{X,x}^1)$. Thus $\Delta_s^h=2\Delta'_s$ for some prime divisor on $\mbP(\Omega_X^1)$ such that the restriction $\Delta'_s|_{\mbP(\Omega_{X,x}^1)}$ is a hyperplane section for any point $x\in X$ such that $s(x)\not=0$.
		
		Let $\MSL$ be the unique line bundle on $X$ such that $\Delta'_s\in |\MSO_{\mbP(\Omega_X^1)}(1)\otimes \pi^*\MSL^{-1}|$ and denote by $\Delta^v_s=\Delta_s-\Delta^h_s$ the $\pi$-vertical part of $\Delta_s$. Since $\pi$ is a projective bundle, there exists an effective Cartier divisor $D$ on $X$ such that $\pi^*D=\Delta^v_s$. In particular, we obtain
		\[
		\pi^*D = \Delta_s^v = \Delta_s - \Delta^h_s \in |\MSO_{\mbP(\Omega_X^1)}(2)\otimes \MSO_{\mbP(\Omega_X^1)}(-2)\otimes \pi^*\MSL^{2}| = |\pi^*\MSL^{2}|.
		\]
		Hence we have $D\in |\MSL^{2}|$. Let $\tau\in H^0(X,\MSL^{2})$ be a section defining $D$. Then $\tau$ is unique up to non-zero scalars. Let 
		\[
		\alpha\in H^0(X,\Omega_X^1\otimes \MSL^{-1}) \xrightarrow{\simeq} H^0(\mbP(\Omega_X^1),\MSO_{\mbP(\Omega_X^1)}(1)\otimes \pi^*\MSL^{-1})
		\]
		be a global section defining $\Delta'_s$. Then we have $s=\alpha^2\tau$ up to non-zero scalars. Note that $\alpha$ does not vanish in codimension one because $\Delta_s^h$ does not contain $\pi$-vertical components.
	\end{proof}
	
	\subsection{Decomposition of $\MB_X$}
	
	Inspired by Proposition \ref{p.decomposition-symmetric-differetials} above, we can decompose the space $\MB_X$ into smaller spaces $\MB_{X,\MSL}$, which can easily be described. More precisely, given a given line bundle $\MSL$ on $X$, we can introduce the following space
	\[
	\MB_{X,\MSL}=\{s\in H^0(X,\Sym^2\Omega_X^1)\,|\,s=\alpha^2\tau,\,\alpha\in H^0(X,\Omega_X^1\otimes \MSL^{-1}),\, \tau\in H^0(X,\MSL^{2})\}.
	\]
	Then clearly we have 
	\begin{equation}
		\label{e.decomposition-BX}
		\MB_X=\bigcup_{\MSL\in \Pic(X)} \MB_{X,\MSL}
	\end{equation}
	Here note that \emph{a priori} we do not require that $\alpha$ does not vanish in codimension one. 
	
	\subsubsection{Reflexive sheaves and saturation}
	
	By Proposition \ref{p.decomposition-symmetric-differetials}, if $\MB_X\not=0$, then it is enough to only consider in \eqref{e.decomposition-BX} the line bundles $\MSL$ such that the general members in $H^0(X,\Omega_X^1\otimes \MSL^{-1})$ do not vanish in codimension one. This leads us to introduce the following notation.
	
	\begin{notation}
		\label{n.saturation-mu}
		Let $X$ be a projective manifold. Denote by $\MP(X)$ the subset of $\Pic(X)$ consisting of lines bundles $\MSL$ such that $H^0(X,\MSL^{2})\not=0$, $H^0(X,\Omega_X^1\otimes \MSL^{-1})\not=0$ and the general element in $H^0(X,\Omega_X^1\otimes \MSL^{-1})$ does not vanish in codimension one. By Proposition \ref{p.decomposition-symmetric-differetials}, $\MP(X)=\emptyset$ if and only if $\MB_X=0$.
	\end{notation}
	
	The non-vanishing condition above is actually equivalent to the saturation of $\MSL$ in $\Omega_X^1$. To be more precise, recall that a coherent sheaf $\MSF$ is called \emph{reflexive} if $\MSF^{**}=\MSF$, where $\msF^{**}$ is the dual double of $\msF$. It is known that a rank one reflexive sheaf over a projective manifold is actually locally free. Given a subsheaf $\MSF$ of a locally free sheaf $\MSE$, then there exists a unique subsheaf $\widetilde{\MSF}$ of $\MSE$ containing $\MSF$ such that $\MSE/\widetilde{\MSF}$ is torsion free and $\MSF=\widetilde{\MSF}$ over a non-empty Zariski open subset. We will call $\widetilde{\MSF}$ the \emph{saturation} of $\MSF$ in $\MSE$ and $\MSF$ is said to be \emph{saturated} in $\MSE$ if $\MSF=\widetilde{\MSF}$. Indeed, the saturation of $\MSF$ is just defined to be the kernel of the composition
	\[
	\MSE \rightarrow \MSE/\MSF \rightarrow (\MSE/\MSF)/\textup{torsion}.
	\]
	We remark that the saturation of $\MSF$ is always reflexive (see for instance \cite[Proposition 1.1]{Hartshorne1980}). The following result is elementary.
	
	\begin{lemma}
		\label{l.saturation-line-bundles}
		Let $X$ be a projective manifold. Let $\MSL$ be a line bundle and let $\MSE$ be a vector bundle over $X$ such that there exists a non-zero morphism $\alpha:\MSL\rightarrow \MSE$. If $\alpha$ does not vanish in codimension one, then the image $\alpha(\MSL)$ is saturated in $\MSE$. 
	\end{lemma}
	
	\begin{proof}
		Denote by $\MSQ$ the cokernel of $\alpha:\MSL \rightarrow \MSE$ and let $\widetilde{\MSL}$ be the saturation of $\alpha(\MSL)$. Since $\alpha$ does not vanish in codimension one, the coherent sheaf $\MSQ$ is locally free is codimension one. In particular, the natural inclusion $\MSL \hookrightarrow \widetilde{\MSL}$ is an isomorphism in codimension one. However, since both $\MSL$ and $\widetilde{\MSL}$ are line bundles, any non-zero morphism $\MSL \rightarrow \widetilde{\MSL}$ either is an isomorphism or vanishes along some divisors. So we must have $\MSL=\widetilde{\MSL}$. 
	\end{proof}
	
	Given an element $\alpha\in H^0(X,\Omega_X^1\otimes \MSL^{-1})$, we can also regard it as a map $\MSL\rightarrow \Omega_X^1$, which we will again denote by $\alpha$ by abuse of notation. Then $\MSL\in \MP(X)$ means that the image $\alpha(\MSL)$ is saturated in $\Omega_X^1$ for a general non-zero element $\alpha\in H^0(X,\Omega_X^1\otimes \MSL^{-1})$.

    \begin{remark}
        In the definition of stability of Higgs bundles (cf. Definition \ref{d.stability}), it is enough to consider only the \emph{saturated} $\vp$-invariant subsheaves of $\MSE$. In particular, if $\MSE$ is of rank two, then we only need to consider all \emph{rank one locally free saturated $\vp$-invariant subsheaves} of $\MSE$ to determine the stability of $\MSE$.
    \end{remark}
	
	\subsubsection{Structure of $\MB_X$}
	
	Note that we have always $\MB_{X,\MSL}\subset \MB_X$ for any line bundle $\MSL$. In particular, if $\MB_X\not=0$, then Proposition \ref{p.decomposition-symmetric-differetials} yields the following decomposition
	\begin{equation}
		\label{e.Decomposition-BX-simplified}
		\MB_X =\bigcup_{\MSL\in \MP(X)} \MB_{X,\MSL}.
	\end{equation}
	To give a detailed description of $\MB_{X,\MSL}$, we need the notion of the so-called \emph{Veronese cone}. Let $V$ be a complex vector space. A \emph{Veronese cone} contained in $V$ is the image of the following composition
	\[
	W\xrightarrow{\nu_2} \Sym^2 W \xrightarrow{\iota} V,
	\]
	where $W$ is a complex vector space, $\nu_2$ is the second Veronese map sending $w$ to $w^2$ and $\iota$ is an injective linear map. In other words, a Veronese cone is the affine cone over a Veronese variety $\nu_2(\mbP(W^*))\subset \mbP(\Sym^2 W^*)$. 
	
	We need the following Bogomolov--Castelnuovo--De Franchis theorem. For a line bundle $\MSL$ on a projective manifold $X$, we denote by $\kappa(X,\MSL)$ its Iitaka dimension.
	
	\begin{theorem}[\protect{\cite[Theorem 2 and Proposition 6]{Reid1977}}]
		\label{t.BCdFtheorem}
		Let $\MSL$ be a line bundle on a projective manifold $X$ such that $\kappa(X,\MSL)\geq 1$. Then there exists a fibration $f:X\rightarrow C$ onto a curve with genus $\geq 1$ such that if $\alpha:\MSL\rightarrow \Omega_X^1$ is a saturated injection, then we have inclusions $f^*\Omega_C^1\subset \alpha(\MSL)\subset \Omega_X^1$ and $\kappa(X,\MSL)=1$.
	\end{theorem}
	
	As an immediate application, we have the following simple description of $\MB_{X,\MSL}$.
	
	\begin{proposition}
		\label{p.structure-BXL}
		For a line bundle $\MSL\in \MP(X)$, If $\MB_{X,\MSL}\not=0$, then one of the following statements hold.
		\begin{enumerate}
			\item $h^0(X,\MSL^{2})=1$ and $\MB_{X,\MSL}$ is a Veronese cone.
			
			\item $h^0(X,\Omega_X^1\otimes \MSL^{-1})=1$ and $\MB_{X,\MSL}$ is a linear subspace.
		\end{enumerate}
	\end{proposition}
	
	\begin{proof}
		Firstly we assume that $h^0(X,\MSL^{2})=1$. Fix a non-zero element $\tau \in H^0(X,\MSL^{2})$. Then it is easy to see that $\MB_{X,\MSL}$ is isomorphic to the Veronese cone defined as following
		\[
		H^0(X,\Omega_X^1\otimes \MSL^{-1}) \xrightarrow{\nu_2} H^0(X,\Sym^2\Omega_X^1\otimes \MSL^{-2}) \xrightarrow{\times \tau} H^0(X,\Sym^2\Omega_X^1).
		\]
		
		Now we assume that $h^0(X,\MSL^{2})\geq 2$. In particular, we have $\kappa(X,\MSL)\geq 1$. Choose a general non-zero element $\alpha\in H^0(X,\Omega^1_X\otimes \MSL^{-1})$ and denote by $\alpha:\MSL \hookrightarrow \Omega_X^1$ the inclusion induced by $\alpha$. As $\MSL\in \MP(X)$, we may assume that $\alpha$ does not vanish in codimension one and hence the image $\alpha(\MSL)$ is saturated in $\Omega_X^1$ by Lemma \ref{l.saturation-line-bundles}. Thanks to Theorem \ref{t.BCdFtheorem}, there exists a fibration $f:X\rightarrow C$ such that the image $\alpha(\MSL)\subset \Omega_X^1$ is nothing but the saturation of the natural inclusion $f^*\Omega_C^1\subset \Omega_X^1$. Given another non-zero element $\alpha'\in H^0(X,\Omega^1_X\otimes \MSL)$, denote by $\MSL'$ the saturation of the induced inclusion $\alpha':\MSL\hookrightarrow \Omega_X^1$. Applying Theorem \ref{t.BCdFtheorem} again yields that we have $\MSL'=\alpha(\MSL)\cong \MSL$. In particular, we have a natural non-zero morphism
		\[
		\MSL\xrightarrow{\alpha'}\alpha'(\MSL) \rightarrow \MSL'=\alpha(\MSL)\cong \MSL,
		\]
		which must be an isomorphism. Hence, we have $\alpha'(\MSL)=\alpha(\MSL)$, i.e., $\alpha=c\alpha'$ for some $c\in \CC^*$. As a consequence we obtain $h^0(X,\Omega_X\otimes \MSL^{-1})=1$ and $\MB_{X,\MSL}$ is the linear subspace obtained as
		\[
		H^0(X,\MSL^{2}) \xrightarrow{\times \alpha^2} H^0(X,\Sym^2\Omega_X^1),
		\]
		where $\alpha$ is any given non-zero element in $H^0(X,\Omega_X^1\otimes \MSL^{-1})$.
	\end{proof}
	
	The proposition above naturally leads us to introduce the following terminologies.
	
	\begin{definition}
		The space $\MB_{X,\MSL}$ is called of V-type (resp. of L-type) if $\MB_{X,\MSL}$ is a Veronese cone (resp. a linear subspace).
	\end{definition}

		According to the proof of Proposition \ref{p.structure-BXL} above, it is easy to see that a space $\MB_{X,\MSL}$ is of both V-type and L-type if and only if $\dim(\MB_{X,\MSL})=1$ and if and only if 
		\[
		h^0(X,\MSL^{2})= h^0(X,\Omega_X^1\otimes \MSL^{-1})=1.
		\]

    \begin{example}
    	\label{e.V-L-examples}
    	In the following we give two examples of projective manifolds $X$ with $\MB_{X}$ of V-type and L-type, respectively.
    	\begin{enumerate}
    		\item Let $X$ be an $n$-dimensional abelian variety. Then $\Omega_X^1\cong \MSO_X^{\oplus n}$ and $\Sym^2\Omega_X^1\cong \MSO_X^{\oplus n(n+1)/2}$. It follows that $\MB_X=\MB_{X,\MSO_X}$ is of V-type, which is exactly the image of the second Veronese map
    		\[
    		\mbC^n\cong H^0(X,\Omega_X^1) \rightarrow H^0(X,\Sym^2\Omega_X^1)=\Sym^2\mbC^n = \mbC^{n(n+1)/2}.
    		\]
    		
    		\item Let $X=C\times F$ be the product of a smooth projective curve $C$ with genus $g\geq 1$ and a rationally connected variety $F$. Set $\MSL=p^*K_C$ and $\MSF=q^*\Omega_F^1$, where $p:X\rightarrow C$ and $q:X\rightarrow F$ are the natural projections. Then we have
    		\begin{center}
    			$\Omega_X^1\cong \MSL\oplus \MSF$ and $\Sym^2\Omega_X^1\cong \MSL^{2}\oplus \MSF\otimes \MSL\oplus \Sym^2\MSF$.
    		\end{center}
    	    Since $F$ is rationally connected, we have $\MA_F=0$ (cf. Example \ref{e.RC} (1)). Then restricting $\Sym^2\Omega_X^1$ to fibres of $p$ implies
    	    \[
    	    H^0(X,\Sym^2\Omega_X^1) = H^0(X,\MSL^{2}) \cong H^0(C,K_C^{2}).
    	    \]
    	    In particular, the space $\MB_X=\MB_{X,\MSL} = H^0(X,\Sym^2\Omega_X^1)$ is of L-type.
    	\end{enumerate}
    \end{example}

    In the following we study the possible inclusion relations between different spaces $\MB_{X,\MSL}$ with different $\MSL$'s.
	
	\begin{lemma}
		\label{l.BXL-intersection-relation}
		Let $\MSL$ and $\MSL'$ be two line bundles over a projective manifold $X$ such that both $\MB_{X,\MSL}$ and $\MB_{X,\MSL'}$ are non-zero. Then the following statements hold.
		\begin{enumerate}
			\item If both $\MB_{X,\MSL}$ and $\MB_{X,\MSL'}$ are of L-type such that $\MB_{X,\MSL}\cap \MB_{X,\MSL'}\not=0$, then $\MSL\cong \MSL'$. In particular, we have $\MB_{X,\MSL}=\MB_{X,\MSL'}$.
			
			\item If the space $\MB_{X,\MSL}$ is of L-type and $\MB_{X,\MSL'}\subset \MB_{X,\MSL}$, then $\MSL\cong \MSL'$. In particular, we have $\MB_{X,\MSL}=\MB_{X,\MSL}$.
			
			\item If the space $\MB_{X,\MSL}$ is of V-type and $\MB_{X,\MSL'}\subset \MB_{X,\MSL}$, then $h^0(X,\MSL'\otimes \MSL^{-1})\not=0$ and the space $\MB_{X,\MSL'}$ is again of V-type .
		\end{enumerate}
	\end{lemma}
	
	\begin{proof}
		For (1), let us fix a non-zero element $s=\alpha^2\tau=\alpha'^2\tau'\in \MB_{X,\MSL}\cap \MB_{X,\MSL'}$ with 
		\begin{center}
			$\alpha\in H^0(X,\Omega_X^1\otimes \MSL^{-1})$ and $\alpha'\in H^0(X,\Omega_X^1\otimes \MSL'^{-1})$.
		\end{center}
		Since $\MB_{X,\MSL}$ and $\MB_{X,\MSL'}$ are of L-type, we must have 
		\[
		h^0(X,\Omega_X^1\otimes \MSL^{-1}) = h^0(X,\Omega_X^1\otimes \MSL'^{-1})=1.
		\]
		In particular, as $\MSL, \MSL'\in \MP(X)$, the sections $\alpha$ and $\alpha'$ do not vanish in codimension one. So applying Proposition \ref{p.decomposition-symmetric-differetials} to $s$ shows that $\MSL\cong \MSL'$ and thus $\MB_{X,\MSL}=\MB_{X,\MSL'}$.
		
		\bigskip
		
		For (2), let us fix a general non-zero element $s=\alpha^2\tau=\alpha'^2\tau'\in \MB_{X,\MSL'}\subset \MB_{X,\MSL}$ as in the proof of (1). Then $\alpha'$ does not vanish in codimension one as $\MSL'\in \MP(X)$. In particular, the image $\alpha'(\MSL')\cong \MSL'$ is the saturation of $\alpha(\MSL)\cong \MSL$ in $\Omega_X^1$. Nevertheless, since $\MB_{X,\MSL}$ is of L-type, we have $h^0(X,\Omega_X^1\otimes \MSL^{-1})=1$ and thus the image $\alpha(\MSL)$ is also saturated in $\Omega_X^1$ as $\MSL\in \MP(X)$. It follows that $\alpha(\MSL)=\alpha'(\MSL')$ and hence $\MSL\cong \MSL'$.
		
		\bigskip
		
		For (3),  let us fix a general non-zero element $s=\alpha^2\tau=\alpha'^2\tau'\in \MB_{X,\MSL'}\subset \MB_{X,\MSL}$ as in the proof of (1). Similar to the proof of (2), the image $\alpha'(\MSL')$ is actually the saturation of $\alpha(\MSL)$ in $\Omega_X^1$, which induces a natural inclusion $\MSL \rightarrow \MSL'$, i.e., $H^0(X,\MSL'\otimes \MSL^{-1})\not=0$. Finally, to see $\MB_{X,\MSL'}$ is of V-type, it is enough to realise that a linear subspace contained in a Veronese cone must be one dimensional.
	\end{proof}

    \begin{remark}
    	It may happen that $\MB_{X,\MSL}$ and $\MB_{X,\MSL'}$ are of L-type and V-type, respectively, such that $\MB_{X,\MSL}\cap \MB_{X,\MSL'}\not=0$, see for instance (4) and (5) in \S\,\ref{example_products_of_curves} below.
    \end{remark}
	
	\subsubsection{Example: products of two curves}
	\label{example_products_of_curves}
	Now we discuss the space $\MB_X$ for $X$ being a product of two curves. More precisely, let $X=C_1\times C_2$ be the product of two smooth projective curve. Denote by $g_i$ the genus of the curve $C_i$ and assume $g_1\geq g_2$. Then we have a natural splitting $\Omega_X^1=\MSL_1\oplus \MSL_2$, where $\MSL_i=p_i^*\Omega_{C_i}^1$ and $p_i:X\rightarrow C_i$ is the natural projection. 
	
	Let $\MSL$ be a line bundle on $X$ such that $\MB_{X,\MSL}\not=0$. For an element $\alpha\in H^0(X,\Omega_X^1\otimes \MSL^{-1})$ which does not vanish in codimension one, if the composition $\MSL \xrightarrow{\alpha} \Omega^1_{X} \rightarrow \MSL_i$ induced by $\alpha$ is non-zero, then the restriction 
	\begin{equation}
		\label{e.restriction-fibres}
		\MSL|_{F_{i}}\rightarrow \MSL_i|_{F_i} \cong \MSO_{F_i}
	\end{equation}
	is non-zero, where $F_i$ is a general fibre of $p_i$. However, as $H^0(X,\MSL^{2})\not=0$, we must have $c_1(\MSL)\cdot F_i\geq 0$. So \eqref{e.restriction-fibres} implies $c_1(\MSL)\cdot F_i=0$ and the map $\MSL|_{F_i}\rightarrow \MSO_{F_i}$ is actually an isomorphism. In particular, there exists an effective divisor $D_i$ on $C_i$ such that 
	\[
	\MSL\cong \MSL_i\otimes \MSO_X(-p_i^* D_i).
	\]
	As a consequence, either $\MSL\cong \MSO_X$ or the composition $\MSL\rightarrow \Omega_X^1\rightarrow \MSL_i$ is zero for some $1\leq i\leq 2$. In the former case, we have $\MB_{X,\MSL}=\MB_{X,\MSO_X}$, which is of V-type. In the latter case, the image $\alpha(\MSL)$ is contained in $\MSL_{j}$, where $1\leq j\not=i\leq 2$. Then we get $\alpha(\MSL)=\MSL_j$ since $\alpha(\MSL)$ is saturated in $\Omega_X^1$ by assumption and therefore $\MB_{X,\MSL}$ is of L-type. In conclusion, we have
	\[
	\MB_X = \MB_{X,\MSO_X} \cup \MB_{X,\MSL_1} \cup \MB_{X,\MSL_2}.
	\]
	According to the different values of $g_i$'s, we have the following detailed description of $\MB_X$.
	\begin{enumerate}
		\item $g_1=g_2=0$. Then $X$ is rationally connected and hence $\MB_{X}=0$.
		
		\item $g_1=1$ and $g_2=0$. Then $\MSL_1\cong \MSO_X$, $\MB_{X,\MSL_2}=0$, and $\MB_{X}=\MB_{X,\MSO_X}$. In particular, we get
		\begin{center}
			$h^0(X,\Sym^2\Omega^1_X)=\dim(\MB_X)=1$ and $\MB_X=H^0(X,\Sym^2\Omega_X^1)$.
		\end{center} 
		
		\item $g_1=g_2=1$. Then we have $\MSL_1\cong \MSL_2\cong \MSO_X$ and $\MB_X=\MB_{X,\MSO_X}$ is the Veronese cone defined as 
		\[
		\mathbb{C}^2\cong H^0(X,\Omega_X^1) \xrightarrow{\nu_2} H^0(X,\Sym^2\Omega_X^1)\cong \mathbb{C}^3.
		\]
		
		\item \label{(4)}$g_1>g_2=1$. Then $\MSL_2\cong \MSO_X$ and $\MB_X=\MB_{X,\MSO_X}\cup \MB_{X,\MSL_1}$ such that
		\[
		\MB_{X,\MSO_X}\cap \MB_{X,\MSL_1}=\{s=\alpha^2_1\tau^2\,|\,\tau\in H^0(X,\MSL_1)\},
		\]
		where $\alpha_1\in H^0(X,\Omega_X^1\otimes \MSL_1^{-1})$ is any non-zero element corresponding to the natural inclusion $\MSL_1\subset \Omega_X^1$. On the other hand, Lemma \ref{l.BXL-intersection-relation} implies that $\MB_{X,\MSO_X}\not\subseteq \MB_{X,\MSL_1}$ as $\MSL_1\not\cong \MSO_X$ and if $\MB_{X,\MSO_X}\subset \MB_{X,\MSL_1}$, then $H^0(X,\MSL_1^{-1})\not=0$, which is absurd as $g_1>1$.
		
		\item $g_1\geq g_2\geq 2$. Then we have $\MB_X=\MB_{X,\MSO_X}\cup \MB_{X,\MSL_1}\cup \MB_{X,\MSL_2}$. Let $\alpha_i\in H^0(X,\Omega_X^1\otimes \MSL_i^{-1})$ be any non-zero element corresponding to the natural inclusion $\MSL_i\subset \Omega_X^1$. By Lemma \ref{l.BXL-intersection-relation}, the same argument as in \eqref{(4)} yields
		\begin{center}
			$\MB_{X,\MSL_1}\cap \MB_{X,\MSL_2}=0$ and $\MB_{X,\MSO_X}\cap \MB_{X,\MSL_i}=\{\alpha_i^2\tau^2\,|\,\tau\in H^0(X,\MSL_i)\}$, $1\leq i\leq 2$.
		\end{center}
		Moreover, by the Riemann-Roch Theorem, one can easily get
		\[
		\dim(\MB_{X,\MSO_X})=h^0(X,\Omega_X^1)=h^0(C_1,\omega_{C_1})+h^0(C_2,\omega_{C_2})=g_1+g_2
		\]
		and
		\[
		\dim(\MB_{X,\MSL_i})=h^0(X,\MSL_i^{2}) = h^0(C_i,\omega_{C_i}^{2})=3g_i-3,\quad \forall\,1\leq i\leq 2.
		\]
	\end{enumerate}
	
	\subsection{Codimension one foliations}
	
	Now we discuss the relation between rank one symmetric differentials and codimension one foliations. For more details about foliations, we refer the reader to \cite{Brunella2015}.
	\begin{definition}
		Let $X$ be a projective manifold. A foliation on $X$ is a subsheaf $\MSF$ of $T_X$ such that
		\begin{enumerate}
			\item the quotient $T_X/\MSF$ is torsion free, and
			
			\item $\MSF$ is closed under the Lie bracket, i.e., $[\MSF,\MSF]\subset \MSF$.
		\end{enumerate}
		The dimension of a foliation $\MSF$ is defined to the rank of $\MSF$ at a general point of $X$ and $\MSF$ is said to be regular at a point $x\in X$ if $T_X/\MSF$ is locally free at $x$.
	\end{definition}
	
	Let $s\in H^0(X,\Sym^2\Omega_X^1)$ be a non-zero rank one symmetric differential. Let $\MSL\xrightarrow{\alpha} \Omega_X^1$ be the rank one subsheaf provided in Proposition \ref{p.decomposition-symmetric-differetials}. Let $\MSF_s\subset T_X$ be the annihilator subsheaf of $s$, i.e. the co-rank one subsheaf of $T_X$ satisfying the following exact sequence
	\begin{equation}
		\label{e.tangent-sequence-foliation}
		0\rightarrow \MSF_s \rightarrow T_X \xrightarrow{\alpha^*} \MSL^*.
	\end{equation}
	Note that $T_X/\MSF_s=\alpha^*(T_X)$ is a subsheaf of the torsion free sheaf $\MSL^*$, so the quotient $T_X/\MSF_s$ is again torsion free. On the other hand, as $h^0(X,\MSL^{2})\not=1$, \cite[Main Theorem]{Demailly2002} says that $\MSF_s$ is closed under the Lie bracket and thus $\MSF_s$ defines a codimension one foliation on $X$. Usually we call $\MSL$ the \emph{conormal bundle} of $\MSF_s$. Thus $\MSF_s$ is an example of codimension one foliations with pseudo-effective normal bundle and the existence of such foliations yields rather strong restriction on the base manifold $X$ (see for instance \cite{Touzet2016} and the references therein). 
	
	In the following example we give a complete description of $\MB_X$ for $X$ being a surface with non positive Kodaira dimension.
	
	\begin{example}
		Let $X$ be a smooth projective surface with $\MP(X)\not=\emptyset$. Then $X$ is not rationally connected by Example \ref{e.RC} above. Choose a line bundle $\MSL\in \MP(X)$. Let $s=\alpha^2\tau\in \MB_{X,\MSL}$ be a general element and denote by $\MF_s$ the associated foliation on $X$.
		\begin{enumerate}
			\item Assume that $X$ is uniruled, i.e. $X$ is dominated by rational curves. Then there exists a fibration $f:X\rightarrow C$ to a smooth projective curve $C$ such that the general fibre $F$ of $f$ is isomorphic to $\mbP^1$ and $g(C)\geq 0$. As the conormal bundle $\MSL$ is pseudoeffective, we have $c_1(\MSL)\cdot F\geq 0$. Then \eqref{e.tangent-sequence-foliation} shows that $F$ is tangent to $\MF_s$ and consequently $\alpha(\MSL)$ is the saturation of $f^*\Omega_C^1\subset \Omega_X^1$. Thus $h^0(X,\Omega_X^1\otimes \MSL^{-1})=1$ and $\MB_X=\MB_{X,\MSL}$ is of L-type.
			
			\item Assume that $X$ is of Kodaira dimension zero. Then there exists a birational map $\pi:X\rightarrow X'$ to a smooth projective surface $X'$ such that $c_1(X')=0$. As $\MB_X$ is a birational invariant \ref{p.birational-invariant-spectral-bases}, after replacing $X$ by $X'$, we may assume $c_1(X)=0$. By the classification of surfaces \cite[Theorem VIII.2]{Beauville1996}, the surface $X$ is isomorphic one of the following: an Enriques surface; a K3 surface; an abelian surface; a bielliptic surface.
			
			\begin{enumerate}
				\item If $X$ is an Enriques surface or a K3 surface, then we have $H^0(X,\Sym^2\Omega_X^1)=0$ and hence $\MB_X=0$. Indeed, if $X$ is a Enriques surface, then there exists an \'etale double cover $\widetilde{X}\rightarrow X$, which is a K3 surface \cite[Proposition VIII.17]{Beauville1996}. Nevertheless, we have $H^0(X,\Sym^k\Omega_X^1)=0$ for any K3 surface $X$ and any integer $k\geq 1$ (see Example \ref{e.RC} (2)).
				
				\item If $X$ is an abelian surface, then $\MB_X=\MB_{X,\MSO_X}$ is the two dimensional Veronese cone by Example \ref{e.V-L-examples}.
				
				\item If $X$ is a bielliptic surface, then $X\cong (E\times F)/G$, where $E$, $F$ are elliptic curves and $G$ is a finite group of translations of $E$ acting on $F$ such that $F/G\cong \mbP^1$ \cite[Definition VI.19]{Beauville1996}. Moreover, we have $\Omega_X^1\cong \MSO_X\oplus K_X$. Here $K_X$ is the canonical bundle of $X$, which is a torsion in $\Pic(X)$. As $\MSL$ is pseudoeffective and $\alpha(\MSL)$ is saturated in $\Omega_X^1$, then one of the following compositions
				\begin{center}
					$\MSL\xrightarrow{\alpha} \Omega_X^1 \rightarrow \MSO_X$
					and
					$
					\MSL \xrightarrow{\alpha} \Omega_{X}^1 \rightarrow K_X
					$
				\end{center}
				is isomorphism. As $K_X\not\cong \MSO_X$, we obtain $h^0(X,\Omega_X^1\otimes \MSL^{-1})=1$ and either $\MSL\cong \MSO_X$ or $\MSL\cong K_X$. For $\MSL\cong \MSO_X$, then $\dim(\MB_{X,\MSO_X})=1$ and the foliation $\MSF_s$ defined by any $0\not=s\in \MB_{X,\MSO_X}$ is induced by the natural fibration $X\rightarrow E/G$.
				
				For $\MSL\cong K_X$, then we must have $K_X^{2}\cong \MSO_X$ as $H^0(X,\MSL^{2})\not=0$. In particular, by the list \cite[List VI.20]{Beauville1996}, the group $G$ is isomorphic to either $\mbZ/2\mbZ$ or $\mbZ/2\mbZ\oplus \mbZ/2\mbZ$. We again have $\dim(\MB_{X,\MSL})=1$ and the foliation $\MSF_s$ defined by any $s\in \MB_{X,\omega_X}$ is induced by the natural fibration $X\rightarrow F/G\cong \mbP^1$. So we have proved that
				\begin{enumerate}
					\item the space $\MB_X=\MB_{X,\MSO_X}$ is an one dimensional linear subspace if $K_X^{2}\not\cong \MSO_X$;
					
					\item the space $\MB_X=\MB_{X,\MSO_X}\cup \MB_{X,\omega_X}$ is a union of two one dimensional linear subspaces if $K_X^{2}\cong \MSO_X$. 
				\end{enumerate}
			\end{enumerate}
		\end{enumerate}
	\end{example}
	
In all the examples above, we note that the decomposition \eqref{e.Decomposition-BX-simplified}, i.e., $\MB_X=\cup_{\MSL\in \MP(X)} \MB_{X,\MSL}$, is actually finite. Thus it is natural to ask the following question.

\begin{question}
	Let $X$ be a projective manifold. Is the decomposition \eqref{e.Decomposition-BX-simplified} always finite?
\end{question}	
	
	\section{Spectral varieties and their Cohen-Macaulayfications}
 \label{sec_spectralvarietyandCM}
	
	This section is devoted to study the spectral variety defined by a non-zero $s\in \MB_X$. More precisely, given a projective manifold $X$ and a non-zero $s\in \MB_X$, we defined the \emph{spectral variety} $X_s$ associated to $s$ to be
	\[
	X_s:=\{\lambda\in \tot(\Omega_X^1)\,|\,\lambda^2+s=0 \}.
	\]
	Then the natural map $p_s:X_s\rightarrow X$ is a double cover. However, in general the spectral variety $X_s$ may be very singular (e.g. reducible and non-normal, etc.). In \cite[Theorem 7.1]{ChenNgo2020}, Chen and Ng\^o showed that if $X$ is a surface, then there exists a Cohen-Macaulayfication $\widetilde{X}_s\rightarrow X_s$ such that the induced morphism $\widetilde{X}_s\rightarrow X$ is a flat finite morphism and then they use it to derive a spectral correspondence \cite[Theorem 7.3]{ChenNgo2020}. The main goal of this section is to generalise this to arbitrary dimension in rank two case.
	
	\subsection{Cohen-Macaulayness}
	
	In this section we collect some basic facts about Cohen-Macaulay modules and we refer the reader to \cite{Serre1965,BrunsHerzog1993} for a more detailed discussion. 
	
	Let $R$ be a ring and let $M$ be an $R$-module. The set $\{\mathfrak{p}\in \textup{Spec}(R)\,|\,M_{\mathfrak{p}}\not=0\}$ is called the \emph{support} of $M$, and written $\textup{supp}(M)$. Define $\dim(M)=\dim(\textup{supp}(R))$. A sequence $a_1,\dots,a_r$ of elements of $R$ is called a \emph{regular sequence} for $M$ if $a_1$ is not a zero divisor in $M$, and for all $i=2,\dots,r$, $a_i$ is not a zero divisor in $M/(a_1,\dots,a_{i-1})M$. If $R$ is a local ring with maximal ideal $\mathfrak{m}$, then the \emph{depth} of $M$ is the maximum length of a regular sequence $a_1,\dots,a_r$ for $M$ with all $a_i\in \mathfrak{m}$. 
	
	\begin{definition}
		Let $R$ be a Noetherian local ring. A finite $R$-module $M\not=0$ is a Cohen-Macaulay module if $\textup{depth}(M)=\dim(M)$. A maximal Cohen-Macaulay module is a Cohen-Macaulay module $M$ such that $\dim(M)=\dim(R)$.
		
		In general, if $R$ is an arbitrary Noetherian ring, then $M$ is a Cohen-Macaulay module if $M_{\mathfrak{m}}$ is a Cohen-Macaulay module for all maximal ideals $\mathfrak{m}\in \textup{Supp}(M)$. However, for $M$ to be a maximal Cohen-Macaulay module, we require that $M_{\mathfrak{m}}$ is such an $R_{\mathfrak{m}}$-module for each maximal ideal $\mathfrak{m}$ of $R$.
		
		A Noetherian ring $R$ is called Cohen-Macaulay if it is Cohen-Macaulay as a $R$-module.
	\end{definition}
	
	We need the following useful fact.
	
	\begin{proposition}
		\label{p.Serre-CM}
		\cite[IV, D, Corollaire 2]{Serre1965}
		Let $k$ be a field and $R$ a regular $k$-algebra of pure dimension $k$. Let $M$ be a finitely generated $R$-module. Then $M$ is Cohen-Macaulay of dimension $k$ if and only if it is locally free over $R$.
	\end{proposition}
	
	\subsection{Cohen-Macaulayfications of spectral varieties}
	
	\label{ss.Cohen-Macaulayfication}
	
	Let $X$ be a projective manifold such that $\MB_X\not=0$. Fix $0\not=s\in \MB_X$ such that $s=\alpha^2\tau$ as in Proposition \ref{p.decomposition-symmetric-differetials}. 
	
	\subsubsection{Spectral variety}
	Given $0\not= s\in \MB_X$, the inclusion $\alpha:\MSL\rightarrow \Omega_X^1$ induces a map $\alpha^{\sharp}:\tot(\MSL) \rightarrow \tot(\Omega_X^1)$ satisfying the following commutative diagram
	\[
	\begin{tikzcd}[column sep=large,row sep=large]
		\tot(\MSL) \arrow[rr,"\alpha^{\sharp}"] \arrow[dr,"p" below]
		& 
		& \tot(\Omega_X^1) \arrow[dl,"q"] \\
		& X &
	\end{tikzcd}
	\]
	and let $Z$ be the zero set of $s$. Then $Z$ is a closed subset of $X$ with $\textup{codim}(Z)\geq 2$ such that $\alpha^{\sharp}(x)$ is an injective linear map for any $x\in X\setminus Z$. We define a double cover $\widetilde{X}_s$ of $X$ as following
	\[
	\widetilde{X}_s:=\{\eta\in \tot(\MSL)\,|\,\eta^2+\tau=0\}.
	\]
	Denote by $\widetilde{p}_s:\widetilde{X}_s\rightarrow X$ the induced covering map. Recall that a variety $X$ is called \emph{Cohen-Macaulay} if its local ring $\MSO_{X,x}$ is Cohen-Macaulay for any point $x\in X$.
	
	\begin{proposition}
		The variety $\widetilde{X}_s$ is Cohen-Macaulay and satisfies the following properties.
		\begin{enumerate}
			\item $\alpha^{\sharp}(\widetilde{X}_s)=X_s$.
			
			\item There exists an open Zariski open subset $U$ of $X$ such that $\textup{codim}(X\setminus U)\geq 2$ and  the restricted map $\alpha^{\sharp}:\widetilde{p}^{-1}_s(U)\rightarrow p_s^{-1}(U)$ is an isomorphism.
			
			\item The morphism $\widetilde{p}_s:\widetilde{X}_s\rightarrow X$ is a flat double covering.
		\end{enumerate}
	\end{proposition}
	
	\begin{proof}
		Since $\tot(\MSL)$ is smooth and $\widetilde{X}_s$ is a codimension one complete intersection in $\tot(\MSL)$, thus the Cohen-Macaulayness of $\widetilde{X}_s$ follows from the fact that every complete intersection ring is Cohen-Macaulay \cite[Definition 2.3.1]{BrunsHerzog1993}.
		
		The statement (1) follows directly from the definitions of $X_s$ and $\widetilde{X}_s$. For (2), we may choose $U=X\setminus Z$, where $Z$ is the set of zeros of $s$. Then $\textup{codim}(Z)\geq 2$ and the statement follows from the fact that $\alpha^{\sharp}$ is an embedding over $p^{-1}(U)$. The statement (3) follows from the miracle flatness since $\widetilde{X}_s$ is Cohen-Macaulay and $X$ is smooth.
	\end{proof}
	
	\subsubsection{Tower of Cohen-Macaulayfications}
 \label{s.towerofCMness}
	
	Similar construction can be applied to any effective divisor $D'\leq D=\textup{div}(\tau)$ to get other Cohen-Macaulafications of $X_{s}$. To be more precise, let us write 
	\[
	D=\sum_{i=1}^r m_i D_i
	\]
	with $D_i$ pairwise distinct prime divisors, and up to permutation we may assume that $m_i\geq 2$ for $1\leq i\leq k$ and $m_i=1$ for $k+1\leq i\leq r$. Choose $\tau_i\in H^0(X,\MSO_X(D_i))$ such that $\textup{div}(\tau_i)=D_i$. Then up to scalars we have
	\[
	\tau = \tau_1^{m_1}\dots\tau_{k}^{m_k} \tau_{k+1}\dots\tau_r \in H^0(X,\MSL^{2}).
	\]
	For a $k$-tuple $\textbf{a}=(a_1,\dots,a_k)$ of non-negative integer such that $2a_i\leq m_i$ for $1\leq i\leq k$, we denote by $X_{\textbf{a}}\subset \tot(\MSL_{\textbf{a}})$ the double cover of $X$ corresponding to the section 
	\[
	\tau_{\textbf{a}}=\tau_1^{m_1-2a_1}\dots\tau_{k}^{m_k-2a_k}\tau_{k+1}\cdots\tau_r \in H^0(X,\MSL_{\textbf{a}}^{2}),
	\]
	where $\MSL_\textbf{a}=\MSL\otimes \MSO_X(-a_1D_1-\dots-a_rD_r)$. Given two $k$-tuple $\textbf{a}$ and $\textbf{a}'$ such that $\textbf{a}'\geq \textbf{a}$, i.e. $a'_i\geq a_i$ for any $1\leq i\leq k$, then there exists a natural finite birational morphism $X_{\textbf{a}'} \rightarrow X_{\textbf{a}}$ induced by the following map defined by multiplying $\tau_1^{a'_1-a_1}\cdots\tau_k^{a'_k-a_k}$
	\[
	\tau_{\textbf{a}'-\textbf{a}}^{\sharp}: \tot(\MSL_{\textbf{a}'}) \rightarrow \tot(\MSL_{\textbf{a}}).
	\]
	Then $X_{\textbf{a}}$ is Cohen-Macaulay and the induced morphism $X_{\textbf{a}}\rightarrow X_{s}$ is a finite birtional morphism. All the double covers $X_{\textbf{a}}$ fit into the following tower of commutative diagrams
	\smallskip
	\adjustbox{scale=0.75,left}{%
		\begin{tikzcd}[row sep=large,column sep=large]
			& X_{(\lfloor{\frac{m_1}{2}}\rfloor-1,\lfloor{\frac{m_2}{2}}\rfloor,\cdots,\lfloor{\frac{m_k}{2}}\rfloor)}  \arrow[ddr, bend left]  
			& 
			& X_{(0,\cdots,0,1)} \arrow[ddr, bend left]
			& \\
			& X_{(\lfloor{\frac{m_1}{2}}\rfloor,\lfloor{\frac{m_2}{2}}\rfloor-1,\lfloor{\frac{m_3}{2}}\rfloor\cdots,\lfloor{\frac{m_k}{2}}\rfloor)} \arrow{dr}
			& 
			& X_{(0,\cdots,0,1,0)} \arrow[dr]
			& \\
			\widetilde{X}_s^n=X_{(\lfloor{\frac{m_1}{2}}\rfloor,\cdots,\lfloor{\frac{m_k}{2}}\rfloor)} \arrow[uur, bend left] \arrow[ur] \arrow[r] \arrow[dr] \arrow[ddr, bend right]
			& \vdots \arrow[r]
			& \vdots \arrow[uur, bend left] \arrow[ur] \arrow[r] \arrow[dr] \arrow[ddr, bend right]
			& \vdots \arrow[r]
			& X_{(0,\dots,0)}=\widetilde{X}_s \\
			& X_{(\lfloor{\frac{m_1}{2}}\rfloor,\cdots,\lfloor{\frac{m_{k-2}}{2}}\rfloor,\lfloor{\frac{m_{k-1}}{2}}\rfloor-1,\lfloor{\frac{m_k}{2}}\rfloor)} \arrow[ur]
			& 
			& X_{(0,1,0,\cdots,0)} \arrow[ur]
			&  \\
			& X_{(\lfloor{\frac{m_1}{2}}\rfloor,\cdots,\lfloor{\frac{m_{k-1}}{2}}\rfloor,\lfloor{\frac{m_k}{2}}\rfloor)} \arrow[uur, bend right]
			&
			& X_{(1,0,\cdots,0)} \arrow[uur, bend right]
			& 
		\end{tikzcd}
	}
	\smallskip
	By Serre's criterion for normality \cite[Chapitre IV, D, Theorem 11]{Serre1965}, a Cohen-Macaulay variety is normal if and only if its singular locus has codimension at least two. In particular, in the diagram above, the only normal variety is $\widetilde{X}_s^n$ and the induced morphism $\widetilde{X}^n_s\rightarrow X_s$ is exactly the normalisation of $X_s$.
	
	\subsubsection{Numerical invariants of double cover}
	
	In the following we collect some basic facts about numerical invariants of push-forward under double coverngs. Let $\pi:X'\rightarrow X$ be a smooth double covering corresponding to $\Delta\in |H^0(X,\MSL^{2})|$. Let $\MSM$ be a line bundle on $X'$. Then the push-forward $\pi_*\MSM$ is a rank $2$ vector bundle on $X$. Let $\iota:X'\rightarrow X'$ be the involution corresponding to $\pi$. Note that $\MSM\otimes \iota^*\MSM$ is invariant under $\iota$, the push-forward $\pi_*(\MSM\otimes \iota^*\MSM)$ carries a natural involution induced by $\iota$ and hence splits as a direct sum of line bundles corresponding to taking $+1$ and $-1$ eigenvalues of $\iota$. Thus we have 
	\[
	\pi_*(\MSM\oplus \iota^*\MSM) \cong (\pi_*(\MSM\otimes \iota^*\MSM))^{\iota}\oplus \MSN,
	\]
	where the first factor is the $\iota$-invariant part. Then the \emph{norm line bundle} $\textup{Nm}(\MSM)$ of $\MSM$ is defined as $(\pi_*(\MSM\otimes \iota^*\MSM))^{\iota}$. 
	
	\begin{proposition}\label{p.chern_class_push_down}
		\cite[p.\,47-p.\,49]{Friedman1998}
		Notations and assumptions as above and let $D$ be any Cartier divisor on $X'$ such that $\MSO_{X'}(D)\cong \MSM$, then we have:
		\begin{enumerate}
			\item $\textup{Nm}(\MSM)\cong \MSO_X(\pi_*D)$;
			
			\item $\pi^*(\textup{Nm}(\MSM))\cong \MSM\otimes \iota^*\MSM$;
			
			\item $\det(\pi_*\MSM)\cong \textup{Nm}(\MSM)\otimes \MSL^{-1}$;
			
			\item $c_1(\pi_*\MSM)=\pi_*c_1(\MSM)-c_1(\MSL)$;
			
			\item $c_2(\pi_*\MSM)=\frac{1}{2}\left((\pi_*c_1(\MSM))^2 - \pi_*(c_1^2(\MSM))-(\pi_*c_1(\MSM))\cdot c_1(\MSL)\right)$;
			
			\item an exact sequence of vector bundles
			\[
			0\rightarrow \iota^*\MSM\otimes \MSL^{-1} \rightarrow \pi^*\pi_*\MSM \rightarrow \MSM \rightarrow 0.
			\]
		\end{enumerate}
	\end{proposition}
	
	\subsection{Spectral correspondence}
	
	The importance of the Cohen-Macaulayfication $\widetilde{X}_s$ of $X_s$ is that there exists a correspondence between Higgs bundles on $X$ with spectral date $s$ and maximal Cohen-Macaulay sheaves with generic rank one on $\widetilde{X}_s$ (cf. \cite[Proposition 7.3]{ChenNgo2020} and \cite[Proposition 3.6]{BeauvilleNarasimhanRamanan1989}). More precisely, given a maximal Cohen-Macaulay sheaf of generic rank one $\MSM$ on $\widetilde{X}_s$, we denote by $\MSE$ the push-forward $\widetilde{p}_{s*}\MSM$. Then $\MSE$ is locally free with rank two by Proposition \ref{p.Serre-CM}. We can associate a natural morphism $\varphi:\MSE\rightarrow \MSE\otimes \Omega_X^1$ via the tautological section $\eta\in H^0(\widetilde{X}_s,\widetilde{p}_s^*\MSL)$ as following
	\[
	\MSE=\widetilde{p}_{s*}(\MSM) \xrightarrow{\times \eta} \widetilde{p}_{s*}(\MSM\otimes \widetilde{p}^*_s\MSL) \rightarrow \widetilde{p}_{s*}\MSM\otimes \MSL=\MSE\otimes \MSL \xrightarrow{\times \alpha} \MSE\otimes \Omega_X^1.
	\]
	
	\begin{example}
		\label{e.canonical-Higgs-bundle}
		Let $\MSM$ be the structure sheaf $\MSO_{\widetilde{X}_s}$ of $\widetilde{X}_s$. Then we have $\MSE\cong \MSO_X\oplus \MSL^{-1}$. A local computation shows that the homomorphism $\varphi:\MSE\rightarrow \MSE\otimes \Omega^1_X$ is given as following
		\[
		\begin{pmatrix}
			0    &     -\alpha\tau \\
			\alpha &  0
		\end{pmatrix}
		\]
		In particular, we have $\varphi\wedge\varphi=0$, $\Tr(\varphi)=0$ and $\det(\varphi)=s$. By Proposition \ref{prop_stability_realHiggs} and Propoisiton \ref{p.Properties-of-Hitchin-section}, the Higgs bundle $(\MSE,\varphi)$ is always polystable. 

	\end{example}
	
    Recall that a \emph{$\MSL$-twisted Higgs bundle} on a projective manifold $X$ is a pair $(\MSE,\varphi)$ consisting of a rank two vector bundle $\MSE$ and $\varphi:\MSE\rightarrow \MSE\otimes \MSL$ a twisted endomorphism, where $\MSL$ is a given line bundle over $X$.
 
	\begin{lemma}
		\label{l.factorisation-varphi}
		Let $(\MSE,\varphi)$ be a rank two Higgs bundle on a projective manifold $X$ such that $\Tr(\varphi)=0$ and $\det(\varphi)=s\not=0$. Then $\varphi$ factors through 
		\[
		\MSE\otimes \MSL \xrightarrow{\times \alpha} \MSE\otimes \Omega^1_X.
		\]
        In other words, the pair $(\MSE,\varphi)$ is actually a $\MSL$-twisted Higgs bundle.
	\end{lemma}
	
	\begin{proof}
		Let $Z\subset X$ be the zero locus of $s$. Then the homomorphism $\varphi$  can be simultaneously diagonalised as $\diag(w,-w)$ at any point $x\in X\setminus Z$ (see the proof of Proposition \ref{p.CN-rankone}). In particular, we have $w^2=s=\alpha^2(x)\tau(x)$. As $\tau(x)\not=0$, the homomorphism $\varphi$ factors through $\MSE\otimes \MSL\xrightarrow{\times \alpha} \MSE\otimes \Omega^1_X$ outside $Z$. Consider the following composition
		\begin{equation}
			\label{e.factorisation-varphi}
			\MSE \xrightarrow{\varphi} \MSE\otimes \Omega^1_X \rightarrow \MSE\otimes (\Omega^1_X/\alpha(\MSL)).
		\end{equation}
		Note that $\alpha(\MSL)$ is saturated in $\Omega_X^1$ and thus the quotient $\Omega_X^1/\alpha(\MSL)$ is torsion free. In particular, as $\MSE$ is locally free, the tensor $\MSE\otimes (\Omega_X^1/\alpha(\MSL))$ is again torsion free. Hence, the composition \eqref{e.factorisation-varphi} vanishes identically on $X$ as it vanishes over $X\setminus Z$. As a consequence, the homomorphism $\varphi$ factors through $\MSE\otimes \MSL\xrightarrow{\times \alpha} \MSE\otimes \Omega^1_X$.
	\end{proof}
	
	\begin{theorem}
		\label{t.Spectral-correspondence}
		Let $X$ be a projective manifold and let $0\not=s\in \MB_X$. Then there is a bijective correspondence between isomorphism classes of maximal Cohen-Macaulay sheaves of generic rank one on $\widetilde{X}_s$ and isomorphism classes of rank two Higgs bundles $(\MSE,\varphi)$ with $\Tr(\varphi)=0$ and $\det(\varphi)=s$. 
	\end{theorem}
	
	\begin{proof}
		Let $Z$ be the zero set of $s$ and set $U=X\setminus Z$. Then $\widetilde{U}=\widetilde{p}_s^{-1}(U)$ is smooth. Let $\MSM$ be a maximal Cohen-Macaulay sheaf of generic rank one over $\widetilde{X}_s$. Then $\MSM$ is actually locally free of rank one over $\widetilde{U}$ by Proposition \ref{p.Serre-CM}. Let $\MSE=\widetilde{p}_{s*}\MSM$ be the rank two vector bundle over $X$ with the homomorphism $\varphi:\MSE\rightarrow \MSE\otimes \Omega^1_X$ as that introduced before Example \ref{e.canonical-Higgs-bundle}. Then Example \ref{e.canonical-Higgs-bundle} implies $\varphi\wedge\varphi=0$, $\Tr(\varphi)=0$ and $\det(\varphi)=s$ over the open subset $U$ and hence over the whole $X$.
		
		Conversely, let $(\MSE,\varphi)$ be a rank two Higgs bundle on $X$ with $\Tr(\varphi)=0$ and $\det(\varphi)=s$. By Lemma \ref{l.factorisation-varphi}, the homomorphism $\varphi$ factors through $\MSE\otimes \MSL\xrightarrow{\times \alpha} \MSE\otimes \Omega_X^1$. Thus we may regard $\varphi$ as a map $\MSO_X\rightarrow \End(\MSE)\otimes\MSL$ and then the Cayley-Hamilton theorem says $\varphi^2+\tau=0$. It follows that the algebra homomorphism
		\[
		\Sym(\MSL^{-1}) \rightarrow \End(\MSE)
		\]
		induced by $\varphi$ factors through the quotient
		\[
		\Sym(\MSL^{-1})\rightarrow \Sym(\MSL^{-1})/\MSI,
		\]
		where $\MSI$ is the ideal sheaf generated by the image of $\id+\tau:\MSL^{-2}\rightarrow \Sym(\MSL^{-1})$. Note that the variety $\widetilde{X}_s$ is just $\textup{Spec}(\Sym(\MSL^{-1})/\MSI)$, we obtain $\widetilde{p}_{s*}\MSO_{\widetilde{X}_s}=\Sym(\MSL^{-1})/\MSI$ and then $\MSE$ admits a natural $\widetilde{p}_{s*}\MSO_{\widetilde{X}_s}$-module structure such that $\eta$ acts as $\varphi$ on $\MSE$, where $\eta$ is the generator of $\MSL^{-1}$. Finally, since $\MSE$ is locally free with rank two, the $\MSO_{\widetilde{X}_s}$-module $\MSM$ is a maximal Cohen-Macaulay sheaf with generic rank one and the homomorphism $\varphi$ is exactly induced by multiplying $\eta$.
	\end{proof}
	
	\begin{remark}
		As indicated in \S\,\ref{ss.Cohen-Macaulayfication}, there may exist many other Cohen-Macaulayficiations $\widetilde{X}_{\textbf{a}}$ of $X_s$ and the same argument as in the proof of Theorem \ref{t.Spectral-correspondence} shows that the maximal Cohen-Macaulay sheaves with generic rank one on $\widetilde{X}_{\textbf{a}}$ also defines a Higgs bundle $(\MSE,\varphi)$ on $X$ with $\varphi\wedge\varphi=0$, $\Tr(\varphi)=0$ and $\det(\varphi)=s$. However, if $\widetilde{X}_{\textbf{a}}\not=\widetilde{X}_s$, then there exist Higgs bundles $(\MSE,\varphi)$ on $X$ which can not be defined in this way, e.g. the Higgs bundle $(\widetilde{p}_{s*}\MSO_{\widetilde{X}_s},\varphi)$ defined in Example \ref{e.canonical-Higgs-bundle}. The key point is that the image of $\MSL_{\textbf{a}}$ in $\Omega_X^1$ is not saturated and Lemma \ref{l.factorisation-varphi} does not hold for $\MSL_{\textbf{a}}$.
	\end{remark}

	\subsection{$\GL_2(\mbC)$ Higgs bundles}
	
	Now we proceed to consider $\GL_2(\mbC)$ Higgs bundles on $X$. Given a spectral datum $\mbfs=(s_1,s_2)\in \MS_X$, by Proposition \ref{p.CN-rankone}, we have 
	\[
	s:=s_2-\frac{1}{4}s_1^2\in \MB_X.
	\]
	Moreover, we can define the \emph{spectral variety} $X_{\mbfs}\subset \tot(\Omega_X^1)$ associated to $\textbf{s}$ as 
	\[
	X_{\mbfs}:=\{\lambda \in \tot(\Omega_X^1)\,|\,\lambda^2-s_1\lambda+s_2=(\lambda-\frac{1}{2}s_1)^2+s=0\}.
	\]
	Then one can easily see that if $s\not=0$, the morphism 
	\[
	s_1^{\sharp}:\tot(\Omega_X^1) \rightarrow \tot(\Omega_X^1), \quad \lambda\mapsto \lambda-\frac{1}{2}s_1
	\]
	induces an isomorphism $X_{\mbfs} \rightarrow X_s$. In particular, the proof of Theorem \ref{t.spectral-correspondence-GL2} can be completed by applying Theorem \ref{t.Spectral-correspondence} to $X_s$.

	\begin{proof}[Proof of Theorem \ref{t.spectral-correspondence-GL2}]
		Given a maximal Cohen-Macaulay sheaf $\MSM$ of generic rank one on $\widetilde{X}_s$, applying Theorem \ref{t.Spectral-correspondence} to $\MSM$ yields a rank two Higgs bundle $(\MSE,\varphi)$ on $X$ with $\varphi\wedge\varphi=0$, $\Tr(\varphi)=0$ and $\det(\varphi)=s$. Then we can define a new Higgs field on $\MSE$ as $\varphi'=\varphi+\frac{1}{2}s_1$. 
		
		Conversely, given a Higgs bundle $(\MSE,\varphi)$ with $\Tr(\varphi)=s_1$ and $\det(\varphi)=s_2$, then we define a new Higgs field on $\MSE$ as $\varphi'=\varphi-\frac{1}{2}s_1$. Applying Theorem \ref{t.Spectral-correspondence} to $(\MSE,\varphi')$ shows that there exists a maximal Cohen-Macaulay sheaf $\MSM$ on $\widetilde{X}_s$ such that $\MSE\cong \widetilde{p}_{s*}\MSM$ and $\varphi'$ is induced by multiplying $\eta$ such that $\varphi=\varphi'+\frac{1}{2}s_1$.
	\end{proof}
	
	\begin{proof}[Proof of Corollary \ref{c.Chen-Ngo-GL2}]
		If $s=s_2-s_1^2/4=0$, we can define a polystable rank two Higgs bundle $(\MSE,\varphi)$ on $X$ as
		\begin{center}
			$\MSE\cong \MSO_X\oplus \MSO_X$ and $\varphi=\diag(s_1/2,s_1/2)$.
		\end{center}
		If $s=s_2-s_1^2/4\not=0$, the statement follows from Theorem \ref{t.spectral-correspondence-GL2} and Example \ref{e.canonical-Higgs-bundle}.
	\end{proof}

    \begin{proof}[Proof of Theorem \ref{t.LtwistedHiggs}]
        For any $\textbf{s}=(s_1,s_2)\in \MS_X\setminus \MS_X^{\nil}$, we set $s=s_2-s_1^2/4$. Then we have $s\not=0$. Applying Lemma \ref{l.factorisation-varphi} yields that there exists a line bundle $\MSL$ and an inclusion $\alpha:\MSL\rightarrow \Omega_X^1$ such that for any Higgs bundle $(\MSE,\varphi')$ with $\Tr(\varphi')=0$ and $\det(\varphi')=s$, there exists a $\MSL$-twisted Higgs field $\varphi_0\in H^0(X,\End(\MSE)\otimes \MSL)$ such that $\varphi'=\alpha\circ \varphi_0$. For any arbitrary Higgs bundle $(\MSE,\varphi)$ such that $\sh_X([\MSE,\varphi])=\textbf{s}$, we can define a new Higgs field $\varphi'$ as $\varphi-\frac{1}{2}\Id_{\MSE}$. Then we conclude by applying the result above to $(\MSE,\varphi')$.
    \end{proof}
	
	Under the normality assumption on the spectral variety $X_s$, we have the following description of the stack $\MM(X_s)$ of maximal Cohen-Macaulay sheaves of rank one on $X_s$, see also \cite[Proposition 7.6]{ChenNgo2020}.
	
	\begin{proposition}
		Assume that the spectral variety $X_s$ is irreducible and normal. Then the action of the Picard scheme $\MP(X_s)$ of line bundles on itself by translation extends to a free action on the moduli stack $\MM(X_s)$.
	\end{proposition}
	
	\begin{proof}
		As $X_s$ is Cohen-Macaulay, every line bundle on $X$ is Cohen-Macaulay, i.e., $\MP(X_s)\subset \MM(X_s)$. Next we can define the action of $\MP(X_s)$ on $\MM(X_s)$ by
		\[
		(\MSL,\MSF) \mapsto \MSF\otimes \MSL
		\]
		for $\MSF\in \MM(X_s)$ and $\MSL\in \MP(X_s)$, since the tensor product $\MSF\otimes \MSL$ is again a maximal Cohen-Macaulay sheaf of rank one on $X$.
		
		Finally, as $X$ is normal, there exists a closed subset $Z$ of codimension at least two such that $X\setminus Z$ is smooth. If a line bundle $\MSL\in \MP(X_s)$ has a stabliser $\MSF\in \MM(X_s)$, then we have $\MSL\otimes \MSF\cong \MSF$. However, as $\MSF|_{X\setminus Z}$ is locally free, it follows that $\MSL|_{X\setminus Z}$ is isomorphic to the trivial bundle $\MSO_{X}|_{X\setminus Z}$. Since $X$ is normal and $\MSL$ is locally free, it follows that $\MSL$ is actually isomorphic to $\MSO_X$. Hence, the action of $\MP(X_s)$ on $\MM(X_s)$ is free.
	\end{proof}
	
	\begin{remark}
		In general it may happen that $\widetilde{X}_s$ is not normal for any $s\in \MB_X$, see Example \ref{e.Spectral-var} (3) below.
	\end{remark}

	\begin{example}
		\label{e.Spectral-var}
		We discuss in the following the geometry of $\widetilde{X}_s$ for several classes of projective manifolds $X$.
		
		\begin{enumerate}
			\item Let $X$ be a projective manifold such that $\MB_{X,\MSO_X}\not=0$. Let $s\in \MB_{x,\MSO_X}$ be a general element. Then the variety $\widetilde{X}_s$ is a disjoint union of two copies of $X$ and the rank two Higgs bundles $(\MSE,\varphi)$ on $X$ with the associated spectral datum $\Tr(\varphi)=0$ and $\det(\varphi)=s$ is of the following form
			\begin{center}
				$\MSE\cong \MSL_1\oplus \MSL_2$ and $\varphi=\diag(\sqrt{-1}\alpha,-\sqrt{-1}\alpha)$,
			\end{center}
			where $\MSL_i$'s are line bundles on $X$ and $\alpha\in H^0(X,\Omega_X^1)$ such that $\alpha^2=s$.
			
			\item Let $X$ be an $n$-dimensional abelian variety. Then $\MB_X=\MB_{X,\MSO_X}$ is of V-type by Example \ref{e.V-L-examples}. In particular, we have $s=\omega^2$ for any $0\not=s\in \MB_X$ and the spectral variety $X_{s}$ is a disjoint union of two copies of $X$. Thus the moduli stack of maximal Cohen-Macaulay sheaves of generic rank one on $X_{s}$ is just $\Pic(X)\oplus \Pic(X)$.

			More generally, let $\pi:X'\rightarrow X$ be a birational map from a projective manifold $X'$ to $X$. Then every element $s'\in \MB_{X'}$ can be written as $\omega'^2$ for some $\omega' \in H^0(X,\Omega_{X'}^1)$ such that $\omega'=\pi^*\omega$ and $s'=\pi^*\omega^2$. In particular, note that a general $s'\in \MB_{X'}$ does not vanish in codimension one. It follows that the Cohen-Macaulayfication $\widetilde{X}'_{s'}$ of the spectral variety $X'_{s'}$ is again a disjoint union of two copies of $X'$ and the moduli stack of maximal Cohen-Macaulay sheaves of generic rank one on $\widetilde{X}'_{s'}$ is $\Pic(X')\oplus \Pic(X')$.

			\item Let $f:X\rightarrow C$ be a fibration over a smooth projective curve $C$. Let $\MSL$ be the saturation of the inclusion $f^*\Omega_C^1\subset \Omega_X^1$. Then we have
			\begin{equation}
				\label{e.Conormal-foliations}
				\MSL \cong f^*\Omega_C^1\otimes \MSO_X(R_f),
			\end{equation}
			where $R$ is an effective divisor defined as 
			\[
			R_f=\sum_{x\in C} (F_x-(F_x)_{\textup{red}}),
			\]
			where $F_x$ is the pull-back $f^*x$ as Cartier divisor and $(F_x)_{\textup{red}}$ is the sum of irreducible component of $F_x$. In particular, the summation in the definition of $R_f$ is actually finite.
			
			Let $\pi:X'\rightarrow X$ be a birational map from a projective manifold $X'$ to $X$ and denote by $f':X'\rightarrow C$ the induced morphism. Let $\MSL'$ be the saturation of $f'^*\Omega_C^1$ in $\Omega_{X'}^1$. Then \eqref{e.Conormal-foliations} implies that there exists an effective $\pi$-exceptional divisor $E_{\pi}$ such that 
			\[
			\MSL'\cong \pi^*\MSL + E_{\pi},
			\] 
			where $E_{\pi}=R_{f'}-\pi^*R_f$. Then we have $|\MSL'^{2}|=\pi^*|\MSL^{2}| + 2E_{\pi}$.
			Let $E$ be an irreducible $\pi$-exceptional divisor and let $x=f'(E)\in C$. Write $F_x=\sum_i m_i F_i$ into irreducible components with $F_i$'s distinct prime divisors. Then a straightforward computation shows
			\[
			\textup{Mult}_E R_{f'}=\textup{Mult}_E (\pi^*F_x) -1=\sum_{i} m_i \textup{Mult}_E (\pi^*F_i)-1
			\]
			and
			\[
			\pi^*R_f = \sum_i (m_i-1) \textup{Mult}_E(\pi^* F_i).
			\]
			This yields
			\[
			\textup{Mult}_E E_{\pi} = \sum_i \textup{Mult}_{E}(\pi^*F_i) - 1. 
			\]
			
			Let $X$ be the blow-up of $F\times C$ at a point $z$, where $F$ is rationally connected. Let $F'$ be the unique reducible fibre of $f$ and let $p$ be the singular point of $F'$. Let $\pi:X'\rightarrow X$ be the blow-up at $p$. Then we have 
			\[
			\MB_{X'}=\MB_{X',\MSL'} = \{s=\alpha^2\tau\,|\,\tau\in H^0(X',\MSL'^{2})\},
			\]
			where $\alpha\in H^0(X',\Omega_{X'}^1\otimes \MSL'^{-1})$ corresponds to the inclusion $\MSL'\subset \Omega_{X'}^1$. In particular, our argument above shows that the vanishing order of $s$ and hence $\tau$ along $E_{\pi}$ is at least two for any $s\in \MB_X$, where $E_{\pi}$ is the unique exceptional divisor of $\pi$. This implies that $\widetilde{X}_s$ is not normal for any $0\not=s\in \MB_X$. Moreover, if $C$ is an elliptic curve, then $\widetilde{X}_s$ is even reducible for any $0\not=s\in \MB_X$.
		\end{enumerate}
	\end{example}
	
\section{Real Higgs bundles and the Hitchin section}
\label{sec_realHiggsbundle_Hitchin_section}
 In this section, we will delve into the $\CS$ action on the moduli space of Higgs bundles and explore the notion of real Higgs bundles. For a more detailed understanding, we refer to \cite[Chapter 4]{Simpson1992} and \cite[Chapter 9,10]{hitchin1987self}. Additionally, we will construct a section of the spectral morphism, extending the construction presented in \cite{hitchin1992lie}.
 
	\subsection{$\CS$ action on the space of Higgs bundles}
	Over the moduli space of Higgs bundle $\MMH$, there is a canonical $\CS$ action given as the follows: given a Higgs bundle $(\msE,\vp)\in \MMH$, for $t\in \CS$, $t$ acts on $(\msE,\vp)$ is defined as $t\cdot(\msE,\vp)=(\msE,t\vp)$. This action preserves the stability condition, the Chern classes of the Higgs bundles, and in particular, the condition of being topologically trivial. The fixed points of the $\CS$ action can be explicitly described as follows.
	
	\begin{proposition}[\protect{\cite[Lemma 4.1]{Simpson1992}}]
		\label{prop_Hodge_bundle_stablity}
		Consider a rank two Higgs bundle $(\msE,\vp\neq 0)\in \MMH$ which is a fixed point of the $\mbC^{*}$ action. Then there exist holomorphic line bundles, $\msL_1$ and $\msL_2$, along with a non-zero section $\al\in H^0(X,\msL_1^{-1}\otimes\msL_2\otimes \Omega_X^1)$ such that we could write 
        \[
        \msE=\msL_1\oplus \msL_2,\;\vp=\begin{pmatrix}
			0 & 0\\
			\al & 0
		\end{pmatrix}.
        \]
        Moreover, the Higgs bundle $(\msE,\vp)$ is stable if and only if $\deg \msL_1>\deg \msL_2$. 
	\end{proposition}
	\begin{proof}
		
If $(\msE,\vp)$ is a $\CS$ fixed point, then there exists $t\neq \pm 1\in \CS$ such that $(\msE,t\vp)\cong (\msE,\vp)$. Then there exists a holomorphic gauge transformation $g$ such that $g\cdot (\msE,t\vp)=(\msE,\vp)$. Since the characteristic polynomial of $g$ is a holomorphic function on $X$, the eigenvalues are constants. For any $\lambda\in \mbC$, denote by $E_{\lambda}$ the kernel of $(g-\lambda)^2$. Now $(g-t\lambda)^2\varphi=t^2\varphi(g-\lambda)^2$, so $\varphi$ maps the $\lambda$ eigenspace $E_{\lambda}$ to $E_{t\lambda}$. In particular, as $\varphi\not=0$ and $t\not=1$, we can assume that $g$ has distinct non-vanishing eigenvalues $\lam_1\neq \lam_2$, which gives a decomposition $\msE\cong \msL_1\oplus \msL_2$ into eigenbundles. Moreover, as $(g-t\lam)\vp=t\vp(g-\lam)$, $\vp$ maps an eigensection $s$ with eigenvalue $\lam$ to an eigensection with eigenvalue $t\lam$. As $\vp\neq 0$, we can assume $\lam_2=t\lam_1$, and the Higgs field has nontrivial action on $\msL_1$. If the restriction $\vp:\msL_2\to \msL_1$ is also non-trivial, then we would have $\lam_1=t\lam_2$, which contradicts the assumption that $t\neq \pm 1$. Therefore, $\vp|_{\msL_2}=0$, and we can write $\vp=\begin{pmatrix}
0 & 0\\
\al & 0
\end{pmatrix}$.
		
For stability, let $d_i=\deg(\msL_i)$. The slope of $\msE$ is given by $\mu(\msE)=\frac{1}{2}(d_1+d_2)$. If $(\msE,\vp)$ is stable, then since $\msL_2\subset \msE$ is a $\vp$-invariant subline bundle, we have $\mu(\msE)>\mu(\msL_2)$, which implies $d_1>d_2$.

Conversely, if $d_1>d_2$ and $(\msE,\vp)$ is not stable, then there exists a $\vp$-invariant rank one sheaf $\msL$ with $\deg(\msL)>\mu(\msE)=\frac12(d_1+d_2)>d_2$. However, as $\msL$ is $\vp$-invariant, the composition $\msL\to\msE\to\msL_2$ is nontrivial. In particular, we have $\deg(\msL)\leq d_2$, which leads to a contradiction.
	\end{proof}
	
	\begin{definition}
		A rank two Higgs bundle $(\msE,\vp)$ is called a Hodge bundle if we could write 
		\begin{equation}
			\label{eq_Hodge_bundle}
			\msE=\msL_1\oplus \msL_2,\;\vp=\begin{pmatrix}
				0 & 0\\
				\al & 0
			\end{pmatrix}
		\end{equation}
		and a Hodge bundle is called a VHS if $(\msE,\vp)$ is stable and $\Delta(\msE)\cdot [\omega]^{n-2}=0$. 
	\end{definition}

In particular, a Higgs bundle $(\msE,\vp)\in\MMD$ is a variation of Hodge structure (VHS for short) if and only if it is a fixed point of the $\CS$-action. The concept of the VHS implies that under the non-abelian Hodge correspondence, the corresponding complex connection defines a flat connection, which gives rise to a non-trivial representation arising from the complex variation of Hodge structure \cite{Simpson1988Construction, Simpson1992}.
	
	\begin{theorem}[\protect{\cite{Simpson1992}\cite[Corollary 6.12]{simpson1994moduli2}}]
		\label{action_fixed_point_Hodge_bundle}
		Let $(\msE,\vp)$ be a topologically trivial stable rank two Higgs bundle, then the limit $[(\msE_0,\vp_0)]:= \lim_{t\to 0}[(\msE,t\vp)]\in \MMD$ exists, which is a fixed point of $\CS$ action.
	\end{theorem}
	
\subsection{The moduli space of real Higgs bundles}
	In this subsection, we will introduce real Higgs bundles and define the moduli space of real Higgs bundles.
 
\subsubsection{Real Higgs bundles}

The concept of real Higgs bundles is derived from \cite[Chapter 10]{hitchin1987self} and further explored in \cite[Chapter 2]{Simpson1992}. Consider a rank two Higgs bundle $(\msE,\vp)$ with $\mathrm{Tr}(\vp)=0$ such that $\Delta(\msE)\cdot \omega^{n-2}=0$. Let $\rho: \pi_1(X)\to \mathrm{PSL}_2(\mathbb{C})$ be the corresponding representation under the non-abelian Hodge correspondence. Then, for the Higgs bundle $(\msE,-\vp)$, the corresponding representation is denoted by $\bar{\rho}$. This leads to the following definition of a real Higgs bundle.

\begin{definition}
A Higgs bundle $(\msE,\vp)$ is called real if it is complex gauge equivalent to $(\msE,-\vp)$.
\end{definition}
	
	Rank two real Higgs bundles have the following canonical shape:
	\begin{proposition}[\protect{\cite[Proposition 10.2]{hitchin1987self}}]
		\label{prop_stability_realHiggs}
		Let $(\msE,\vp)$ be a real rank two Higgs bundle. Then there exists a decomposition
\begin{equation}
\label{eq_shape_real_Higgs_bundle}
\begin{split}
\msE=\msL_1\oplus \msL_2,\;\vp=\begin{pmatrix}
\om & \beta \\
\alpha & \om
\end{pmatrix}
\end{split}
\end{equation}
where $\om\in H^0(X,\Omega_X^1)$, $\alpha\in H^0(X,\msL_1^{-1}\otimes \msL_2\otimes \Omega_X^1)$ and $\beta\in H^0(X,\msL_1\otimes \msL_2^{-1}\otimes \Omega_X^1)$ such that $\alpha\wedge\beta=0$. Regarding the stability condition, the following statements hold.
\begin{enumerate}
\item Suppose $\deg(\msL_1)>\deg(\msL_2)$. Then $(\msE,\vp)$ is stable if and only if $\alpha\neq 0$.
\item Suppose $\deg(\msL_1)=\deg(\msL_2)$. Then we have:
\begin{enumerate}
\item If $\msL_1$ is not isomorphic to $\msL_2$, then $(\msE,\vp)$ is stable if and only if $\alpha\neq 0$ and $\beta\neq 0$.
\item If $\msL_1\cong \msL_2$, then $\alpha,\beta\in H^0(X,\Omega_X^1)$. Moreover, the Higgs bundle $(\msE,\vp)$ is stable if and only if $\alpha$ and $\beta$ are not proportional. Otherwise $(\msE,\vp)$ is polystable.
\end{enumerate}
\end{enumerate}
	\end{proposition}
	\begin{proof}
		Given a real Higgs bundle, then there exists $g\in \Aut(\msE)$ such that $g\cdot\bar{\pa}_{\msE}=\bar{\pa}_{\msE}$ and $g^{-1}\vp g=-\vp$. As in the proof of Proposition \ref{prop_Hodge_bundle_stablity}, the holomorphic gauge transformation $g$ has two distinct non-zero eigenvalues. So $\msE$ splits and  we could write $\msE=\msL_1\oplus \msL_2$ and $g=\diag(a,b)$. Moreover, we could write 
        \[
        \vp=\begin{pmatrix}
			\omega_1 & \be\\
			\al & \omega_2
		\end{pmatrix}.
        \]
        Then the equality $g^{-1}\vp g=-\vp$ implies that $\omega_1=\omega_2$. The condition $\vp\we\vp=0$ is equivalent to $\al\we \be=0$. 
  
        For the stablity condition, it is sufficient to assume that $\vp$ is trace-free. The the statement (1) follows directly from Proposition \ref{prop_Hodge_bundle_stablity}. For (2), set $d:=\deg(\msL_1)=\deg(\msL_2)$. Then $\msE$ is a semi-stable vector bundle. Let $\msL\subset \msE$ be a saturated rank one subsheaf with $\deg(\msL)\geq d=\deg(\MSL_i)$. Then it follows that if the composition $\msL\to \msE\to \msL_i$ is non-zero, we must have $\msL\cong\msL_i$ and $\deg(\msL)=\deg(\msL_i)$. 
		
		For case (a) of (2), since $\MSL_1$ is not isomorphism to $\MSL_2$, the argument above show that the subsheaves $\msL_1$ and $\msL_2$ are the only possible $\vp$-invariant saturated subsheaves with degree $\geq d$. In particular, if $\al\not=0$ and $\be\not=0$, then $\MSL_1$ and $\MSL_2$ are not $\vp$-invariant and so $\MSE$ is stable. For the converse, we assume that $(\MSE,\varphi)$ is stable. If $\al=0$ (resp. $\be=0$), then $\msL_1$ (resp. $\MSL_2$) is a $\vp$-invariant subbundle of $\msE$ whose slope is equal to $\mu(\msE)$, which is a contradiction. 
		
		For case (b) of (2), the first part of the statement is easy. Now suppose that $\al$ and $\be$ are proportional. Then up to an automorphism of $\msE$, we could write 
        \[
        \msE=\msL_1\oplus \msL_1,\quad \vp=\begin{pmatrix}
			0 & \al\\
			\al & 0
		\end{pmatrix}.
        \]
        In parituclar, using the complex gauge transformation 
        \[
        g=\begin{pmatrix}
			\frac{\sqrt{2}}{2} & \frac{\sqrt{2}}{2}\\
			-\frac{\sqrt{2}}{2} & \frac{\sqrt{2}}{2}
		\end{pmatrix},
        \]
        We have $g^{-1}\vp g=\diag(-\al,\al)$ and hence $(\msE,\vp)=(\msL_1,-\al)\oplus (\msL_1,\al)$ is not stable but only polystable. Moreover, when $(\msE=\msL_1\oplus\msL_1,\vp)$ is polystable but not stable, then $(\msE,\vp)\cong (\msL_1,\omega)\oplus (\msL_1,-\omega)$ and we have $\al\otimes\be=\omega\otimes \omega$, which implies that  $\al$ and $\be$ are proportional. 
        
        For the other direction, suppose that $(\msE,\vp)$ is not stable. Then there exists a $\vp$-invariant rank one saturated subsheaf $\MSL$ of $\MSE$ such that $\deg(\msL)=\deg(\msL_1)=\deg(\MSL_2)=d$. Thus the compositions 
        \begin{center}
            $\msL\to \msE\to \msL_1$\quad and\quad $\msL\to \msE\to \msL_2$
        \end{center}
        are both isomorphisms. As 
        $$\vp\begin{pmatrix}
			1\\
			0
		\end{pmatrix}=\begin{pmatrix}
			0\\
			\al
		\end{pmatrix} \quad \textup{and}\quad
         \vp\begin{pmatrix}
			0\\
			1
		\end{pmatrix}=\begin{pmatrix}
			\beta\\
			0
		\end{pmatrix},
        $$
        the subsheaf $\msL$ is $\vp$-invariant only if $\al$ and $\be$ are proportional.
	\end{proof}
	
	\subsubsection{Moduli space of $\SLR$ Higgs bundles}
	Now we will construction the moduli space of $\SLR$ Higgs bundles. By the canonical shape of real Higgs bundle \eqref{eq_shape_real_Higgs_bundle},  an $\SLR$ Higgs bundle could be written as 
	\begin{equation}
		\label{eq_SLR_Higgs}
		\msE=\msL\oplus \msL^{-1},\;\vp=\begin{pmatrix}
			0 & \be\\
			\al & 0
		\end{pmatrix}
	\end{equation}
	with $\al\in H^0(X,\msL^2\otimes \Omega_X^1)$ and $\be\in H^0(X,\msL^{-2}\otimes \Omega_X^1)$ such that $\al\we\be=0.$ Note that $c_1(\msE)=0$ and $c_2(\msE)=-c_1(\msL)\cdot c_1(\msL).$ By the Bogomolov--Gieseker inequality (cf. Theorem \ref{thm_bogomolov}), the vector bundle $\msE$ is topologically trivial if and only if $c_1(\msL)^2\cdot [\omega]^{n-2}=0$.
	
	Now, we could define the \emph{moduli space of $\SLR$ Higgs bundles} as 
	\begin{equation}
		\MMH^{\SLR}:=\{(\msE,\vp)\in\MMH\,|\,(\msE=\msL\oplus \msL^{-1},\;\vp=\begin{pmatrix}
			0 & \be\\
			\al & 0
		\end{pmatrix})\mathrm{\;satisfies\;}\eqref{eq_SLR_Higgs}\}.
	\end{equation}
	In addition, the \emph{Dolboult moduli space of $\SLR$ Higgs bundles} could be defined to be
	\begin{equation}
		\begin{split}
			\MMD^{\SLR}:=\{(\msE=\msL\oplus \msL^{-1},\vp)\in \MMH^{\SLR}\,|\,c_1(\MSL)^2\cdot [\omega]^{n-2}=0\}
		\end{split}
	\end{equation}
	and by the construction, we have $\MMD^{\SLR}=\MMH^{\SLR}\cap \MMD$. The \emph{$\SLR$ character variety} is defined to be
	\begin{equation}
		\mfR_{\SLR}:=\{\rho:\pi_1(X)\to \SLR\,|\,\rho\;\mathrm{completely\;reducible}\}/\sim.
	\end{equation}
	Then by the non-abelian Hodge correspondences, Theorem \ref{thm_NonabelianHodge} and the Rieman--Hilbert correspondence, we have the real analytic isomorphism 
	$$
	\MMD^{\SLR}\cong \mfR_{\SLR}.
	$$
	
	\subsection{Hitchin section}
	In \cite{hitchin1987self,hitchin1992lie}, Hitchin constructed a section of the Hitchin fibration for Higgs bundles over a curve, which could be generalized to various contexts, see \cite{goldman1990convex,guichard2012anosov}. Now, we would like to give a generalization of Hitchin's construction to rank two Higgs bundles on higher dimensional projective variety. 
	
	\subsubsection{Hitchin section for the moduli space of trace free polystable Higgs bundles} 
 
    Recall we write $\MMH^{0}$ to be the moduli space of trace free polystable Higgs bundle. Then the image of $\ssd_X^0:=\ssd_X|_{\MMH^{0}}$ lies in $\MB_X$. 
    
    We can construct a section of the map $\ssd_X^0:\MMH^0\to \MB_X$ outside the origin as following. By proposition \ref{p.decomposition-symmetric-differetials}, for any $0\not=s\in \MB_X$, we could uniquely find a line bundle $\msL$, a section $\alpha\in H^0(X,\msL^{-1}\otimes\Omega_X^1)$ and $\tau\in H^0(X,\msL^2)$ such that $s=\alpha^2\tau$ and $\alpha$ does not vanish in codimension one. Set $\beta=-\alpha\tau$. Then we could explicitly write down the Higgs bundle 
	\begin{equation}
		\label{eq_Hitchin_section_Higgs_bundle}
		\msE_s=\MSO_X\oplus \msL^{-1},\;\vp_s=\begin{pmatrix}
			0 & \be\\
			\al & 0
		\end{pmatrix},
	\end{equation}
	which is exactly the Higgs bundle obtained in Example \ref{e.canonical-Higgs-bundle} constructed by push-down of the structure sheaf of the Cohen-Macaulayfication $\widetilde{X}_s$ of the spectral variety $X_s$. 
 \begin{definition}
The Hitchin section for traceless Higgs bundles is defined to be the map 
	\begin{equation}
		\chi_{\Hit}^0:\MB_X\setminus\{0\}\to \MMH^0,\;s\to [(\msE_s,\vp_s)]
	\end{equation}.     
 \end{definition}
Based on definition, we have $\ssd_X^0\circ\chi^0_{\Hit}(s)=s$. Though the choices of $\al$ and $\be$ are only unique up to scalars, i.e., $s=(c\al)\cdot (c^{-1}\beta)$ for any $c\in \mbC^*$, it is straightforward to check that the map $\chi_{\Hit}(s)$ is independent of the choice of the constant $c$. 

   Now We discuss some properties of these Higgs bundles in the following. First we consider the stability of $(\MSE,\varphi)$. If $D=\textup{div}(\tau)\neq 0$, then we have
    \[
    \deg(\msL)=\frac{1}{2}\deg(D)>0
    \]
    and thus Proposition \ref{prop_stability_realHiggs} implies that $(\msE_s,\vp_s)$ is stable. On the other hand, suppose $D=0$, then $s=\omega^2$ with $\omega\in H^0(X,\Omega_X^1)$ and $\Div(\omega)=0$. This implies that $\msL=\MSO_X$ and by Proposition \ref{prop_stability_realHiggs}, the Higgs bundle $(\msE_s,\vp_s)$ is polystable and is complex gauge equivalent to the Higgs bundle $(\MSO_X,\omega)\oplus (\MSO_X,-\omega)$. 
	
	Next we note that the Higgs bundles as in \eqref{eq_Hitchin_section_Higgs_bundle} are real by Proposition \ref{prop_stability_realHiggs}, and one can easily see that $\Delta(\msE_s)\cdot[\omega]^{n-2}=0$ if and only if $c_1(D)^2\cdot [\omega]^{n-2}=0$. Moreover, if $\msL$ admits a square root $\msL^{\frac12}$, then we could construct an $\SLC$ Higgs bundle as 
	\begin{equation}
		\label{eq_SLC_Hitchin_section}
		\msE_s'=\msL^{\frac12}\oplus \msL^{-\frac12},\;\vp_s'=\begin{pmatrix}
			0 & \be\\
			\al & 0
		\end{pmatrix}.
	\end{equation}
	In summary, we have proved the following results.
	\begin{proposition}
 \label{p.Properties-of-Hitchin-section}
		Let $0\not=s\in \MB_X$ with $D=\textup{div}(s)$ and let $(\msE_s,\vp_s)=\chi_{\Hit}^0(s)$ be the Higgs bundle in the Hitchin section. Then The following statements hold.
  \begin{enumerate}
            \item The Higgs bundle $(\msE_s,\vp_s)$ is real. 
            \item The Higgs bundle $(\msE_s,\vp_s)$ is stable if $D\neq 0$ and it is polystable if $D=0$.
			\item If $c_1(D)^2\cdot [\omega]^{n-2}=0$, then $\Delta(\msE_s)\cdot[\omega]^{n-2}=0$ and under the non-abelian Hodge correspondence, we obtain a representation $\rho:\pi_1(X)\to \PSLR$.
			\item If $\msL$ admits a square root and $c_1(D)^2\cdot[\omega]^{n-2}=0$, then we obtain from \eqref{eq_SLC_Hitchin_section} a representation $\rho:\pi_1(X)\to \SLR$.
		\end{enumerate}
	\end{proposition}
	
	Our construction recover the original Hitchin section that constructed by Hitchin over Riemann surface in \cite{hitchin1987self,hitchin1992lie}.  Let $X=\Sigma$ be a curve with genus $g(\Sigma)\geq 2$ with $\Omega_{\Sigma}^1=K_X$ the canonical bundle, then $\MB_X=H^0(X,K_X^2)$ is the space of quadratic differential and 
	the integrability condition $\vp\wedge\vp=0$ is automatically satisfied. Given $s\in H^0(X,K_X^2)$ and $\msL=K_X$, then $\msL^{-1}\otimes K_X=\MO_X$ and $\msL\otimes K_X=K_X^2.$ Then $\al\in H^0(X,\msL^{-1}\otimes K_X)$ is a non-zero constant function and $\beta\in H^0(X,K_X^2)$ with $\al\be=-s$. In addition, as $\deg(K)$ is even and every even degree line bundle over a Riemann surface admits a square root, after choosing a square root $K_X^{\frac12}$, and taking $\al=1$, $\beta=-s$, then we obtain the Hitchin section for Riemann surface \cite[Section 11]{hitchin1987self}:
	\begin{equation}
		\msE=K^{\frac12}\oplus K^{-\frac12},\;\vp=\begin{pmatrix}
			0 & -s \\
			1 & 0
		\end{pmatrix}
	\end{equation}
	
	\subsubsection{Hitchin section for general rank two Higgs bundle.} Now, we would like to construct the Hitchin section without the trace free assumption of the Higgs field.
	
	For general rank two Higgs bundles, recall from \eqref{eq_spectral_base_GL2Higgs_bundle} that the spectral base could be written as  $$\MS_X:=\{s=(s_1,s_2)\in \MA_X\,|\,4s_2-s_1^2\in \MB_X\}.$$
	In particular, if $4s_2-s_1^2\not=0$, by Proposition \ref{p.decomposition-symmetric-differetials}, there is a unique way to write
	$\frac{1}{4}s_1^2-s_2=\al\be$ with $\al\in H^0(X,\msL^{-1}\otimes\Omega_X^1)$ and $\be\in H^0(X,\msL\otimes\Omega_X^1)$ such that $\textup{div}(\al)=0$. Denote by $\MS_X^{\nil}$ the closed subset of $\MS_X$ consisting of $s=(s_1,s_2)$ such that $4s_2-s_1^2=0$. Now the Hitchin section $\chi_{\Hit}:\MS_X\setminus \MS_X^{\nil}\rightarrow \MMH$ is defined as
	\begin{equation}
		\begin{split}
			\chi_{\Hit}(s):=\left(\msE=\msL\oplus \MSO_X,\vp=\begin{pmatrix}
				\frac12 s_1 & \beta \\
				\al & \frac12s_1
			\end{pmatrix}\right).
		\end{split}
	\end{equation}
	Clearly we have $\ssd_X\circ \chi_{\Hit}(s)=s$. The stability condition and $\CS$-limit could be understood as following.
	\begin{enumerate}
	    \item Suppose $\textup{div}(4s_2-s_1^2)\neq 0$. As the traceless Higgs bundle $(\msE,\vp-\frac12\Tr(\vp)\Id)$ is stable, the Higgs bundle $(\msE,\vp)$ is also stable. Moreover, by a straight forward computation, the $\CS$-limit would be 
     $$
     \lim_{t\to 0}[(\msE,t\vp)]=\left(\msE,\begin{pmatrix}
		0 & 0 \\
		\al & 0
	\end{pmatrix}\right).
 $$
 \item If $\frac14s_1^2-s_2=\omega^2$ for some $0\not=\omega\in H^0(X,\Omega_X^1)$ such that $\textup{div}(\omega)=0$, then $\msL=\MSO_X$. Then based on our construction, we have
 $$
 \chi_{\Hit}(s)=\left(\msE=\MSO_X\oplus \MSO_X,\vp=\begin{pmatrix}
		\frac12 s_1 & \omega \\
		\omega & \frac12s_1
	\end{pmatrix}\right).
 $$
 Using the gauge transformation $g=\begin{pmatrix}
		\frac{\sqrt{2}}{2} & \frac{\sqrt{2}}{2}\\
		-\frac{\sqrt{2}}{2} & \frac{\sqrt{2}}{2}
	\end{pmatrix},$ it is gauge equivalent to the polystable Higgs bundles $(\MSO,\frac{1}{2}s_1-\omega)\oplus (\MSO,\frac{1}{2}s_1+\omega).$ Moreover, as $\msE$ is polystable, the $\CS$-limit would be 
    \[
    \lim_{t\to 0}[(\msE,t\vp)]=[(\MSO_X\oplus \MSO_X,0)]\in \MMD.
    \]
    \item If $\frac14s_1^2-s_2=0$ but $s_1\neq 0$, then we define
     $$
 \chi_{\Hit}(s)=\left(\msE=\MSO_X\oplus \MSO_X,\vp=\begin{pmatrix}
		\frac12 s_1 & 0\\
		0 & \frac12s_1
	\end{pmatrix}\right).
 $$
	\end{enumerate}
	In summary, we conclude the following:
	\begin{theorem}
		\label{thm_Hitchinsection_existence}
		There exist sections $\chi^0_{\Hit}:\MB_X\setminus \{0\}\to \MMH^0$ and $\chi_{\Hit}:\MS_X\setminus\{(0,0)\}\to \MMH$ such that $\ssd_X^0\circ\chi_{\Hit}^0=\Id$ and $\ssd_X\circ\chi_{\Hit}=\Id$.
	\end{theorem}
 
	\subsubsection{$\SLR$ Higgs bundle and the Hitchin morphism.}
	Now, we will discuss other $\SLR$ Higgs bundle besides the Hitchin section. 
 
    Firstly, we consider a $\SLR$ Higgs bundle of the following form
    \[
    \left(\msE=\MSN\oplus \MSN^{-1},\vp=\begin{pmatrix}
		0 & \be\\
		\al & 0
	\end{pmatrix}\right)
 \]
 with $\al\in H^0(X,\MSN^{-2}\otimes\Omega_X^1)$, $\be\in H^0(X,\MSN^{2}\otimes\Omega_X^1)$ and $\al\we \be=0$. Then we have
 \[
 s:=\ssd_X(\vp)=-\al\be\in \MB_X.
 \]
 Assume $s\not=0$ and let $D=\textup{div}(s)$. By Proposition \ref{p.structure-BXL}, we could write $s=\widetilde{\alpha}^2\tau$, where $\tau$ is the global section of $\MSO_X(D)=\msL^2$ corresponding to $D$ and $\widetilde{\alpha}\in H^0(X,\msL^{-1}\otimes\Omega_X^1)$. On the other hand, if we write $D_{\al}=\textup{div}(\al)$ and $\;D_{\be}=\textup{div}(\be)$ with the canonical sections $s_{\al}\in H^0(X,\MSO_X(D_{\al}))$ and $s_{\be}\in H^0(X,\MSO_X(D_{\be}))$, then it follows from $D=D_{\al}+D_{\be}$ and $\al\wedge\be=0$ that up to scaling we have 
 \[
 \frac{\al}{s_{\al}}=\frac{\be}{s_{\be}}=\widetilde{\alpha}.
 \]
	
	Next, for any given $0\not=s\in \MB_X$, one could construct all possible $(\msE,\vp)\in \MMD^{\SLR}$ with $\ssd_X(\vp)=s$ as following. We first write $s=\alpha^2\tau$ with $D=\textup{div}(s)$. Set $\MSL=\MSO_X(D)$. Given a decomposition $D=D_1+D_2$ with $D_1,D_2$ effective divisors. Let $\tau_1\in H^0(X,\MSO_X(D_1))$ and $\tau_2\in H^0(X,\MSO_X(D_2))$ be the canoncial sections. Then we have $\tau=\tau_1\tau_2$. We could define 
 $$
 \al:=\tau_1\alpha\in H^0(X,\msL^{-1}\otimes \MSO_X(D_1)\otimes\Omega_X^1)
 $$
 and 
 $$
 \be:=-\tau_2\alpha\in H^0(X,\msL\otimes \MSO_X(D_2)\otimes\Omega_X^1).
 $$
 The $\SLR$ condition would requires that there exists a square root of the line bundle $\msL^{-1}\otimes \MSO_X(D_1)=\MSN^{-2}$ (compare with \S\,\ref{s.towerofCMness}). In this situation, the vector bundle $\MSN\oplus \MSN^{-1}$ is topologically trivial if and only if 
	$$
	(c_1(D_1)-c_1(D_2))^2\cdot [\omega]^{n-2}=0.
	$$
 In addition, since there are only finitely many choices of the decomposition $D$ into $D_1$ and $D_2$, and the choice of square root only depends on $H^1(X;\ZT)$, which is finite. The preimage $\ssd_X^{-1}(s)\cap \MMD^{\SLR}$ only contains finite number of Higgs bundles.
 
	\begin{definition}
 \label{d.MXDOLSL2R}
		The $\SLR$ Dolbeault base $\MB_{X;\Dol}^{\SLR}\subset \MB_X$ is defined to be $s\in \MB_X$ with $D=\textup{div}(s)$ and $\MSO_X(D)=\MSL^{2}$ such that the following conditions hold:
		\begin{enumerate}
			\item there exists a decomposition of effective divisors $D=D_1+D_2$ with $$(c_1(D_1)-c_1(D_2))^2\cdot [\omega]^{n-2}=0,$$
			\item the line bundle $\msL^{-1}\otimes \MSO_X(D_1)$ admits a square root.
		\end{enumerate}
	\end{definition}
	
	In summary of the previous discussions, we conclude the following:
	\begin{theorem}
		\label{thm_slr_Hitchin_morphism}
		The restriction of the spectral morphism $\ssd_X:\MMD^{\SLR}\to \MB_{X;\Dol}^{\SLR}$ is surjective and for each $0\not=s\in \MB_{X;\Dol}^{\SLR}$, the preimage $\ssd_X^{-1}(s)$ only contains finite number of Higgs bundles.
	\end{theorem}
	
	\section{Rigidity of the character variety}
	\label{sec_rigidity_character_variety}
	Throughout this section, we will denote by $G$ one of the gourps $\SLR$, $\SLC$ and $\GLC$. A representation $\rho:\pi_1(X)\to G$ is called \emph{rigid} if it is an isolated point in the $G$-character variety. The $G$-character variety $\mfR^G$ is called \emph{rigid} if every element $\rho\in \mfR^G$ is rigid. By \cite{Simpson1992}, a rigid representation must be a complex variation of Hodge structure and it is conjectured to be of geometric origin. This conjecture have been confirmed for rank two and three bounded local systems \cite{simpsoncorlette2008classification,simpsonadrian2018rank}. Therefore, it will be in particular interest to understand the rigidity of the character variety.
	
	\subsection{Rigidity and the spectral base}
	Now, we will discuss the relationship between the rigidity of the character variety and the spectral base, follows from \cite[Proposition 2.4]{arapura2002higgs}, see also \cite[Theorem 1.6]{klingler2013symmetric}, \cite[Section 4.1.4]{zuo2000negativity}, \cite{brotbek2022representations}, \cite{brunebarbe2013symmetric} for generalization to various content.
	
	Let $\ssd_X:\MMH\to \MS_X$ be the spectral morphism, we definte $\ssd_X^{\Dol}:=\ssd_X|_{\MMD}$ be the restriction on the Dolbeault space. The image of $\ssd_X^{\Dol}$ is called the Dolbeault base, which we denote as $\MS_X^{\Dol}:=\text{Im}(\ssd^{\Dol}_X).$ Similarily, we denote $\MB_X^{\Dol}$ be the imgae of the restriction of the Hitchin morphism into the $\SLC$ Dolbeault space, $\MB_X^{\Dol}:=\text{Im}(\ssd_X|_{\MMD^{\SLC}})\subset \MB_X.$ 
	
	\begin{definition}
		The nilpotent cone of the Higgs bundle moduli spaces $\MMH$ is called rigid if for every Higgs bundle $(\msE_0,\vp_0)$ lying in the nilpotent cone, there doesn't exists a family of Higgs bundles $(\msE_t,\vp_t)$ with $\ssd_X(\msE_t,\vp_t)\neq 0$ such that 
        \[
        \lim_{t\to 0}[(\msE_t,\vp_t)]=[(\msE_0,\vp_0)].
        \]
        This concept could be extended to $\MMH^0$, $\MMD$, $\MMD^{\SLC}$ and $\MMD^{\SLR}$ similarly.
	\end{definition}
	
	\begin{theorem}{\cite[Proposition 2.4]{arapura2002higgs}}
		\label{thm_Dolboult_base_rigidity}
		The following statements are equivalent.
		\begin{enumerate}
			\item The character variety $\mfR^{\GLC}$ $($resp. $\mfR^{\SLC})$ is rigid. 
			\item $\MS_X^{\Dol}=\{0\}$ $($resp. $\MB_X^{\Dol}=\{0\})$.
			\item The nilpotent cone of $\MMD$ $($resp. $\MMD^{\SLC})$ is rigid.
		\end{enumerate}
		In particular, $\mfR^{\GLC}\;($resp. $\mfR^{\SLC})$ is non-rigid if and only if the Dolboult base $\MS_X^{\Dol}\;($ resp. $\MB_X^{\Dol})$ is non zero.
	\end{theorem}
	\begin{proof}
		The proof for $\SL_2(\mbC)$ is the same as that for $\GL_2(\mbC)$, so we will only prove the equivalences for $\GLC$ case. 
  
        For $(1)$ to $(2)$, suppose $\mfR^{\GLC}$ is rigid, but $\MS_X^{\Dol}\neq 0$, then there exists a topological trivial polystable Higgs bundle $(\msE,\vp)$ with $s:=\ssd_X(\msE,\vp)\neq 0$. For every $t\neq 1\in \CS$, as $\sh(\msE,t\vp)=ts\neq 0$, we obtain a deformation which contradicts to the assumption that $\mfR^{\GLC}$ is rigid. 
		
		As $\MS_X^{\Dol}$ is defined to be the image of the Hitchin morphism, $(2)$ to $(3)$ is based on definition. For $(3)$ to $(1)$, by Theorem \ref{action_fixed_point_Hodge_bundle}, as the $\CS$-limit of every Higgs bundle exists and lies in the nilpotnent cone. Therefore, the nilpotent cone is rigid if and only if $\MMD$ is supported in the nilpotent cone. By Theorem \ref{thm_Hitchinmap_proper}, the Hitchin map is proper, which implies $\MMD$ is compact. Thus the character variety $\mfR^{\GLC}$ is compact by non-abelian Hodge correspondence. However, as $\mfR^{\GLC}$ is an affine variety, it is compact if and only if 
		$\mfR^{\GLC}$ consists of finite number of points, which implies $(1)$.  
	\end{proof}
	The above theorem could be easily generalized to higher rank Higgs bundles and the main difficulty is to understand what is the $\MS_X^{\Dol}$. In addition, let $b_1(X)$ be the first betti number of $X$, then the condition $b_1(X)>0$ will violate the rigidity of the character variety as follows.
	\begin{proposition}
        \label{prop_firstbettinumber_notrigid}
		Suppose $b_1(X)\neq 0$, then $\MMD$ is not rigid. In addition, every $\VHS$ could be deformed outside of the nilpotent cone.
	\end{proposition}
	\begin{proof}
		By Hodge decomposition theorem, the fact $b_1(X)\neq 0$ implies that there exists a non-zero $\omega\in H^0(X,\Omega_X^1)$. The period of $t\omega$ for $t$ various could generates non-rigid representations. Indeed, we could construct a family of polystable $\SLC$ Higgs bundles as following
        \[
        (\msE,\vp_t)=(\MSO_X,t\omega)\oplus (\MSO_X,-t\omega).
        \]
        As $\Tr(\vp_t^2)=t^2\omega^2$, it follows that $(\msE,\vp_{t_1})$ and $(\msE,\vp_{t_2})$ are not gauge equivalent for $t_1^2\neq t_2^2$. Now we consider a $\VHS$ of the following form
        \[
        \left(\msE_0=\msL\oplus \msL^{-1},\;\vp_0=\begin{pmatrix}
			0 & 0\\
			\al & 0
		\end{pmatrix}\right).
        \]
        Then the following family of stable Higgs bundles 
        \[
        \left(\msE_t=\msL\oplus \msL^{-1},\;\vp_t=\begin{pmatrix}
			t\omega & 0\\
			\al & -t\omega
		\end{pmatrix}\right)
        \]
        provides a deformation of the $(\msE_0,\vp_0)$ outside of the nilpotent cone.
	\end{proof}
	
	The criteria for rigidity used in \cite{klingler2013symmetric,simpsonadrian2018rank} are formulated as the  vanishing of the Hitchin base $\MA_X$ and this could be strength to the following statement.
	
	\begin{corollary}
		\label{corollary_criterion_Dolbeault_Base}
		Let $X$ be a projective manifold. If $H^0(X,\msL^{2})=0$ for every line bundle $\msL$ with $H^0(X,\msL^{-1}\otimes\Omega_X^1)\neq 0$, then the character varieties $\mfR^{\SLC}$ and $\mfR^{\GLC}$ are rigid. 
	\end{corollary}
	\begin{proof}
		The condition implies $\MB_X=0$ and $H^0(X,\Omega_X^1)=0$. Thus one derives $\MS_X=0$ and the rigidity then follows from Theorem \ref{thm_Dolboult_base_rigidity}.
	\end{proof}

 \begin{remark}
	Unfortunately, the rigidity of the character variety does not imply  $\MB_X=0$. By Example \ref{example_simply_connected_nonvanishing_rankone_differential}, there is an example of simply connected projective manifold $X$ with $\MB_X\neq 0.$     
 \end{remark}
	Now, we will introduces several examples that the Dolbeault basis is known.
	\begin{example}
		Let $X=\Sigma_1\ti \Sigma_2$ with genus $g_1,g_2\geq 2$ and let $\rho:\pi_1(X)\to \SLC$ be an irreducible representation. As $\pi_1(X)=\pi_1(\Sigma_1)\ti\pi_1(\Sigma_2)$, we write $\rho_i:=\rho|_{\pi_1(\Sigma_i)}$, then suppose $\rho_1$ and $\rho_2$ are both non-trivial. We write $\gamma_1,\cdots,\gamma_{2g_1}$ be generators of $\pi_1(\Sigma_1)$, $\gamma_1',\cdots,\gamma_{2g_2}'$ be generators of $\pi_1(\Sigma_2)$ and we write $A_i=\rho_1(\gamma_i)$, $B_i=\rho_2(\gamma_i')$. WLOG, we assume $B_1\neq \Id$, then we have $A_iB_1=B_1A_i$. Let $\pm \lam$ be two eigenvalues of $B$ with eigenvectors $v_{\pm}$, then $A_iv_{\pm}=\lam_{\pm}v_{\pm}$. In addition, for any $B_j$, then condition $B_jA_i=A_iB_j$ implies that $B_jv_{\pm}=\lam_{\pm}v_{\pm}$. Therefore, any irreducible representation $\rho:\pi_1(X)\to \SLC$ must factor through $\pi_1(\Sigma_1)$ or $\pi_1(\Sigma_2)$. Therefore, stable topological trivial $\SLC$ Higgs bundle must be pull-back of stable Higgs bundle over $\Sigma_1$ or $\Sigma_2$. 
		
		In addition, for $\omega\in H^0(X,\Omega_X^1)$, we have the polystable $\SLC$ Higgs bundles $(\msL,\omega)\oplus (\msL^{-1},-\omega)$. By the computation in Example \ref{example_products_of_curves}, we have $\MB_X^{\Dol}=\MB_X$.
	\end{example}

 \begin{example}
     Let $X$ be an arithematic variety of rank $\geq 2$, then follows from Margulis superrigidity \cite{margulis1991discrete}, we have $\MB_X^{\Dol}=0$. In addition, it is shown in \cite{heliumok2023rigidity} that $\MB_X=0$. Therefore, every Higgs bundle over $X$ is nilpotent.
 \end{example}

\subsection{Rigidity of projective manifolds with Picard number one}
The condition in Corollary \ref{corollary_criterion_Dolbeault_Base} is generally challenging to verify, but it becomes more understandable in the context where the smooth projective variety $X$ has Picard number one.

We denote $\Pic(X)$ as the Picard group and $\Pic^0(X)$ as the connected component of the trivial bundle. The Néron–Severi group, denoted as $\NS(X)$, is defined as $\Pic(X)/\Pic^0(X)$, and the \emph{Picard number} $\rho(X)$ is defined as the rank of $\NS(X)$. This Picard number must satisfy the following inequality: $$1\leq \rank\;\NS(X)\leq b_2-2h^{2,0}.$$ An illustrative example of a smooth projective variety with Picard number one is the class of fake projective spaces. We refer the reader to the comprehensive studies on the rigidity problem of fake projective planes in \cite{Klingler2003fakeprojective,Yeung2004integralityandarithmeticity}.

The rigidity of the $\GLC$ character variety over a smooth projective variety $X$ with Picard number one can be fully understood. According to Corollary \ref{corollary_criterion_Dolbeault_Base}, if we assume $\MB_X\neq 0$, then there exists a holomorphic line bundle $\MSL$ with $H^0(X,\MSL^{-1}\otimes \Omega_X^1)\neq 0$ and $H^0(X,\MSL^2)\neq 0$. The condition $H^0(X,\MSL^{-1}\otimes \Omega_X^1)\neq 0$ implies $\kappa(X,\MSL)\leq 1$ by Theorem \ref{t.BCdFtheorem}. In particular, the line bundle $\MSL$ cannot be ample if $\dim(X)\geq 2$. On the other hand, as $H^0(X,\MSL^2)\neq 0$, the dual line bundle $\MSL$ also cannot be ample. As a consequence, if $\rank\;\NS(X)=1$ and $\dim(X)\geq 2$, then we must have $c_1(\MSL)=0$ in $H^2(X,\mbQ)$. In particular, as $H^0(X,\MSL^{2})\not=0$, we get $\MSL^2\cong \MSO_X$; that is, either $\MSL\cong \MSO_X$ or $\MSL$ is a torsion line bundle of order two. Recall that the two-torsion line bundles are in one-to-one correspondence with the homology classes $H^1(X;\ZT)$. 

Now we assume in addition that $\MSL$ is a torsion line bundle of order two. Then $\MSL$ defines a unramified double covering $p:\tX\to X$ such for each $0\not=\al\in H^0(X,\MSL^{-1}\otimes \Omega_{\tX}^1)$, we have $0\not=p^*\al\in H^0(\tX, \Omega_{\tX}^1).$ Conversely, let $p:\tX\rightarrow X$ be a unramified double covering with the corresponding torsion line bundle $\MSL$ on $X$ of order two. Let $0\not=\widetilde{\alpha}\in H^0(\tX,\Omega_{\tX}^1)$ be non-zero one form. As $p$ is unramified, we have the following natural isomorphisms:
\[
H^0(\tX,\Omega_{\tX}^1) = H^0(\tX,p^*\Omega_X^1) \cong H^0(X,p_*p^*\Omega_X^1)\cong H^0(X,\Omega_X^1)\oplus H^0(X,\MSL^{-1}\otimes \Omega_X^1).
\]
In particular, if $b_1(X)=0$, then $\widetilde{\alpha}$ defines a natural element $0\not=\alpha\in H^0(X,\MSL^{-1}\otimes \Omega_X^1)$ such that $p^*\al=\widetilde{\al}$. As consequence, the elements in $\MB_X$ can be fully determined by the topology of $X$; that is, $\MB_X=0$ if and only if $b_1(X)=0$ and for any unramified double covering $\tX\to X$, we have $b_1(\tX)=0$. 

Finally, if $\MB_X\not=0$, we remark that the Higgs bundles in the Hitchin section are actually topologically trivial by our construction. Indeed, if $\MSL$ is a torsion line bundle of order two, for any $0\not=\al\in H^0(X,\MSL^{-1}\otimes \Omega_X^1)$, the Hitchin section over the point $\alpha^2\in \MB_X$ can be written as
\begin{equation}
\MSE=\MSO_X\oplus \MSL^{-1},\;\vp=\begin{pmatrix}
0 & \al\\
-\al & 0
\end{pmatrix}.
\end{equation}
Since $c_1(\MSL)=0$, $c_i(\MSE)=0$ for any $i$; so $\MSE$ is topologically trivial. According to Proposition \ref{prop_stability_realHiggs}, the Higgs bundle $(\MSE,t\vp)$ is stable for $t\not=0$ as $\al\not=0$. As $\det(\vp)=\al^2\not=0$, it defines a non-rigid family of topologically trivial Higgs bundles by Proposition \ref{prop_Hodge_bundle_stablity}. Moreover, as the pull-back $p^*(\MSE,\vp)$ is a direct sum of two Higgs line bundles, the corresponding representation will not be Zariski dense. 

In summary, we can conclude the following:

\begin{theorem}
Let $X$ be a smooth projective variety with Picard number one. Then $\MB_X=\MB_X^{\Dol}$. Moreover, the $\GLC$ character variety $\mfR^{\GLC}$ is rigid if and only if $b_1(X)=0$ and for any unramified double covering $\tX\to X$, we have $b_1(\tX)=0.$
\end{theorem}

\begin{proof}
    If $\dim(X)\geq 2$, the result follows from our discussion above. If $\dim(X)=1$, the result is obvious from the classical results.
\end{proof}

In particular, we have:
\begin{corollary}
Let $X$ be a smooth projective variety with Picard number one. If $b_1(X)=0$ and $H^1(X;\ZT)=0$, then $\mfR^{\GLC}$ is rigid.
\end{corollary}

	\subsection{Rigiditiy of the nilpotent cone}
	Now, we will discuss the rigidity of the nilpotent cone in various content. 
	\subsubsection{Rigidity of $\MMH$}
	We first consider the rigidity of the nilpotent cone of $\MMH$.
	
	\begin{theorem}
		The nilpotent cone of the moduli space of rank two polystable Higgs bundle $\MMH$ is rigid if and only if $\MS_X=0$. 
	\end{theorem}
	\begin{proof}
		If the nilpotnent cone is not rigid, then it follows from the definition that $\MS_X\neq 0$. For the other direction, by Theorem \ref{thm_Hitchinsection_existence}, for any $0\not=s\in \MS_X$, there exists a polystable Higgs bundle $(\msE,\vp)$ with $\ssd_X(\msE,\vp)=s$ such that the $\mbC^*$-limit $\lim_{t\to 0}[(\msE,t\vp)]$ exists in $\MMH$ and lies in the nilpotent cone. Thus $\MS_X\neq 0$ implies the nilpotent cone is not rigid. 
	\end{proof}
	
	In particular, for the projective manifold $X$ in Example \ref{example_simply_connected_nonvanishing_rankone_differential} $(5)$, the nilpotent cone of $\MMH$ is not rigid but the nilpotent cone of $\MMD$ is rigid.
	
	\subsubsection{Rigidity of $\SLR$ character variety}
	As all Higgs bundles in $\MMD^{\SLR}$ are tracefree, the spectral base for  $\MMD^{\SLR}$ would be a subset of the rank $1$ symmetric differentials $\MB_X$ and now we want to understand the image of $\ssd_X|_{\MMD^{\SLR}}$. We refer the reader to Definition \ref{d.MXDOLSL2R} for the notation $\MB_{X;\Dol}^{\SLR}$.
	
	\begin{theorem}
		For the $\SLR$ Dolbeault moduli space, the following statements hold.
		\begin{enumerate}
			\item The nilpotent cone of $\MMD^{\SLR}$ is rigid if and only if $\MB_{X;\Dol}^{\SLR}=0$.  
			\item The $\SLR$ character variety $\mfR^{\SLR}$ is rigid if and only if $\MB_{X;\Dol}^{\SLR}=0$ and there exists only finite number of determinant trivial rank two $\VHS.$
		\end{enumerate}
	\end{theorem}
	\begin{proof}
		The statement $(1)$ follows directly from Theorem \ref{thm_slr_Hitchin_morphism}. For $(2)$, the chracter variety $\mfR^{\SLR}$ is rigid if and only if the nilpotent cone of $\MMD^{\SLR}$ is rigid and every element of the nilpotent cone is rigid. However, by Proposition \ref{prop_stability_realHiggs}, the nilpotent cone for $\MMD^{\SLR}$ is exactly the same as determiant trivial $\VHS$, this implies $(2)$.
	\end{proof}
	
	\section{Further discussions and applications}
	\label{sec_further_discussions_applications}
Over a Riemann surface, there exists a fascinating theory related to the Hitchin section and Hitchin map, the Milnor-Wood inequality \cite{hitchin1987self}, the Hitchin integral system \cite{hitchin1987hyperkahler}, and the Hitchin component \cite{hitchin1992lie}. In this section, we will explore various generalizations of these constructions to higher dimensions.
	
	\subsection{Milnor-Wood type inequality}
	Let $X$ be a smooth projective curve and let $\rho:\pi_1(X)\rightarrow G$ be a representation into a linear connected simple noncompact Lie group of Hermitian type. One can introduce the so-called \email{Toledo invariant} $\tau(\rho)$. In particular, in the case $G=\GL_2(\mbC)$, it is equivalent to considering the Euler class and it satisfies the Milnor-Wood inequality. 
 
    Follows from Burger-Iozzi and Koziarz-Manbon \cite{burger2007bounded,KoziarzMaubon2010}, one could extend the definition of the Toledo invariant to representations of fundamental groups of smooth projective varieties into $\SL_2(\mbC)$. Furthermore, we will demonstrate that this extended Toledo invariant also satisfies a Milnor-Wood type inequality.
	
	\subsubsection{Mobile curve classes}
	
	Let $X$ be a projective manifold, and let $N_1(X)_{\mbR}$ denote the set of numerical classes of one-cycles. We introduce the following definition.
	
	\begin{definition}
		Let $X$ be an $n$-dimensional projective manifold. A non-zero class $\gamma\in N_1(X)_{\mathbb{R}}$ is called mobile if there exists a birational projective morphism $\pi:X'\rightarrow X$ from a smooth projective variety $X'$ to $X$, a collection of nef and big divisors $\{A_i\}_{1\leq i\leq n-1}$ and some $c\in \mbR_{>0}$, such that 
		\[
		\gamma = c\pi_*(A_1\cdots A_{n-1}).
		\]	
	\end{definition}

    \begin{remark}
    \label{r.movable-curves}
    \begin{enumerate}
        \item According to \cite{BoucksomDemaillyPuaunPeternell2013}, a line bundle $\MSL$ is pseudoeffective if and only if $\deg_{\gamma}(\MSL)\geq 0$ for any mobile curve class $\gamma$. On the other hand, by the projection formula, given any line bundle $\MSL$ on $X$, we have
        \[
        \deg_{\gamma}(\MSL):=c_1(\MSL)\cdot \gamma=c\pi^*\MSL\cdot c_1(A_1)\cdots c_1(A_{n-1}).
        \]

    \item Recall that nef line bundles can be approximated by ample line bundles. So there exist ample $\mbQ$-divisors $A_i^k$ such that $c_1(A_i^k)\to c_1(A_i)$ as $k\to +\infty$. In particular, there exists a sequence of smooth projective curves $C_k\subset X'$ with non-constant maps $f_k:C_k\rightarrow X$ and positive rational numbers $c_k$ such that
    \[
    c_k C_k \equiv A_1^k\cdots A_{n-1}^k.
    \]
    In particular, we have $c_k \pi_*C_k \to \gamma$ as $k\to \infty$
    \end{enumerate}
    \end{remark}

    \subsubsection{Toledo invariant}
    Let $\mbfE$ be a complex vector space of dimension $2$, endowed with a non-degenerate Hermitian form $h$ of signature $(1,1)$. Let $\mbfW$ be the $1$-dimensional complex subspace of $\mbfE$ on which $h$ is negative-definite and let $\mbfV$ be its $h$-orthogonal complement. Let $G=\SU(1,1)$ be the subgroup of $\SL_2(\mbC)$ preserving $h$. Recall that $G$ is isomorphic to $\SL_2(\mbR)$. Let $K$ be the isotropy subgroup $K$ of $G$ at $\mbfW$. Then the associated Hermitian symmetric space $Y=G/K$ is biholomorphic to a complex unit ball $\mbB^1$ contained in $\textrm{Gr}(1,\mbfE)\cong \mbP^1$. Let $\MSL_Y$ be the restriction of the line bundle $\MSO_{\mbP^1}(-1)$ to $Y$.
    
    Let $X$ be a projective manifold and fix a mobile curve class $\gamma$ on $X$. Let $\rho:\pi_1(X)\rightarrow G$ be a representation. Let $\widetilde{X}$ be the universal cover of $X$. Since $Y$ is contractible, there exists a $\rho$-equivariant $\MC^{\infty}$ map $f:\widetilde{X}\rightarrow Y$ and any two of them are homotopic. Since the complexification of $\SU(1,1)$ is $\SL(2,\mbC)$, which is simply connected, the group $G$ acts by automorphisms on $\MSL_{Y}$ and the smooth complex line bundle $L_{\widetilde{X}}:=f^*\msL_X$ descends to a natural smooth line bundle $L_{X}$ on $X$ \cite[\S\,2]{KoziarzMaubon2010}. Now we define the \emph{Toledo invariant of $\rho$ with respect to $\gamma$} as
    \[
    \tau_{\gamma}(\rho):=\deg_{\gamma}(L_X)=c_1(L_X)\cdot \gamma.
    \]
    \begin{remark}
        This definition is independent of the choice of $f$ and clearly it can be easily generalised to other linear simply connected noncompact Lie groups of Hermitian type. Moreover, if $X$ is of general type, i.e, $K_X$ is big,  there exists a birational contraction $g:X\dashrightarrow X_{\textup{can}}$ with $K_{X_{\textup{can}}}$ ample. Let $\pi:\widetilde{X}\rightarrow X$ be a resolution of $g$ with the induced morphism $\widetilde{g}:\widetilde{X}\rightarrow X_{\textup{can}}$. We define
    \[
    \gamma:=\pi_*(\widetilde{g}^*K_{X_{\textup{can}}})^{n-1}.
    \]
    As $\widetilde{X}$ is birational and $K_{\textup{can}}$ is ample, the pull-back $\widetilde{g}*K_{X_{\textup{can}}}$ is big and nef. In particular, the class $\gamma$ is mobile and the corresponding Toledo invariant $\tau_{\gamma}(\rho)$ is exactly that defined in \cite[Definition 3.1]{KoziarzMaubon2010}.
    \end{remark}

    \subsubsection{Milnor-Wood type inequality}
    
    Now we are in the position to prove a Milnor-Wood type inequality for the Toledo invariant introduced above.

    \begin{proof}[Proof of Theorem \ref{t.MWineq}]
        In fact the proof is exactly the same as that of \cite[Proposition 4.3]{KoziarzMaubon2010} and the key point is the remarkable uniruledness criterion established in \cite{BoucksomDemaillyPuaunPeternell2013}. By \cite[Corollary 0.3]{BoucksomDemaillyPuaunPeternell2013}, the canonical bundle $K_X$ is pseudoeffective as $X$ is non-uniruled. In particular, we have $\deg_{\gamma}(K_X)\geq 0$.
        
        Let $\mbfE=\mbfW\oplus \mbfV$ be the standard  representation of $\SU(1,1)$ and let $(\msE,\varphi)$ be the associated flat rank two Higgs bundle which is polystable with respect to some ample divisor $A$ (see \cite[\S\,3.2 and Example 3.6.3]{Maubon2015}). Then $\msE$ splits holomorphically as a direct sum $\msW\oplus \msV$ and the Higgs field $\varphi$ has the form
        \begin{center}
            $\begin{pmatrix}
            0  &  \beta \\
            \alpha & 0
        \end{pmatrix}$, with 
        $\begin{cases}
            \beta: \msV\rightarrow \msW\otimes \Omega_X^1\\
            \alpha:\msW\rightarrow \msV\otimes \Omega_X^1
        \end{cases}$.
        \end{center}
        Moreover, the line bundle $\msW\otimes \msV^{-1}$ is isomorphic to $L_X^{2}$ as smooth complex line bundles. In particular, we have
        \[
        \deg_{\gamma}(\msW)=-\deg_{\gamma}(\msV)=\tau_{\gamma}(\rho).
        \]
        
        If $\alpha=0$, then $\msW$ is a Higgs subbundle of $\msE$. However, since the Higgs bundle $(\msE,\varphi)$ is flat and polystable with respect to some polarisation, it follows that for any projective smooth curve $C$ with non-constant map $f:C\rightarrow X$, the pull-back Higgs bundle $(f^*\msE,f^*\varphi)$ is semistable. Applying this to $f_k:C_k\rightarrow X$ given in Remark \eqref{r.movable-curves} shows that $\deg_{C_k}(f_k^*\msW)\leq 0$ by the semi-stability of the Higgs bundle $(f_k^*\msE,f_k^*\varphi)$. In particular, letting $k\to \infty$ yields
        \[
        \tau_{\gamma}(\rho)\deg_{\gamma}(\msW) = \lim_{k} c_k \deg_{C_k}(\msW) \leq 0\leq 2\deg_{\gamma}(K_X).
        \]

        If $\alpha\not=0$, then $\alpha$ induces a non-zero holomorphic map $\msW\otimes \msV^{-1}\rightarrow \Omega_X^1$. Let $\msF$ be the saturation of the image. Then \cite[Theorem 2.7]{BoucksomDemaillyPuaunPeternell2013} implies that 
        \[
        c_1(K_X)-c_1(\msF) = c_1(\Omega_X^1/\msF)
        \] 
        is pseudoeffective as $X$ is non-uniruled. In particular, we have
        $\deg_{\gamma}(c_1(\msF)) \leq \deg_{\gamma}(c_1(K_X))$. On the other hand, the natural inclusion $\msW\otimes \msV^{-1}\rightarrow \msF$ implies
        \[
        2\tau_{\gamma}(\rho)=\deg_{\gamma}(\msW\otimes \msV^{-1})\leq \deg_{\gamma}(\msF) \leq \deg_{\gamma}(K_X).
        \]
        Finally the same argument applied to $\beta$ yields $-2\tau_{\gamma}(\rho)\leq \deg_{\gamma}(K_X)$ and we are done.
    \end{proof}
    
\subsection{Possion structure on the fiber}
In this subsection, we will introduce the Poisson structure on the fiber, as obtained in \cite[Section 10]{schottenloher1995metaplectic}. This Poisson structure was also explored in \cite{biswas1994aremark, biswas2021branes, schottenloher1995metaplectic}. Additionally, we will examine the conditions under which this Poisson structure gives rise to an algebraic integrable system. Our focus in this subsection will be on tracefree rank two Higgs bundles.

Let $\ssd_X:\MMD\to \MB_X$ be the restriction of the spectral morphism on the Dolbeault moduli space. Given fixed $(\msL,\alpha\in H^0(X,\msL^{-1}\Omega_X^1))$, we can utilize \eqref{e.Decomposition-BX-simplified} to define $\MB_{X,\msL}\subset \MB_X$. Now, we proceed to define $\MM_{\Dol,\msL}:=\ssd_X^{-1}(\MB_{X,\msL})\cap \MMD$, allowing us to establish a sub-spectral morphism as follows:
\begin{equation}
\begin{split}
\ssd_{X,\msL}:\MM_{\Dol,\msL}\to \MB_{X,\msL}\subset H^0(X,\Sym^2\Omega_X^1).
\end{split}
\end{equation}
	
We choose $\gamma_i\in \MB_{X,\msL}^{*}\subset H^0(X,S^2T_X^1)$ be a basis of the dual vector space, and define 
	\begin{equation}
		f_i:\MM_{\Dol,\msL}\to \mbC,\;f_i[A,\vp]:=\gamma_i(\Tr(\vp^2)).
	\end{equation}
	
	The K"ahler form $\omega$ induces isomorphisms between $T_X^{1,0}$ and $\Omega_X^{1,0}$, as well as between $T_X^{1,0}$ and $\Omega_X^{0,1}$. Consequently, utilizing the K"ahler metric, we can define a canonical pairing $\lan\;,\;\ran:\Sym^2T_X^{1,0}\ti \Sym^2\Omega_X^{1,0}\to \Omega_X^{1,1}$. This pairing is defined as follows: consider an orthonormal frame $dz_1,\cdots,dz_n$ and its dual frame $\partial_{z_1},\cdots,\partial_{z_n}$. Given $A=\sum_{i,j=1}^nA_{ij}\partial_{z_i}\otimes \partial_{z_j}\in \Sym^2T_X^{1,0}$ with $A_{ij}=A_{ji}$ and $B=\sum_{i,j=1}^nB^{ij}dz_i\otimes dz_j$ with $B_{ij}=B_{ji}$, the pairing is defined as $\langle A,B\rangle:=\frac{1}{2}\sum_{i,j,k=1}^nA_{ik}B_{kj}d\bar{z}_i\otimes dz_j$. Since $(\Omega_X^{1,0})^{*}=T_X^{1,0}$, there exists a canonical pairing $A(B)=\sum_{i,j=1}^nA_{ij}B_{ij}$. Furthermore, we have the following relation:

\begin{equation}
\label{eq_paring_Kahler}
\sqrt{-1}\Lambda \langle A,B\rangle=\sum_{i,j=1}^nA_{ij}B_{ij}=A(B).
\end{equation}
	
The Hermitian metric on $X$ induces a Hermitian metric on $\Sym^2\Omega_X^1$, which allows us to identify $\Sym^2\Omega_X^1$ with $\Sym^2T_X^1$. For any $\gamma\in H^0(X,\Sym^2\Omega_X^1)$, we can apply Serre duality to find $A\in \Gamma(\Sym^2T_X^1)$ such that for all $\kappa\in H^0(X,\Sym^2\Omega_X^1)$, the following equality holds:

\begin{equation}
\gamma(\kappa)=\int_XA(B)\vol=\int_X\sqrt{-1}\Lambda\langle A,B\rangle \vol,
\end{equation}

where $\langle\;,\;\rangle$ is the pairing defined above, and the last equality follows from \eqref{eq_paring_Kahler}. In summary, we can conclude the following lemma:

\begin{lemma}
\label{lemmma_dual_space_computation}
Let $\gamma\in \MB_X^{}\subset H^0(X,\Sym^2\Omega^1_X)^{}$. For any $\kappa\in H^0(X,\Sym^2\Omega^1_X)$, there exists $\xi\in \Sym^2\Omega^{1,0}_X$ such that $\gamma(\kappa)=\int_Xi\Lambda\langle \xi, \kappa\rangle \vol$, where $\langle\;,\;\rangle:\Sym^2T_X^{1,0}\times \Sym^2_X\Omega_X^{0,1}\to \Omega_X^{1,1}$ is the pairing induced by the K"ahler metric.
\end{lemma}
	
	Recall the holomorphic symplectic form would be $$\Omega_I:=(\omega_J+i\omega_K)((a_1,b_1),(a_2,b_2))=2i\int_X\Tr\Lam(b_2a_1-b_1a_2)\vol.$$ By Lemma \ref{lemmma_dual_space_computation}, we could write
	\begin{equation}
		f_i(A,\vp)=i\int_M\Lam \lan \xi_i\Tr(\vp^2)\ran \vol.
	\end{equation}
	
	Let $X_i$ be the Hamiltonian vector field corresponding to $f_i$, that is $\iota_{X_i}\Omega_I=df_i$. If we write $X_i=(a_i,b_i)$, for any $(a,b)$, we compute
	\begin{equation}
		\begin{split}
			\iota_{X_i}\Omega_I((a,b))=\Omega_I((a_i,b_i),(a,b))=2i\int_X\Tr\Lam(ba_i-b_ia)\vol
		\end{split}
	\end{equation}
	and 
	\begin{equation}
		\begin{split}
			df_i(a,b)=2i\int_M\Lam\lan \xi_i \Tr(\vp b) \ran \vol.
		\end{split}
	\end{equation}
	Therefore, we obtain $(a_i,b_i)=(\xi_i\vp,0)$. We compute 
	\begin{equation}
		\begin{split}
			\{f_i,f_j\}=-df_j(X_i)=2i\int_M\Lam\lan\xi_i \Tr(\vp b_i)\ran\vol=0.
		\end{split}
	\end{equation}
	
In summary, we can state the following proposition:

\begin{proposition}
The functions $f_i$ and $f_j$ Poisson commute with each other. In particular, over every fiber, there exist $\dim H^0(X,\msL^2)$ Poisson-commuting vector fields.
\end{proposition}
	
Suppose $\msL^2$ is base-point-free; then, according to Bertini's theorem, for a generic $\tau\in H^0(X,\msL^2)$, $X_{\tau}$ is smooth. On $X_{\tau}$, we can consider the generalization of the Prym variety defined as:

\begin{equation}
\Prym(X_{\tau}):=\{J\in\Pic(X_{\tau})|\msJ\otimes \sigma^{*}\msJ=\MO_{X_{\tau}}\}.
\end{equation}

By Proposition \ref{p.chern_class_push_down}, for $s=\alpha^2\tau$, we have $\MH^{-1}(s)\cong \Prym(X_{\tau})$. Consequently, we arrive at the following conclusion:

\begin{proposition}
If $H^0(X,\msL^2)$ is base-point-free and $\dim\Prym(X_{\tau})=\dim H^0(X,\msL^2)$, then the sub-spectral morphism $\ssd_{X,\msL}:\MM_{\Dol,\msL}\to \MB_{X,\msL}$ forms an algebraic integrable system.
\end{proposition}
	
\subsection{Higher rank generalizations}
The existence of the spectral base for $\SLC$ Higgs bundles also offers a method to construct higher rank canonical Higgs bundles, as described in \cite[Section 3]{hitchin1992lie}.
	
Given $s$, we can express $s=\tau\alpha^2$ with $\alpha\in H^0(X,\msL^{-1}\Omega_X^1)$, $\text{Div}(\alpha)=0$, and $\tau\in H^0(X,\msL^2)$. Additionally, we assume that $\msL$ admits a square root $\msL^{\frac12}$ and satisfies $c_1(\msL)\cup c_1(\msL)=0$. Under these conditions, we can explicitly construct the unit rank $2$ $\msL$-twisted Higgs bundle as follows:

\begin{equation}
\msE_2=\msL^{\frac12}\oplus \msL^{-\frac12}, \quad \vp_2=\begin{pmatrix}
0 & 0\\
1 & 0
\end{pmatrix}.
\end{equation}

The pair $(\msE_2,\alpha\circ\vp_2)$ forms a stable Higgs bundle with $c_1(\msE_2)=c_2(\msE_2)=0$. We then define $(\msE_n,\vp_n):=\Sym^{n-1}(\msE_2,\vp_2)$, which leads to:

\begin{equation}
\msE_n=\msL^{\frac{n-1}{2}}\oplus \msL^{\frac{n-3}{2}}\oplus \cdots \oplus \msL^{-\frac{n-1}{2}},
\end{equation} and with respect to this direct sum decomposition, we obtain:
	\begin{equation}
		\begin{split}
			\vp_n=\begin{pmatrix}
				0 & 0 & 0 & 0&\cdots &0\\
				1 & 0 & 0 & 0&\cdots &0\\
				0 & 1 & 0 & 0&\cdots & \vdots \\
				\vdots & \vdots& \vdots & \vdots & \ddots & \vdots\\
				\vdots & \vdots& \vdots & \vdots & \ddots & \vdots\\
				0 & 0 & 0 & \cdots & 1 & 0\\
			\end{pmatrix},
		\end{split}
	\end{equation}
where the "1" lies in the off-diagonal parts, representing the identity map $1:\msL^{\frac{n-1}{2}+i}\to \msL^{\frac{n-1}{2}+i-1}\otimes \msL$.
 
 Let $\rho_2:\pi_1(X)\to \SLC$ be the corresponding irreducible representation of $(\msE_2,\vp_2)$. Then $\rho_2$ induces a representation $\rho_n:\pi_1(X)\to \SL(n,\mbC)$, which is also irreducible, and the corresponding Higgs bundle would be $(\msE_n,\vp_n)$. Therefore, $(\msE_n,\vp_n)$ is stable.
	
	For the $\msL-$twisted Higgs bundle with trace 0, the Hitchin base for
	$\msL$-twisted Higgs bundle could be written as
	$$\MC_{X;\msL}:=\oplus_{i=2}^nH^0(X,\msL^i).$$
	
	Given $\mbfc=(c_2,\cdots, c_n)\in \MC_{X;\msL}$, we could define a Higgs bundle $(\msE_n,\vp_{n,c})$ analogous to the Hitchin section for curves \cite[Section 3]{hitchin1992lie} in the higher rank situation. The Higgs field $\vp_{n,\mathbf{c}}$ can be written as:
	\begin{equation}
		\begin{split}
			\vp_{n,\mbfc}=\begin{pmatrix}
				0 & c_2 & c_3 & c_4&\cdots &c_n\\
				1 & 0 & 0 & 0&\cdots &0\\
				0 & 1 & 0 & 0&\cdots & \vdots \\
				\vdots & \vdots& \vdots & \vdots & \ddots & \vdots\\
				\vdots & \vdots& \vdots & \vdots & \ddots & \vdots\\
				0 & 0 & 0 & \cdots & 1 & 0\\
			\end{pmatrix}.
		\end{split}
	\end{equation}

    Moreover, for the stability, we have
	\begin{proposition}{\cite[P.454]{hitchin1992lie}}
		For each $\mbfc$, the Higgs bundle $(\msE_n,\vp_{n,\mbfc})$ is a stable Higgs bundle.
	\end{proposition}
	\begin{proof}
		Let $g_t=\diag(1,t,t^2,\cdots, t^{n-1}),$ $\mbfc_t:=(t^2c_2,t^3c_3,\cdots, t^nc_n)$, then we have $g_t^{-1}\msE_n g_t=\msE_n$ and $g_t^{-1}\vp_{n,\mbfc}g_t=t^{-1}\vp_{n,\mbfc_t}.$
		As $(\msE_n,\vp_n)$ is stable and stable is an open condition, for $t$ sufficiently small, $(\msE_n,\vp_{n,\mbfc_t})$ is also stable, which implies $(\msE_n,\vp_{n,\mbfc})$ is stable.
	\end{proof}
	
	In summary, we conclude the following:
	\begin{theorem}
		Given $s=\al^2\tau\in \MB_X$ with $\al\in H^0(X,\msL^{-1}\Omega_X^1)$ and $\tau\in H^0(X,\msL^2)$, suppose $c_1(\msL).c_1(\msL).[\omega]^{n-2}=0$ and a square root $\msL^{\frac12}$ exists, then for any element $$\mbfc=(c_2,\cdots,c_n)\in \oplus_{i=2}^nH^0(X,\msL^i),$$ there exists a rank $n$ $\msL$-twsted stable Higgs bundle $(\msE_n,\vp_{n,\mbfc})$ with $c_1(\msE_n)=c_2(\msE_n)=0$ and $\MH(\msE_n,\vp_{n,\mbfc})=\mbfc$. Moreover, the corresponding representation of $(\msE_n,\al\circ\vp_{n,\mbfc})$ would be a real representation $\rho_{n,\mbfc}:\pi_1(X)\to \SL(n,\mbR)$. 
	\end{theorem}

	\bibliographystyle{alpha}
	\bibliography{references}

\begin{thebibliography}{BDDM22}

\bibitem[Ara02]{arapura2002higgs}
Donu Arapura.
\newblock Higgs bundles, integrability, and holomorphic forms.
\newblock {\em Motives, polylogarithms and Hodge theory, Part II (Irvine, CA,
  1998)}, 3:605--624, 2002.

\bibitem[BDDM22]{brotbek2022representations}
Damian Brotbek, Georgios Daskalopoulos, Ya~Deng, and Chikako Mese.
\newblock Representations of fundamental groups and logarithmic symmetric
  differential forms.
\newblock {\em arXiv preprint arXiv:2206.11835}, 2022.

\bibitem[BDO11]{bogomolov2011symmetric}
Fedor Bogomolov and Bruno De~Oliveira.
\newblock Symmetric differentials of rank 1 and holomorphic maps.
\newblock {\em Pure Appl. Math. Q.}, 7(4, Special Issue: In memory of Eckart
  Viehweg):1085--1103, 2011.

\bibitem[BDO13]{bogomolov2013closed}
Fedor Bogomolov and Bruno De~Oliveira.
\newblock Closed symmetric 2-differentials of the 1st kind.
\newblock {\em arXiv preprint arXiv:1310.0061}, 2013.

\bibitem[BDPP13]{BoucksomDemaillyPuaunPeternell2013}
S{\'e}bastien Boucksom, Jean-Pierre Demailly, Mihai P{\u a}un, and Thomas
  Peternell.
\newblock The pseudo-effective cone of a compact {K}\"ahler manifold and
  varieties of negative {K}odaira dimension.
\newblock {\em J. Algebraic Geom.}, 22(2):201--248, 2013.

\bibitem[Bea96]{Beauville1996}
Arnaud Beauville.
\newblock {\em Complex algebraic surfaces}, volume~34 of {\em London
  Mathematical Society Student Texts}.
\newblock Cambridge University Press, Cambridge, second edition, 1996.
\newblock Translated from the 1978 French original by R. Barlow, with
  assistance from N. I. Shepherd-Barron and M. Reid.

\bibitem[BH93]{BrunsHerzog1993}
Winfried Bruns and J\"{u}rgen Herzog.
\newblock {\em Cohen-{M}acaulay rings}, volume~39 of {\em Cambridge Studies in
  Advanced Mathematics}.
\newblock Cambridge University Press, Cambridge, 1993.

\bibitem[BHS21]{biswas2021branes}
Indranil Biswas, Sebastian Heller, and Laura~P Schaposnik.
\newblock Branes and moduli spaces of higgs bundles on smooth projective
  varieties.
\newblock {\em Research in the Mathematical Sciences}, 8(3):52, 2021.

\bibitem[BI07]{burger2007bounded}
Marc Burger and Alessandra Iozzi.
\newblock Bounded differential forms, generalized milnor--wood inequality and
  an application to deformation rigidity.
\newblock {\em Geometriae Dedicata}, 125(1):1--23, 2007.

\bibitem[Bis94]{biswas1994aremark}
I.~Biswas.
\newblock A remark on a deformation theory of green and lazarsfeld.
\newblock {\em Journal für die reine und angewandte Mathematik (Crelles
  Journal)}, 1994(449):103--124, 1994.

\bibitem[BKT13]{brunebarbe2013symmetric}
Yohan Brunebarbe, Bruno Klingler, and Burt Totaro.
\newblock Symmetric differentials and the fundamental group.
\newblock 2013.

\bibitem[BNR89]{BeauvilleNarasimhanRamanan1989}
Arnaud Beauville, M.~S. Narasimhan, and S.~Ramanan.
\newblock Spectral curves and the generalised theta divisor.
\newblock {\em J. Reine Angew. Math.}, 398:169--179, 1989.

\bibitem[Bro16]{Brotbek2016}
Damian Brotbek.
\newblock Symmetric differential forms on complete intersection varieties and
  applications.
\newblock {\em Math. Ann.}, 366(1-2):417--446, 2016.

\bibitem[Br{\"u}86]{Brueckmann1986}
Peter Br{\"u}ckmann.
\newblock Some birational invariants of algebraic varieties.
\newblock In {\em Proceedings of the conference on algebraic geometry
  ({B}erlin, 1985)}, volume~92 of {\em Teubner-Texte Math.}, pages 65--73.
  Teubner, Leipzig, 1986.

\bibitem[Bru15]{Brunella2015}
Marco Brunella.
\newblock {\em Birational geometry of foliations}, volume~1 of {\em IMPA
  Monographs}.
\newblock Springer, Cham, 2015.

\bibitem[CN20]{ChenNgo2020}
Tsao-Hsien Chen and Bao~Ch\^{a}u Ng\^{o}.
\newblock On the {H}itchin morphism for higher-dimensional varieties.
\newblock {\em Duke Math. J.}, 169(10):1971--2004, 2020.

\bibitem[Cor88]{corlette1988flat}
Kevin Corlette.
\newblock Flat {$G$}-bundles with canonical metrics.
\newblock {\em J. Differential Geom.}, 28(3):361--382, 1988.

\bibitem[CS08]{simpsoncorlette2008classification}
Kevin Corlette and Carlos Simpson.
\newblock On the classification of rank-two representations of quasiprojective
  fundamental groups.
\newblock {\em Compos. Math.}, 144(5):1271--1331, 2008.

\bibitem[Dem02]{Demailly2002}
Jean-Pierre Demailly.
\newblock On the {F}robenius integrability of certain holomorphic {$p$}-forms.
\newblock In {\em Complex geometry ({G}\"{o}ttingen, 2000)}, pages 93--98.
  Springer, Berlin, 2002.

\bibitem[Don87]{donaldson1987twisted}
S.~K. Donaldson.
\newblock Twisted harmonic maps and the self-duality equations.
\newblock {\em Proc. London Math. Soc. (3)}, 55(1):127--131, 1987.

\bibitem[Don95]{donagi1995spectral}
Ron Donagi.
\newblock Spectral covers.
\newblock {\em arXiv preprint alg-geom/9505009}, 1995.

\bibitem[Fri98]{Friedman1998}
Robert Friedman.
\newblock {\em Algebraic surfaces and holomorphic vector bundles}.
\newblock Universitext. Springer-Verlag, New York, 1998.

\bibitem[Fuj06]{fujiki2006hyperkahler}
Akira Fujiki.
\newblock Hyperk{\"a}hler structure on the moduli space of flat bundles.
\newblock In {\em Prospects in Complex Geometry: Proceedings of the 25th
  Taniguchi International Symposium held in Katata, and the Conference held in
  Kyoto, July 31--August 9, 1989}, pages 1--83. Springer, 2006.

\bibitem[GGPN21]{gallego2021higgs}
Guillermo Gallego, Oscar Garcia-Prada, and MS~Narasimhan.
\newblock Higgs bundles twisted by a vector bundle.
\newblock {\em arXiv preprint arXiv:2105.05543}, 2021.

\bibitem[Gol90]{goldman1990convex}
William~M Goldman.
\newblock Convex real projective structures on compact surfaces.
\newblock {\em Journal of Differential Geometry}, 31(3):791--845, 1990.

\bibitem[GRR15]{garcia2015introduction}
Alberto Garc{\'\i}a-Raboso and Steven Rayan.
\newblock Introduction to nonabelian hodge theory: flat connections, higgs
  bundles and complex variations of hodge structure.
\newblock {\em Calabi-Yau Varieties: Arithmetic, Geometry and Physics: Lecture
  Notes on Concentrated Graduate Courses}, pages 131--171, 2015.

\bibitem[GW12]{guichard2012anosov}
Olivier Guichard and Anna Wienhard.
\newblock Anosov representations: domains of discontinuity and applications.
\newblock {\em Inventiones mathematicae}, 190(2):357--438, 2012.

\bibitem[Har80]{Hartshorne1980}
Robin Hartshorne.
\newblock Stable reflexive sheaves.
\newblock {\em Math. Ann.}, 254(2):121--176, 1980.

\bibitem[He20]{he2020behavior}
Siqi He.
\newblock The behavior of sequences of solutions to the {H}itchin-{S}impson
  equations.
\newblock {\em arXiv preprint arXiv:2002.08109}, 2020.

\bibitem[Hit87a]{hitchin1987self}
N.~J. Hitchin.
\newblock The self-duality equations on a {R}iemann surface.
\newblock {\em Proc. London Math. Soc. (3)}, 55(1):59--126, 1987.

\bibitem[Hit87b]{hitchin1987stable}
Nigel Hitchin.
\newblock Stable bundles and integrable systems.
\newblock {\em Duke Math. J.}, 54(1):91--114, 1987.

\bibitem[Hit92]{hitchin1992lie}
N.~J. Hitchin.
\newblock Lie groups and {T}eichm\"{u}ller space.
\newblock {\em Topology}, 31(3):449--473, 1992.

\bibitem[HKLR87]{hitchin1987hyperkahler}
Nigel~J Hitchin, Anders Karlhede, Ulf Lindstr{\"o}m, and Martin Ro{\v{c}}ek.
\newblock Hyperk{\"a}hler metrics and supersymmetry.
\newblock {\em Communications in Mathematical Physics}, 108(4):535--589, 1987.

\bibitem[HLM23]{heliumok2023rigidity}
Siqi He, Jie Liu, and Ngaiming Mok.
\newblock Rigidity and integrality of the character variety of arithematic
  variety with rank {$\geq 2$}.
\newblock {\em to appear}, 2023.

\bibitem[Kli03]{Klingler2003fakeprojective}
Bruno Klingler.
\newblock Sur la rigidit\'{e} de certains groupes fondamentaux,
  l'arithm\'{e}ticit\'{e} des r\'{e}seaux hyperboliques complexes, et les
  ``faux plans projectifs''.
\newblock {\em Invent. Math.}, 153(1):105--143, 2003.

\bibitem[Kli13]{klingler2013symmetric}
Bruno Klingler.
\newblock Symmetric differentials, k{\"a}hler groups and ball quotients.
\newblock {\em Inventiones mathematicae}, 192:257--286, 2013.

\bibitem[KM10]{KoziarzMaubon2010}
Vincent Koziarz and Julien Maubon.
\newblock The {T}oledo invariant on smooth varieties of general type.
\newblock {\em J. Reine Angew. Math.}, 649:207--230, 2010.

\bibitem[Kob80]{Kobayashi1980a}
Shoshichi Kobayashi.
\newblock First {C}hern class and holomorphic tensor fields.
\newblock {\em Nagoya Math. J.}, 77:5--11, 1980.

\bibitem[Kol96]{Kollar1996}
J{\'a}nos Koll{\'a}r.
\newblock {\em Rational curves on algebraic varieties}, volume~32 of {\em
  Ergeb. Math. Grenzgeb., 3. Folge}.
\newblock Berlin: Springer-Verlag, 1996.

\bibitem[LS18]{simpsonadrian2018rank}
Adrian Langer and Carlos Simpson.
\newblock Rank 3 rigid representations of projective fundamental groups.
\newblock {\em Compos. Math.}, 154(7):1534--1570, 2018.

\bibitem[Mar91]{margulis1991discrete}
G.~A. Margulis.
\newblock {\em Discrete subgroups of semisimple {L}ie groups}, volume~17 of
  {\em Ergebnisse der Mathematik und ihrer Grenzgebiete (3) [Results in
  Mathematics and Related Areas (3)]}.
\newblock Springer-Verlag, Berlin, 1991.

\bibitem[Mau15]{Maubon2015}
Julien Maubon.
\newblock Higgs bundles and representations of complex hyperbolic lattices.
\newblock In {\em Handbook of group actions. {V}ol. {II}}, volume~32 of {\em
  Adv. Lect. Math. (ALM)}, pages 201--244. Int. Press, Somerville, MA, 2015.

\bibitem[Rei77]{Reid1977}
Miles Reid.
\newblock Bogomolov's theorem $c^2_1\leq 4c_2$.
\newblock In {\em Proceedings of the International Symposium on Algebraic
  Geometry (Kyoto Univ., Kyoto, 1977)}, pages 623--642, 1977.

\bibitem[Ser65]{Serre1965}
Jean-Pierre Serre.
\newblock {\em Alg{\`e}bre locale. {M}ultiplicit{\'e}s}, volume~11 of {\em
  Lecture Notes in Mathematics}.
\newblock Springer-Verlag, Berlin-New York, 1965.
\newblock Cours au Coll{\`e}ge de France, 1957--1958, r{\'e}dig{\'e} par Pierre
  Gabriel, Seconde {\'e}dition, 1965.

\bibitem[Sha78]{Shavel1978}
Ira~H. Shavel.
\newblock A class of algebraic surfaces of general type constructed from
  quaternion algebras.
\newblock {\em Pacific J. Math.}, 76(1):221--245, 1978.

\bibitem[Sim88]{Simpson1988Construction}
Carlos~T. Simpson.
\newblock Constructing variations of {H}odge structure using {Y}ang-{M}ills
  theory and applications to uniformization.
\newblock {\em J. Amer. Math. Soc.}, 1(4):867--918, 1988.

\bibitem[Sim91]{Simpson1991}
Carlos~T. Simpson.
\newblock The ubiquity of variations of {H}odge structure.
\newblock In {\em Complex geometry and {L}ie theory ({S}undance, {UT}, 1989)},
  volume~53 of {\em Proc. Sympos. Pure Math.}, pages 329--348. Amer. Math.
  Soc., Providence, RI, 1991.

\bibitem[Sim92]{Simpson1992}
Carlos~T. Simpson.
\newblock Higgs bundles and local systems.
\newblock {\em Inst. Hautes \'{E}tudes Sci. Publ. Math.}, (75):5--95, 1992.

\bibitem[Sim94a]{simpson1994moduli}
Carlos~T. Simpson.
\newblock Moduli of representations of the fundamental group of a smooth
  projective variety. {I}.
\newblock {\em Inst. Hautes \'{E}tudes Sci. Publ. Math.}, (79):47--129, 1994.

\bibitem[Sim94b]{simpson1994moduli2}
Carlos~T. Simpson.
\newblock Moduli of representations of the fundamental group of a smooth
  projective variety. {II}.
\newblock {\em Inst. Hautes \'{E}tudes Sci. Publ. Math.}, (80):5--79 (1995),
  1994.

\bibitem[Sim97]{simpson1996hodge}
Carlos Simpson.
\newblock The {H}odge filtration on nonabelian cohomology.
\newblock In {\em Algebraic geometry---{S}anta {C}ruz 1995}, volume~62 of {\em
  Proc. Sympos. Pure Math.}, pages 217--281. Amer. Math. Soc., Providence, RI,
  1997.

\bibitem[SS95]{schottenloher1995metaplectic}
M~Schottenloher and P~Scheinost.
\newblock Metaplectic quantization of the moduli spaces of flat and parabolic
  bundles.
\newblock 1995.

\bibitem[SS21]{SongSun2021}
Lei Song and Hao Sun.
\newblock On the image of {H}itchin morphism for algebraic surfaces: The case
  $\textrm{GL}_n$.
\newblock {\em Int. Math. Res. Not. IMRN}, to appear, 2021.

\bibitem[Tou16]{Touzet2016}
Fr\'{e}d\'{e}ric Touzet.
\newblock On the structure of codimension $1$ foliations with pseudoeffective
  conormal bundle.
\newblock In {\em Foliation theory in algebraic geometry}, Simons Symp., pages
  157--216. Springer, Cham, 2016.

\bibitem[Wen16]{Wentworth2016}
Richard~A. Wentworth.
\newblock Higgs bundles and local systems on {R}iemann surfaces.
\newblock In {\em Geometry and quantization of moduli spaces}, Adv. Courses
  Math. CRM Barcelona, pages 165--219. Birkh\"{a}user/Springer, Cham, 2016.

\bibitem[Wie18]{wienhard2018invitation}
Anna Wienhard.
\newblock An invitation to higher {T}eichm\"{u}ller theory.
\newblock In {\em Proceedings of the {I}nternational {C}ongress of
  {M}athematicians---{R}io de {J}aneiro 2018. {V}ol. {II}. {I}nvited lectures},
  pages 1013--1039. World Sci. Publ., Hackensack, NJ, 2018.

\bibitem[Yeu04]{Yeung2004integralityandarithmeticity}
Sai-Kee Yeung.
\newblock Integrality and arithmeticity of co-compact lattice corresponding to
  certain complex two-ball quotients of {P}icard number one.
\newblock {\em Asian J. Math.}, 8(1):107--129, 2004.

\bibitem[Zuo00]{zuo2000negativity}
Kang Zuo.
\newblock On the negativity of kernels of kodaira--spencer maps on hodge
  bundles and applications.
\newblock {\em Asian Journal of Mathematics}, 4(1):279--302, 2000.

\end{thebibliography}
\end{document}